\documentclass[sn-basic]{sn-jnl}

\jyear{2021}%

\usepackage{graphicx}
\usepackage{dcolumn}
\usepackage[utf8]{inputenc}
\usepackage[T1]{fontenc}
\usepackage{etoolbox}
\usepackage{natbib}
\usepackage{dirtytalk}
\usepackage{amssymb} 
\usepackage{amsmath}        
\usepackage{amsfonts}       
\usepackage{amsthm}         
\usepackage{bbding}         
\usepackage{bm}             
\usepackage{graphicx}       
\usepackage{fancyvrb}       
\usepackage{dcolumn}        
\usepackage{booktabs}       
\usepackage{paralist}       
\usepackage{multirow}
\usepackage{mathtools}
 \usepackage{bbm}
 \usepackage[table,xcdraw]{xcolor}
 \usepackage{tabularray}
 \usepackage[normalem]{ulem}
 \useunder{\uline}{\ul}{}
\makeatletter
\def\@email#1#2{%
 \endgroup
 \patchcmd{\titleblock@produce}
  {\frontmatter@RRAPformat}
  {\frontmatter@RRAPformat{\produce@RRAP{*#1\href{mailto:#2}{#2}}}\frontmatter@RRAPformat}
  {}{}
}%
\makeatother

\newtheorem{theorem}{Theorem}
\newtheorem{consequence}{Consequence}
\newtheorem{innercustomthm}{Theorem}
\newenvironment{customthm}[1]
  {\renewcommand\theinnercustomthm{#1}\innercustomthm}
  {\endinnercustomthm}
\newtheorem{lemma}{Lemma}
\newtheorem{innercustomlem}{Lemma}
\newenvironment{customlem}[1]
  {\renewcommand\theinnercustomlem{#1}\innercustomlem}
  {\endinnercustomlem}
\newtheorem{innercustomconsequence}{Consequence}
  \newenvironment{customconsequence}[1]
  {\renewcommand\theinnercustomconsequence{#1}\innercustomconsequence}
  {\endinnercustomconsequence}
  
\newtheorem{proposition}{Proposition}
\newtheorem{claim}{Claim}
\newtheorem{definition}{Definition}

\newtheorem*{Results}{Conclusion}

\newtheorem{model}{Model}

\theoremstyle{remark}

\newtheorem*{remark}{Remark}
\newtheorem{example}{Example}

\DefineVerbatimEnvironment{code}{Verbatim}{fontsize=\small, frame=single}

\newcommand{\N}{\mathbb{N}}

\DeclareMathOperator{\E}{\mathbb{E}\,}
\newcommand{\pr}{\mathbb{P}}

\DeclareMathOperator{\deter}{det}
\DeclareMathOperator{\var}{var}
\DeclareMathOperator{\cov}{cov}

\DeclarePairedDelimiter\floor{\lfloor}{\rfloor}

\newcommand{\RV}{{\rm RV}}
\newcommand{\VAR}{{\rm VAR}}
\newcommand{\NAAR}{{\rm NAAR}}
\newcommand{\hVAR}{{\rm hVAR}}
\newcommand{\hNAAR}{{\rm hNAAR}}

\newcommand{\diff}{{\rm d}}
\newcommand{\1}{\mathbbm{1}}

\begin{document}

\title[]{Causality in extremes of time series}


\author*[1,2]{\fnm{Juraj} \sur{Bodik}}\email{Juraj.Bodik@unil.ch}

\author[1]{\fnm{Zbyněk} \sur{Pawlas}}\email{pawlas@karlin.mff.cuni.cz}

\author[2]{\fnm{Milan} \sur{Paluš}}\email{mp@cs.cas.cz}

\affil[1]{\orgdiv{Department of Probability and Mathematical Statistics}, \orgname{Charles University}, \orgaddress{\city{Prague}, \country{Czech Republic}}}

\affil[2]{\orgdiv{Institute of Computer Science}, \orgname{The Czech Academy of Sciences}, \orgaddress{\city{Prague}, \country{Czech Republic}}}


\abstract{Consider two stationary time series with heavy-tailed marginal distributions. We aim to detect whether they have a causal relation, that is, if a change in one causes a change in the other. Usual methods for causal discovery are not well suited if the causal mechanisms only appear during extreme events. We propose a framework to detect a causal structure from the extremes of time series, providing a new tool to extract causal information from extreme events. We introduce the causal tail coefficient for time series, which can identify asymmetrical causal relations between extreme events under certain assumptions. This method can handle nonlinear relations and latent variables. Moreover, we mention how our method can help estimate a typical time difference between extreme events. Our methodology is especially well suited for large sample sizes, and we show the performance on the simulations. Finally, we apply our method to real-world space-weather and hydro-meteorological datasets.}


\keywords{{Granger causality}, {causal inference}, {nonlinear time series}, {causality-in-tail}, {extreme value theory}, {heavy tails}}



\maketitle


\section{Introduction}
\label{Section_introduction}
The ultimate goal of causal inference is understanding relationships between random variables and predicting the outcomes resulting from their modification or manipulation \citep{Elements_of_Causal_Inference}. 
Causal inference finds utility across a wide array of scientific domains. For instance, in the field of medicine, it aids in comprehending the propagation of epileptic seizures across distinct brain regions \citep{Rubin}. Or in climate science, it facilitates the prediction of variables such as temperature and rainfall \citep{Naveau} while unraveling the causes behind sudden changes in river discharges \citep{Linda}. Extensive efforts have been invested in establishing its mathematical foundation \citep{Pearl}. Structural causal models (SCMs) serve as mathematical tools for representing causal relationships between variables, especially in scenarios where temporal order is unavailable. Recent advancements include estimating causal mechanisms within structural causal models \citep{Peters2014,reviewANMMooij, ZhangReview}. 
However, causal discovery in time series often necessitates alternative models  \citep[Chapter 10]{Elements_of_Causal_Inference} and faces different problems \citep{RungeReview}. 

A commonly used concept for describing a causality in time series is Granger causality \citep{GrangerOriginal, GRANGER1980329}.   Though various definitions of causality, such as Sims causality \citep{sims}, structural causality \citep{White2010}, and interventional causality \citep{Eichler}, coexist, Granger causality holds a prominent position in the literature \cite[Chapter 22.2]{berzuini2012causality}. It rests upon two key principles:  precedence of cause over effect and the unique information contained in the cause that is otherwise unattainable. In this paper, we primarily consider strong Granger causality, which employs conditional independence (see Section \ref{Preliminaries}). Variations include linear Granger causality \citep{Hosaya} and Granger causality in mean \citep{ho2015granger}. In the autoregressive models considered in this paper, all previously mentioned notions are closely related, and their differences are typically not of practical relevance \citep[Chapter 22.8]{berzuini2012causality}. 

Several approaches exist for learning a causal structure in time series. State-of-the-art methods for causal discovery in time series use consecutive conditional independence testing  \citep{Zhang2008} or fitting the vector autoregressive (VAR) models \citep{Eicher}. \cite{Runge} introduce a PCMCI method that combines a state-of-the-art PC method (named after the inventors Peter Spirtes and Clark Glymour;  \cite{PCalgorithm}) and MCI (momentary conditional independence; \cite{Momentary}) to identify and adjust for many potential confounders. However, there is still a need for more robust methods against a variety of hidden confounders. \cite{Gerhardus} introduce the LPCMCI algorithm that is adapted to handle latent confounders. Another approach leverages Shannon's information theory, utilizing entropy and mutual information to probe causal aspects of dynamic and complex systems \citep{HLAVACKOVASCHINDLERPALUS}.

Typically, causal inference methods in time series describe the causality in the body of the distribution (causality in the mean). Several articles deal with the second-order causality (causality in variance; \cite{CausalityInMean}) using, e.g., GARCH (Generalized Autoregressive Conditional Heteroskedastic) modeling \citep{CausalityInVariance}. However, causality in extremes is a new field of research.  Different perspectives can be seen when looking mainly at the tails of the distributions \citep{Coles}. Many causal mechanisms are present only during extreme events, and interventions often carry information that is likely to be causal \citep{Cox}. While the connection between extreme value theory and time series has been thoroughly studied \citep{heavy_tailed_nonlinear_time_series,Mikosch+Wintenberger,Heavy_tailed_time_series}, the area of causal inference at this juncture remains unexplored.

Recent work intertwines extreme value theory and causality in SCM. \cite{EngelkeGraphicalModels} propose graphical models in the context of extremes.  \cite{10.1214/20-AOAS1355} study probabilities of necessary and sufficient causation as defined in the counterfactual theory using the multivariate generalized Pareto distributions. \cite{Deuber} develop a method for estimating extremal quantiles of treatment effects. Further approaches involve recursive max-linear models on directed acyclic graphs \citep{gissibl2017maxlinear, wdesfrgth}. 

Our paper aims to establish a framework of causality in extremes of time series. This work builds on the work of \cite{SvajciariGneccoClanok}, who first introduced the concept of causal tail coefficient in the context of SCM, followed by \cite{Pasche}, who extended this coefficient by incorporating possible covariates into the model. We aim to move the theory of causality in extremes from SCMs into a context of Granger-type causality in time series. 

The paper is organized as follows.  The subsequent section provides a concise overview of the causal tail coefficient's existing developments, presents a motivating time series example, and introduces preliminary results and notation.
Section~\ref{chapter 2} contains the main results and provides an illustrative model example. 
Section~\ref{chapter 3} extends the proposed method and discusses outcomes under assumptions violations. The section also addresses scenarios involving hidden confounders and a choice of the minimal time delay. 
Section~\ref{chapter 4} addresses the estimation problem, examining estimator properties and employing simulations on synthetic datasets. 
In Section~\ref{APPLICATIONS}, we apply the method to two real-world datasets. First, we consider an application concerning space weather and geomagnetic storms, corroborating prior findings that employ conditional mutual information. Lastly, we employ our methodology on a hydrometeorological dataset, investigating six distinct weather and climate phenomena. To maintain brevity, proofs are relocated to the \hyperref[chapter 5]{Appendix}. Appendix \ref{chapter 5} introduces auxiliary propositions used in the proofs, while the proofs themselves can be found in Appendix \ref{AppendixB}.

\subsection{SCM and existing work on the causal tail coefficient}\label{s1.2}

 \cite{SvajciariGneccoClanok} established the groundwork for the causal tail coefficient within linear SCMs \citep{Pearl}. A linear SCM over real random variables $X_1,\dots,X_p$ is a collection of $p$ assignments 
\begin{equation}\label{equation111}
    X_j = \sum_{k\in pa(j)}\beta_{jk}X_k +\varepsilon_j, \quad j\in V,
\end{equation}
where $pa(j)\subseteq V=\{1, \dots, p\}$  denotes parents of $j$, $\varepsilon_1, \dots, \varepsilon_p$ are jointly independent random variables and $\beta_{j,k}\in\mathbb{R}\setminus\{0\}$ are the causal weights. We suppose that the associated graph $\mathcal{G}=(V,E)$, with the directed edge $(i,j)$ present if and only if $i\in pa(j)$, is a directed acyclic graph (DAG) with nodes $V$ and edges $E$. We say that $X_i$ causes $X_j$ if a directed path exists from $i$ to $j$ in $\mathcal{G}$. Structural causal models describe not only observational distributions but also distributions under interventions (or manipulations) on the variables. By intervening on $X_j$, we understand a new SCM with $p$ assignments, all of them identical with (\ref{equation111}) except the assignment for $X_j$ \citep{Pearl}.     

Let $X_i, X_j$ be a pair of random variables from linear SCM. We assume that $X_i, X_j$ are heavy-tailed with respective distributions $F_i, F_j$.  The causal (upper) tail coefficient of  $X_i$ on $X_j$ is defined in \cite{SvajciariGneccoClanok} as
$$
\Gamma_{i,j} := \lim_{u\to 1^-}\mathbb{E}[ F_j(X_j)\mid F_i(X_i)>u ], 
$$
if the limit exists. This coefficient lies between zero and one and captures the causal influence of $X_i$ on $X_j$ in the upper tail since, intuitively, if $X_i$ has a monotonically increasing causal influence on $X_j$, we expect $\Gamma_{i,j}$ to be close to unity. The coefficient is asymmetric, as extremes of $X_j$ need not lead to extremes of $X_i$, and in that case, $\Gamma_{j,i}$ will be appreciably smaller than $\Gamma_{i,j}$. In Section \ref{ESXDRCFTGV}, we discuss the intuition of the choice of this coefficient in more detail (in the context of time series).

Under certain assumptions on the tails of $\varepsilon_i$ and the linear SCM, the values of $\Gamma_{i,j}$ and $\Gamma_{j,i}$ allow us to discover the causal relationship between $X_i$ and $X_j$. These relations are summarized in Table \ref{TableGnecco}.
\begin{table}
\centering
\begin{tblr}{
  vline{-} = {1}{},
  vline{1-2,5} = {2-4}{},
  hline{1-2,5} = {-}{},
  hline{3-4} = {1}{},
}
                                     & $\Gamma_{j,i}=1$ & $\Gamma_{j,i}\in (0.5, 1)$ & $\Gamma_{j,i}=0.5$ \\
$\Gamma_{i,j}=1$ &                                      & $X_i$ causes $X_j$                                      &                                        \\
$\Gamma_{i,j}\in (0.5, 1)$ & $X_j$ causes $X_i$             & Common cause only                                             &                                        \\
$\Gamma_{i,j}=0.5$ &                                      &                                                               & No causal link                         
\end{tblr}
\caption{Causal relationship between a pair of random variables $X_i, X_j$ following appropriately restricted linear SCM under different values of the causal tail coefficients. Blank entries and cases when $\Gamma_{i,j} < 0.5$ or $\Gamma_{j,i}<0.5$ cannot occur \citep[Theorem 1]{SvajciariGneccoClanok}.}
\label{TableGnecco}
\end{table}

\cite{SvajciariGneccoClanok} proposed a consistent non-parametric estimator of $\Gamma_{i,j}$. Without loss of generality, consider $i = 1$ and $j = 2$. If $(x_{1,1},x_{1,2}), \dots,(x_{n,1},x_{n,2})$ are $n\in\mathbb{N}$ independent replicates of $(X_1, X_2)$, the estimator of $\Gamma_{1,2}$ is defined as 
\begin{equation}
\label{2fdsa}
\hat{\Gamma}_{1,2}:=\frac{1}{k}\sum_{i=1}^n\hat{F}_2(x_{i,2})\1(x_{i,1}>x_{(n-k+1),1}),    
\end{equation}
 where $x_{(n-k+1),1}$ is the $k$-th largest value of $x_{1,1}, \dots, x_{n,1}$, $\hat{F}_2(t)=\frac{1}{n}\sum_{i=1}^n \1(x_{i,2}\leq t)$, and $\1(\cdot)$ is the indicator function.

\cite{Pasche} adapted the estimation process to account for potential confounding factors. The authors modified \eqref{2fdsa} by replacing $\hat{F}_2(x)$ by a covariate-dependent estimator, where the upper tail of $\hat{F}_2(x)$ is modeled by a Pareto approximation \citep{Coles}. 
implemented a permutation test to formally assess the hypothesis that the causal tail coefficient is equal to $1$.

\subsection{Motivating example and the main idea in the context of time series}
\label{ESXDRCFTGV}
The following example illustrates a typical case considered in this paper. Let $(\mathbf{X},\mathbf{Y})^\top= ((X_t,Y_t)^\top, t\in\mathbb{Z})$ be a bivariate strictly stationary\footnote{A stochastic process is strictly stationary if the joint distributions of $n$ consecutive variables are time-invariant (e.g., Section 2.1.3 in \cite{MVT_timeseries_SPRINGER}). We will not work with other stationarity types.} time series described by the following recurrent relations
\begin{align*}
X_t&=\frac{1}{2}X_{t-1}+\varepsilon_t^X,\\
Y_t&=\frac{1}{2}Y_{t-1}+ \sqrt{X_{t-5}} + \varepsilon_t^Y,
\end{align*}
where $\varepsilon_t^X, \varepsilon_t^Y\overset{iid}{\sim}$ Pareto$(1,1)$ \footnote{$\varepsilon_t^X, \varepsilon_t^Y$ are iid (independent and identically distributed), 
following a Pareto distribution with parameters equal to $1$. 
The distribution function of a  Pareto$(a,b)$ random variable is in the form $F(x)=1-(\frac{a}{x})^b$ for $x\geq a$, 
zero otherwise. When $a=b=1$, it is often called the standard Pareto 
distribution.
}. 
\begin{figure}[t]
\centering
\includegraphics[scale=0.65]{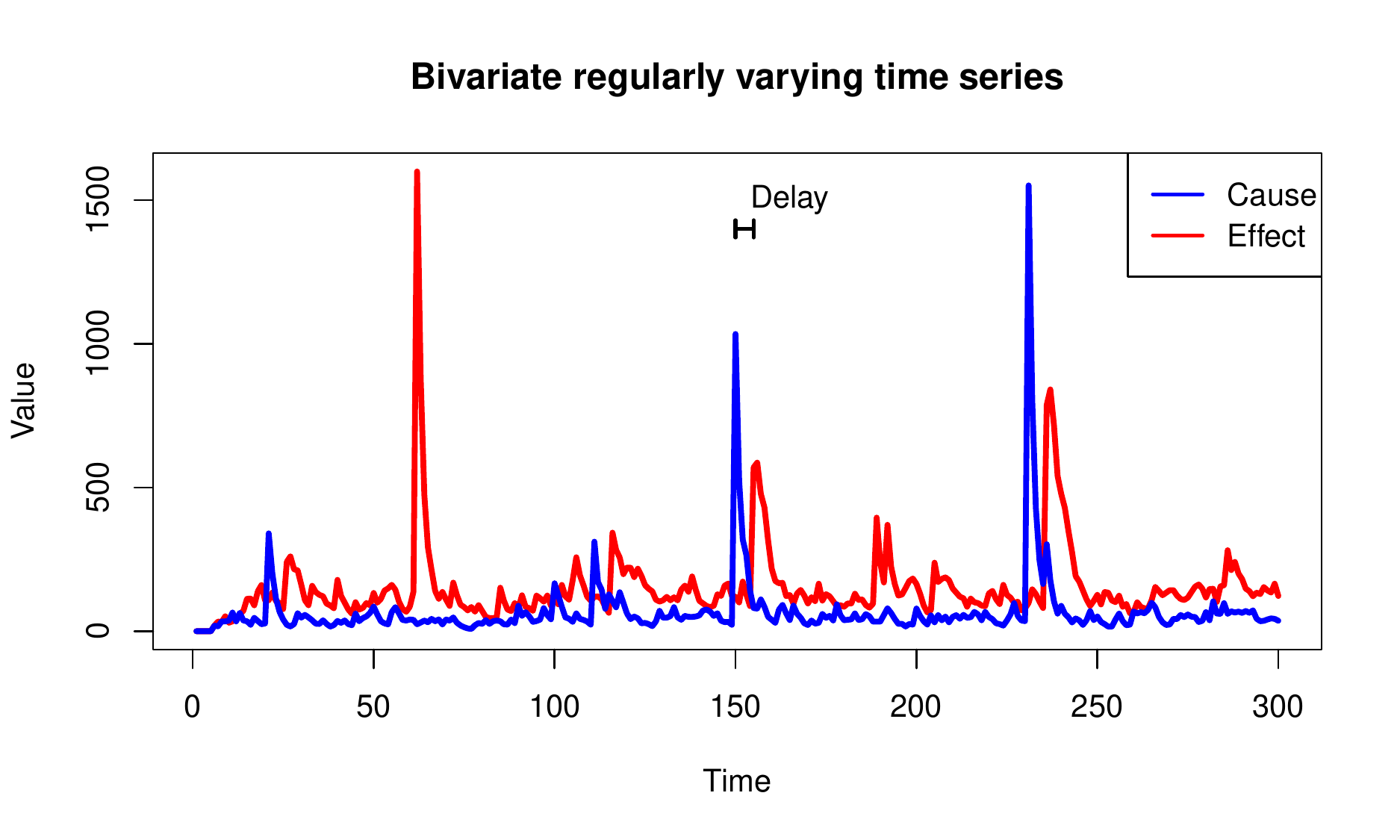}
\caption{The figure represents a sample realisation of $(\mathbf{X},\mathbf{Y})^\top$ from Subsection~\ref{ESXDRCFTGV} ($\mathbf{X}$ is the cause and $\mathbf{Y}$ is the effect). The delay represents the time delay between the time series, in this case, equal to $5$. }
\label{Pekne_Pareto_grafy}
\end{figure}
This scenario is depicted in Figure~\ref{Pekne_Pareto_grafy}. 
Here, $\mathbf{X}$ causes $\mathbf{Y}$ (in Granger sense, see Definition \ref{AFVAR} given later), simply because the knowledge of $\mathbf{X}$ improves the prediction of $\mathbf{Y}$. However, the converse is not true.

Consider the data shown in Figure~\ref{Pekne_Pareto_grafy}. Our goal is to identify any causal relationship between these time series.  There is (at least in this realization) an evident asymmetry between the two time series in the extremes. If the cause ($\mathbf{X}$) is extremely large, then the effect ($\mathbf{Y}$) \textit{will} be also extremely large (see the second and third \say{jump}). However, if $\mathbf{Y}$ is extremely large, then $\mathbf{X}$ will not necessarily be extremely large (as evident in the first \say{jump}). This indicates that an extreme value of $\mathbf{X}$ tends to trigger an extreme value of $\mathbf{Y}$, implying a causal link in an intuitive sense. Yet, a crucial factor is the presence of a \textit{time} \textit{delay} (or \textit{time} \textit{lag}). The extremes don't have to be concurrent -- there's a time lag during which the influence of $\mathbf{X}$ on $\mathbf{Y}$ becomes apparent. In this specific example, this lag is exactly $5$ time units.

To encapsulate this idea mathematically, we introduce the \textit{causal tail coefficient for time series}
\begin{equation}\label{2.1}
\Gamma^{time}_{\mathbf{X}\to \mathbf{Y}}(p):=\lim_{u\to 1^-}\E[\max\{F_Y(Y_0), \dots, F_Y(Y_{p})\}\mid F_X(X_0)>u],
\end{equation}
where $F_X, F_Y$ are the marginal distribution functions of $\mathbf{X}, \mathbf{Y}$, respectively. This coefficient mathematically expresses natural questions: Does an extreme value in $\mathbf{X}$ invariably lead to an extreme value in $\mathbf{Y}$? How large $\mathbf{Y}$ will be in the next $p$ steps if $\mathbf{X}$ is extremely large (in their respective scales)? In our example, we can consider $p=5$. If $X_0$ is extremely large, then $Y_5$ will surely also be extremely large (large cause implies large effect), but not the other way around (large effect does not imply large cause). Hence intuitively, the following should hold: $\Gamma^{time}_{\mathbf{X}\to \mathbf{Y}}(p)=1$, but $\Gamma^{time}_{\mathbf{Y}\to \mathbf{X}}(p)<1$. The main part of the paper consists of determining the assumptions under which this is true.

\subsection{Preliminaries and notation}
\label{Preliminaries}
We use bold capital letters to represent notation for time series and random vectors. A time series, or a random process, consists of a set of random variables $\mathbf{Z} = (Z_t, t \in \mathbb{Z})$ defined on the same probability space. We exclusively work with time series defined over integers.

We will use the standard notion of regular variation \citep{Reg_Var_Resnick_kniha, RegularVariationBook}.  A real random variable $X$ is regularly varying with tail index $\theta>0$, if its distribution function has a form $F_X(x)=1-x^{-\theta} L(x)$ for some slowly varying function $L$, i.e., a function satisfying $\lim_{x\to\infty}\frac{L(\alpha x)}{L(x)} = 1$ for every $\alpha > 0$  \citep[Section 1.3]{Heavy_tailed_time_series}. This property is denoted by $X\sim \RV(\theta)$. Examples of regularly varying distributions include Pareto, Cauchy or Fréchet distributions, to name a few \citep{Heavy_tail_article}. For real functions $f,g$, we denote $f(x)\sim g(x) \iff \lim_{x\to\infty}\frac{f(x)}{g(x)}=1$. 

The main principle that we aim to use is the so-called max-sum equivalence, that is, for two random variables $X,Y$ we have $\pr(X+Y>x)\sim \pr(X>x)+\pr(Y>x)\sim \pr(\max(X,Y)>x)$ as $x\to\infty$. This is satisfied when $X,Y\overset{iid}{\sim} \RV(\theta)$ \citep[ Section 1.3.1]{RegularVariationBook}. Similar results hold even if we deal with finite or infinite sums of random variables \citep[Section 4.5]{Reg_Var_Resnick_kniha}. 

Consider a bivariate process  $(\mathbf{X},\mathbf{Y})^\top =( (X_t, Y_t)^\top, t\in\mathbb{Z})$. The concept of strong Granger causality \citep[Chapter 22.2]{berzuini2012causality} can be expressed as follows: The process $\mathbf{X}$ is said to cause $\mathbf{Y}$ (denoted as $\mathbf{X}\to \mathbf{Y}$) if $Y_{t+1}$ is not independent with the past of $\mathbf{X}$ given  all relevant variables in the universe up to time $t$ except the past values of $\mathbf{X}$; that is, 
$$Y_{t+1} \not\!\perp\!\!\!\perp \mathbf{X}_{past(t)}\mid \mathcal{C}_{t}\setminus \mathbf{X}_{past(t)},$$
where $past(t) = (t, t-1, t-2, \dots)$ and $\mathcal{C}_t$ represents all relevant variables in the universe up to time $t$. The philosophical notion of $\mathcal{C}_{t}$ is typically replaced by only a finite set of relevant variables. 
To illustrate with an example, given a three-dimensional process  $(\mathbf{X},\mathbf{Y},\mathbf{Z})^\top=( (X_t, Y_t, Z_t)^\top, t\in\mathbb{Z})$, we replace the information set $\mathcal{C}_t$  by $(\mathbf{X}_{past(t)}, \mathbf{Y}_{past(t)}, \mathbf{Z}_{past(t)})^\top$  and say that the process $\mathbf{X}$ causes $\mathbf{Y}$ with respect to $(\mathbf{X},\mathbf{Y},\mathbf{Z})^\top$ if $Y_{t+1} \not\!\perp\!\!\!\perp \mathbf{X}_{past(t)}\mid  (\mathbf{Y}_{past(t)}, \mathbf{Z}_{past(t)})$.  We have to note that such $\mathbf{X}$ has to be seen only as a potential cause, since enlarging the information set can lead to a change in the causal structure. 

In certain models, the definition of causality can be  simplified  \citep{Web_stranka_causality_in_time_series}.  For instance, \cite{sims} demonstrated that the Granger causality definition is equivalent to certain parameter restrictions within a linear framework. We provide the formal definition in Subsection \ref{SectionModels} for a specific class of autoregressive models.

\section{The causal tail coefficient for time series}
\label{chapter 2}

The central concept introduced in this paper is the \textit{causal tail coefficient for time series} $\Gamma^{time}_{\mathbf{X}\to \mathbf{Y}}(p)$ \textit{with extremal delay} $p\in\mathbb{N}$ defined in (\ref{2.1}). In scenarios where we exclude instantaneous effects (such as when $X_0$ directly causes $Y_0$), we can directly employ the following coefficient:
\begin{equation*}\label{Gamma(p,-0)}
\begin{split}
    &\Gamma^{time, -0}_{\mathbf{X}\to \mathbf{Y}}(p):=\lim_{u\to 1^-}\E[\max\{F_Y(Y_1), \dots, F_Y(Y_{p})\}\mid F_X(X_0)>u],
\end{split}
\end{equation*}
where, as usual,  $F_X, F_Y$ are the marginal distribution functions of $\mathbf{X}, \mathbf{Y}$, respectively.

Notice that $\Gamma^{time}_{\mathbf{X}\to \mathbf{Y}}(p)\in[0,1]$, and $\Gamma^{time,-0}_{\mathbf{X}\to \mathbf{Y}}(p)\leq \Gamma^{time}_{\mathbf{X}\to \mathbf{Y}}(p)\leq \Gamma^{time}_{\mathbf{X}\to \mathbf{Y}}(p+1)$. Furthermore, for any increasing functions $h_1$ and $h_2:\mathbb{R}\to\mathbb{R}$, we observe $\Gamma^{time}_{\mathbf{X}\to \mathbf{Y}}(p) = \Gamma^{time}_{h_1(\mathbf{X})\to h_2(\mathbf{Y})}(p)$, where $h_1(\mathbf{X})=(h_1(X_t), t\in\mathbb{Z})$, as $\Gamma^{time}_{\mathbf{X}\to \mathbf{Y}}(p)$ depends solely on rescaled margins $F_X(X_i)$ and $F_Y(Y_i)$.

\subsection{Models}
\label{SectionModels}
In this paper, we work with two models of time series--- $\VAR(q)$ process (vector autoregressive process of order $q$, Section 2.3.1 in  \cite{MVT_timeseries_SPRINGER}) and $\NAAR(q)$ model (nonlinear additive
autoregressive model of order $q$). We now introduce the notation. 

\begin{definition}
\label{VAR(q)}
We say that $(\mathbf{X},\mathbf{Y})^\top=((X_t,Y_t)^\top, t\in\mathbb{Z})$ follows the bivariate $\VAR(q)$ model, if it has the following representation:
\begin{align*}
 X_t&= \alpha_{1}X_{t-1}+\dots + \alpha_{q}X_{t-q}+ \gamma_{1}Y_{t-1}+\dots +\gamma_{q}Y_{t-q} + \varepsilon_t^X,\\
Y_t&=\beta_{1}Y_{t-1}+\dots + \beta_{q}Y_{t-q} + \delta_{1}X_{t-1}+\dots +\delta_{q}X_{t-q} + \varepsilon_t^Y,
\end{align*}
where $\alpha_i, \beta_i, \gamma_i, \delta_i\in\mathbb{R}$, $i=1,\dots,q$, are real constants, and $(\varepsilon_t^X, t\in\mathbb{Z})$, $(\varepsilon_t^Y, t\in\mathbb{Z})$ are white noises. We say that it satisfies the stability condition, if $\deter(I_{d}-A_1z-\dots-A_qz^q)\neq 0$ for all $|z|\leq 1$, where $I_d$ denotes the $d$-dimensional identity matrix and $A_i = \begin{pmatrix}
\alpha_i & \gamma_i  \\
\beta_i & \delta_i 
\end{pmatrix}$.
We say that $\mathbf{X}$ (Granger) causes $\mathbf{Y}$ (notation $\mathbf{X}\to\mathbf{Y}$) if there exists $i\in \{1, \dots, q\}$ such that $\delta_i\neq 0$.  
\end{definition}
Let $i,j\in\mathbb{Z}: 0\leq j-i\leq q$. We can specify that $X_i$ causes $Y_j$ if $\delta_{j-i}\neq 0$. Note that $\mathbf{X}$ causes $\mathbf{Y}$ if and only if there exist $i\leq j$ such that $X_i$ causes $Y_j$. In this paper, our focus is on exploring whether $\mathbf{X}$ causes $\mathbf{Y}$. It's also important to highlight that a situation where $\mathbf{X}$ causes $\mathbf{Y}$ and vice versa simultaneously is admissible.

Under the stability assumption, we can rewrite these time series using causal representations \citep[Section 2.1.3]{MVT_timeseries_SPRINGER}:
\begin{equation} \label{causalVAR}
X_t=\sum_{i=0}^\infty a_i\varepsilon^X_{t-i}+\sum_{i=0}^\infty c_i\varepsilon^Y_{t-i};\qquad
Y_t=\sum_{i=0}^\infty b_i\varepsilon^Y_{t-i}+\sum_{i=0}^\infty d_i\varepsilon^X_{t-i},
\end{equation}
with suitable constants $a_i, b_i, c_i, d_i\in\mathbb{R}$. Thus, $\mathbf{X}$ causes $\mathbf{Y}$ if and only if there exists $i: d_i\neq 0$ \citep{GrangerSims}. 

Now, we state our model assumptions. 
\begin{definition}[Heavy-tailed VAR model] \label{heavy-tailed-VAR}
Let $(\mathbf{X},\mathbf{Y})^\top$ follow the stable $\VAR(q)$ model as defined above with its causal representation given by \eqref{causalVAR}. We introduce the assumptions:
\begin{itemize}
\item   $\varepsilon_t^X, \varepsilon_t^Y\overset{\text{iid}}{\sim}\RV(\theta)$ for some $\theta>0$,
\item  $\alpha_i, \beta_i, \gamma_i, \delta_i\geq 0$,
\item  $\exists \delta>0$ such that $\sum_{i=0}^\infty a_i^{\theta-\delta}<\infty, \sum_{i=0}^\infty b_i^{\theta-\delta}<\infty, \sum_{i=0}^\infty c_i^{\theta-\delta}<\infty$,\\$ \sum_{i=0}^\infty d_i^{\theta-\delta}<\infty$.
\end{itemize}
Under these assumptions, we refer to $(\mathbf{X},\mathbf{Y})^\top$ as following the \textit{heavy-tailed $\VAR(q, \theta)$ model} ($\hVAR(q,\theta)$ for short).
\end{definition}

The first assumption is crucial, as it ensures the regular variation of our time series. The second assumption can be relaxed and can be replaced by the extremal causal condition discussed in Subsection \ref{real-valued}. The third assumption is aimed at ensuring the a.s. summability of the sums $\sum_{i=0}^\infty a_i\varepsilon^X_{t-i}$.  Moreover, this assumption guarantees the stationarity of $\mathbf{X}$ and $\mathbf{Y}$. It also establishes a crucial max-sum equivalence relationship, namely $\pr(\sum_{i=0}^\infty \alpha_i\varepsilon_i^\cdot>u)\sim [\sum_{i=0}^\infty \alpha_i^\theta] \pr(\varepsilon_1^\cdot>u)$ \citep[Lemma A.3]{Theorem1.7}.

We consider a nonlinear generalization in order to minimize the assumptions on the body of the time series. First, we introduce the nonlinear counterpart of Definition~\ref{VAR(q)}.

\begin{definition}
\label{AFVAR}
We define $(\mathbf{X},\mathbf{Y})^\top=((X_t,Y_t)^\top, t\in\mathbb{Z})$ to follow the bivariate $\NAAR(q)$ model, specified by the equations:
\begin{align*}
X_t=f_{1}(X_{t-1}) + f_{2}(Y_{t-q}) + \varepsilon_t^X; 
\qquad
Y_t=g_{1}(Y_{t-1}) + g_{2}(X_{t-q}) + \varepsilon_t^Y.
\end{align*}
We state that $\mathbf{X}$ (Granger) causes $\mathbf{Y}$ if $g_2$ is a non-constant function on the support of $X_{t-q}$ (a.s.).
\end{definition}
It is interesting to note that in the univariate case $d=1$, an often used condition $\lim_{|x|\to\infty}\frac{|f_1(x)|}{|x|}<1$ is \say{almost} necessary for stationarity (\cite[Corollary 2.2]{Bhattacharya} or \cite[Theorem 2.2]{Andel}). 

\begin{definition}[Heavy-tailed NAAR model]
Let $(\mathbf{X},\mathbf{Y})^\top$ follow the stationary $\NAAR(q)$ model from Definition \ref{AFVAR}. We require functions $f_1, f_2, g_1, g_2$ to either be zero functions or continuous non-negative functions satisfying  $\lim_{x\to\infty}h(x)=\infty$ and $\lim_{x\to\infty}\frac{h(x)}{x}<1$ for $h=f_1, f_2, g_1, g_2$.
Moreover, let $\varepsilon_t^X, \varepsilon_t^Y\overset{\text{iid}}{\sim}\RV(\theta)$ be non-negative for some $\theta>0$. Then, we say that $(\mathbf{X},\mathbf{Y})^\top$ follows the \textit{heavy-tailed $\NAAR(q,\theta)$ model} ($\hNAAR(q,\theta)$ for short). 
\end{definition}

It's important to emphasize that our restrictions on the functions $f_1, f_2, g_1, g_2$ are minimal in the body -- we've only imposed conditions on their tails.  By imposing the constraint $\lim_{x\to\infty}\frac{h(x)}{x}<1$ for $h=f_2, g_2$, we ensure that the tail of $f_2(Y_{t-q})$ (resp. $g_2(X_{t-q})$) is not larger than the tail of $Y_{t-q}$ (resp. $X_{t-q}$). Assumption  $\lim_{x\to\infty}\frac{h(x)}{x}<1$ for $h=f_1, g_1$ is closely related to the time series' stationarity. In the $\hNAAR$ models, our assumption of regular variation for the noise variables directly implies that $X_t, Y_t\sim\RV(\theta)$ \cite[Theorem 2.3]{heavy_tailed_nonlinear_time_series}. 
Moreover, our framework doesn't assume the existence of any moments; it accommodates cases like $\theta<1$, where the expectation $\mathbb{E}(X_t)$ does not exist.

Obviously, $\hNAAR(q,\theta)$ models encompass a broader class of time series models compared to $\hVAR(q,\theta)$ models. However, they are not nested. In the $\hNAAR(q,\theta)$ case, we assumed that the $X_t$ and $Y_t$ are functions of only two previous values ($X_t$ being a function of $X_{t-1},Y_{t-q}$ and $Y_t$ being a function of $Y_{t-1}, X_{t-q}$). In the VAR case, $X_t, Y_t$ can depend on more than two previous values. Moreover, the NAAR case has an additional assumption $\varepsilon_t^X, \varepsilon_t^Y\geq 0$. Nevertheless, up to these two differences, the class of  $\hVAR(q,\theta)$  models lies inside of the class of $\hNAAR(q,\theta)$ models. 

\subsection{Causal direction}
\label{Section 2.2}
The subsequent two theorems constitute the core of this paper, connecting the classical concept of causality with causality in extreme events.
\begin{theorem}\label{Theorem 2.1} 
Let  $(\mathbf{X},\mathbf{Y})^\top$ be a bivariate time series which follows either the $\hVAR(q,\theta)$ model or the $\hNAAR(q,\theta)$ model. If $\mathbf{X}$ causes $\mathbf{Y}$, then $\Gamma^{time}_{\mathbf{X}\to \mathbf{Y}}(q)=1$. 
\end{theorem}
The proof can be found in \hyperlink{Proof of Theorem 2.1.}{Appendix \ref{AppendixB}}.
Intriguingly, the regular variation condition is not used within the proof.  We assume that we know the exact (correct) order $q$. Nevertheless, for every $p\geq q$, we have $\Gamma^{time}_{\mathbf{X}\to \mathbf{Y}}(p)\geq \Gamma^{time}_{\mathbf{X}\to \mathbf{Y}}(q) =1$. The choice of an appropriate delay $p$ will be discussed in Subsection \ref{optimal lag section}. 

\begin{theorem} \label{Theorem 2.2.}
Let  $(\mathbf{X},\mathbf{Y})^\top$ be a bivariate time series which follows either the $\hVAR(q, \theta)$ model or the $\hNAAR(q, \theta)$ model. If $\mathbf{Y}$ does not cause $\mathbf{X}$, then $\Gamma^{time}_{\mathbf{Y}\to \mathbf{X}}(p)<1$ for all $p\in\mathbb{N}$. 
\end{theorem}
The proof can be found in \hyperlink{Proof of Theorem 2.2.}{Appendix \ref{AppendixB}}.
The primary step of the proof stems from Proposition \ref{TentoTheorem}, presented in Appendix \ref{chapter 5}. The core idea is that large sums of independent, regularly varying random variables tend to be driven by only a single large value. Consequently, if $Y_0$ is large, it could be attributed to the largeness of an $\varepsilon_i^Y$, which does not affect $\mathbf{X}$. 

Note that distinct notation is employed for the time series order (denoted as $q\in\mathbb{N}$) and the extremal delay (denoted as $p\in\mathbb{N}$). Although these two coefficients need not be equivalent, Theorem \ref{Theorem 2.1} prompts our primary focus on scenarios where $p\geq q$.

\begin{example}
Consider the bivariate time series $(\mathbf{X},\mathbf{Y})^\top$ described by the equations:
\begin{align*}
X_t&=0.5X_{t-1}+ \varepsilon_t^X;\qquad
Y_t=0.5Y_{t-1}+0.5X_{t-1}+ \varepsilon_t^Y,
\end{align*}
where $\varepsilon_t^X, \varepsilon_t^Y\overset{\text{iid}}{\sim}{\rm Pareto}(1,1)$, with a tail index $\theta=1$.
A causal representation of this $\hVAR(1,1)$ model is given by:
\begin{align*}
X_t=\sum_{i=0}^\infty \frac{1}{2^i}\varepsilon^X_{t-i}; \qquad
Y_t=\sum_{i=0}^\infty \frac{1}{2^i}\varepsilon^Y_{t-i}+\sum_{i=0}^\infty \frac{i}{2^i}\varepsilon^X_{t-i}.
\end{align*}
In this case, the order is $q=1$, and it is sufficient to take only $$\Gamma^{time, -0}_{\mathbf{X}\to \mathbf{Y}}(1)=\lim_{u\to 1^-}\E[F_Y(Y_1)\mid F_X(X_0)>u],$$ as discussed in Subsection \ref{optimal lag section}. Let us give some vague computation of this coefficient. From Theorem \ref{Theorem 2.1}, we know that $\Gamma^{time}_{\mathbf{X}\to \mathbf{Y}}(1)=1$. For the other direction, rewrite 
\begin{equation*}
\begin{split}
&\lim_{u\to 1^-}\E[F_X(X_1)\mid F_Y(Y_0)>u]= \lim_{v\to \infty}\E[F_X(X_1)\mid \sum_{i=0}^\infty \frac{1}{2^i}\varepsilon^Y_{-i}+\sum_{i=0}^\infty \frac{i}{2^i}\varepsilon^X_{-i}>v].
\end{split}
\end{equation*}
First, note the following relations (the first follows from the independence and the second from  Lemma \ref{dva} in Appendix \ref{chapter 5}):  
\begin{align*}
&\lim_{v\to \infty}\E[F_X(X_1)\mid \sum_{i=0}^\infty \frac{1}{2^i}\varepsilon^Y_{-i}>v]=\E[F_X(X_1)]=1/2,\\
&\lim_{v\to \infty}\E[F_X(X_1)\mid \sum_{i=0}^\infty \frac{i}{2^i}\varepsilon^X_{-i}>v]=1. 
\end{align*}
Furthermore, we know that $\pr(X_1<\lambda\mid \sum_{i=0}^\infty \frac{1}{2^i}\varepsilon^Y_{-i}+\sum_{i=0}^\infty \frac{i}{2^i}\varepsilon^X_{-i}>v)= \frac{\pr(X_1<\lambda)}{2}$ for every constant $\lambda\in\mathbb{R}$ (using Proposition \ref{TentoTheorem} \footnote{Note the identities $\sum_{i=0}^\infty \frac{i}{2^i}=2=\sum_{i=0}^\infty \frac{1}{2^i}$.}). This implies that, with a probability of $1/2$, $X_1|{F_Y(Y_0)>u}$ has the same distribution as non-conditional $X_1$, and with complementary probability, it tends to $\infty$ as $u\to 1^-$. Combining these results, we obtain: 
$$\Gamma^{time, -0}_{\mathbf{Y}\to \mathbf{X}}(1) = \lim_{u\to 1^-}\E[F_X(X_1)\mid F_Y(Y_0)>u]=\frac{1}{2}\cdot  \frac{1}{2} + \frac{1}{2}\cdot 1=\frac{3}{4}.$$ 
The order $q$ is usually unknown. If we take $p=2$, then the value of $$\Gamma^{time}_{\mathbf{Y}\to \mathbf{X}}(2)=\lim_{u\to 1^-}\E[\max\{F_X(X_0), F_X(X_1), F_X(X_2)\}\mid F_Y(Y_0)>u]$$ will be slightly larger than $\frac{3}{4}$. More precisely, it will be equal to $\frac{1}{2}\cdot  \E[\max\{F_X(X_0), F_X(X_1),F_X(X_2)\}] + \frac{1}{2}\cdot 1 $. The true value is around $0.80$, as determined through simulations and the use of computer software.

\end{example}

\section{Properties and extensions}
\label{chapter 3}

In this section, we exclusively focus on Heavy-tailed VAR models. While we believe that similar results also apply to Heavy-tailed NAAR models, the formulation and proof of such results are beyond the scope of this paper.

\subsection{Real-valued coefficients}
\label{real-valued}

We discuss the extension to include potentially negative coefficients (as indicated by the second assumption in Definition \ref{heavy-tailed-VAR}) and non-directly proportional dependencies. Until now, we have assumed that a large positive $\mathbf{X}$ causes large positive $\mathbf{Y}$. In other words, we have assumed that all coefficients in the $\hVAR$ model are non-negative.

The most straightforward modification arises when a large positive $\mathbf{X}$ causes large negative $\mathbf{Y}$ (where the gain for one causes a loss for others). In such cases, it suffices to consider the maxima of the pair $(\mathbf{X}, -\mathbf{Y})$. While we apply this simple modification in our application, it cannot be universally applied.

The concept behind extending the causal tail coefficient for time series involves utilizing the absolute values $|\mathbf{X}|$ and $|\mathbf{Y}|$ instead of $\mathbf{X}$ and $\mathbf{Y}$. However, to implement this extension, we need to restrict the family of $\VAR(q)$ models in a specific manner. Theorems \ref{Theorem 2.1} and \ref{Theorem 2.2.} do not hold if we allow arbitrary coefficients $\alpha_i$, $\beta_i$, $\gamma_i$, $\delta_i \in \mathbb{R}$, whether positive or negative, where $i=1, \dots, q$. The following example illustrates a problematic scenario where Theorems \ref{Theorem 2.1} and \ref{Theorem 2.2.} do not hold.

\begin{example}
\label{priklad s divnymi minusmi}
Consider a vector $(\mathbf{X}, \mathbf{Y})^\top$ following a $\VAR(2)$ model defined as:
\begin{align*}
X_t&=0.5X_{t-1}+ \varepsilon_t^X; \qquad
Y_t=X_{t-1} - 0.5 X_{t-2}+ \varepsilon_t^Y. 
\end{align*}
Its causal representation can be written as 
\begin{align*}
X_t&=\sum_{i=0}^\infty \frac{1}{2^i}\varepsilon^X_{t-i};\qquad
Y_t=\varepsilon_t^Y + \varepsilon^X_{t-1}. 
\end{align*}

Detecting extreme causal relationships can be challenging in this model. This is because, despite $\mathbf{X}$ being a cause of $\mathbf{Y}$, it cannot be assumed that the extreme in $X_{t-1}$ will necessarily result in an extreme in $Y_t$. Even if $X_{t-2}$ is large, leading to a large $X_{t-1}$, the same does not hold for $Y_t$. To address this, we will impose a restriction on our model in the following manner.
\end{example}

\begin{definition}[Extremal causal condition]
Consider the pair $(\mathbf{X},\mathbf{Y})^\top$, which follows the stable $\VAR(q)$ model, with its causal structure taking the form:
\begin{align*}
X_t=\sum_{i=0}^\infty a_i\varepsilon^X_{t-i}+\sum_{i=0}^\infty c_i\varepsilon^Y_{t-i};\qquad
Y_t=\sum_{i=0}^\infty b_i\varepsilon^Y_{t-i}+\sum_{i=0}^\infty d_i\varepsilon^X_{t-i}.
\end{align*}
Let $\mathbf{X}$ cause $\mathbf{Y}$. We say that $(\mathbf{X},\mathbf{Y})^\top$ satisfies an \textit{extremal causal condition}, if there exists an integer $r \leq q$ such that, for every $i\in \mathbb{N}\cup \{0\}$: $$a_i\neq 0 \implies d_{i+r} \neq 0.$$ 
\end{definition}

\begin{lemma}
\label{lemma Extremal causal condition}
The extremal causal condition holds in the $\hVAR(q, \theta)$ models (i.e., where the coefficients are non-negative) when $\mathbf{X}$ causes $\mathbf{Y}$. 
\end{lemma}
The proof can be found in \hyperlink{proof of lemma Extremal causal condition}{Appendix \ref{AppendixB}}. Clearly, the aforementioned condition holds in models where $d_i \neq 0$ for all $i > 0$.

The extremal causal condition embodies the following concept: Take $k\geq r$. If $\varepsilon_t^X$ is \say{extreme}, this extreme event will also influence $Y_{t+k}$. This implication is in particular true if some of the coefficients are negative. If $\varepsilon_t^X$ is \say{extremely positive}, this implies that $Y_{t+k}$ will be \say{extremely  negative}. Nevertheless, under the extremal causal condition, if  $\varepsilon_t^X$ is \say{extreme}, $Y_{t+k}$ will also be \say{extreme}. 

Formulating similar conditions for $\NAAR$ models is challenging, since we would require that an appropriate combination of \textit{functions} is non-zero. Due to the extremal causal condition, our focus in this section remains exclusively on $\VAR$ models, excluding $\NAAR$ models. The following theorem demonstrates that, even when negative coefficients are introduced, the principles outlined in Section \ref{chapter 2} remain applicable, thanks to the extremal causal condition.

\begin{theorem} 
\label{Absolute value theorem}
Let  $(\mathbf{X},\mathbf{Y})^\top$ be a time series which follows the $\hVAR(q, \theta)$ model, with possibly negative coefficients, satisfying the extremal causal condition. Moreover, let $\varepsilon_t^X, \varepsilon_t^Y$ have full support on $\mathbb{R}$, and $|\varepsilon_t^X|, |\varepsilon_t^Y|\sim \RV(\theta)$.  If $\mathbf{X}$ causes $\mathbf{Y}$, and $\mathbf{Y}$ does not cause $\mathbf{X}$, then $\Gamma^{time}_{|\mathbf{X}|\to |\mathbf{Y}|}(q)=1$, and $\Gamma^{time}_{|\mathbf{Y}|\to |\mathbf{X}|}(q)<1$. \end{theorem}

The proof can be found in \hyperlink{Proof Absolute value theorem}{Appendix \ref{AppendixB}}.

\subsection{Common cause and different tail behavior}

Reichenbach's common cause principle \citep{Reichenbach} states that for every given pair of random variables $(\mathbf{X}, \mathbf{Y})$, precisely one of the following statements holds true: $\mathbf{X}$ causes $\mathbf{Y}$, $\mathbf{Y}$ causes $\mathbf{X}$, they are independent or there exists $\mathbf{Z}$ causing both $\mathbf{X}$ and $\mathbf{Y}$. The challenge lies in distinguishing between true causality and dependence stemming from a common cause. The subsequent theorem demonstrates that our methodology allows us to distinguish between causality and a correlation due to a common cause.
\begin{theorem}
\label{Common cause theorem}
Let $(\mathbf{X,Y,Z})^\top = ((X_t,Y_t,Z_t)^\top, t\in\mathbb{Z})$  follow the three-dimensional stable $\VAR(q)$ model, with iid regularly varying noise variables. Let $\mathbf{Z}$ be a common cause of both $\mathbf{X}$ and $\mathbf{Y}$, and neither $\mathbf{X}$ nor $\mathbf{Y}$ cause $\mathbf{Z}$. If $\mathbf{Y}$ does not cause $\mathbf{X}$, then $\Gamma^{time}_{\mathbf{Y}\to \mathbf{X}}(p)<1$ for all $p\in\mathbb{N}$. 
\end{theorem}
The proof can be found in \hyperlink{Proof of Common cause theorem}{Appendix \ref{AppendixB}}. Theorem \ref{Common cause theorem} makes an initial assumption that the tail indexes of $\mathbf{X}, \mathbf{Y},$ and $\mathbf{Z}$ are equal. However, it can be easily extended to the case when  $\mathbf{Z}$ has lighter tails (see discussion in Section \ref{section hidden confounder}). In practical scenarios, complete observation of all pertinent data is often unattainable. Nevertheless, Theorem \ref{Common cause theorem} remains valid even when we do not observe the common cause. However, the common cause still needs to fulfill the condition that noise is regularly varying with not heavier tails than those of $\mathbf{X}$ and $\mathbf{Y}$. We can not check this assumption in practice. 

\begin{example}\label{Example3}
Let $(\mathbf{X,Y,Z})^\top$ follow the three-dimensional $\VAR(1)$ model, specified by
\begin{align*}
Z_t&=0.5Z_{t-1}+\varepsilon_t^Z,\\
X_t&=0.5X_{t-1}+0.5Z_{t-1} + \varepsilon_t^X,\\
Y_t&=0.5Y_{t-1}+0.5Z_{t-1} +0.5X_{t-1} + \varepsilon_t^Y,
\end{align*}
where $\varepsilon_t^X, \varepsilon_t^Y\overset{\text{iid}}{\sim}{\rm Pareto}(2,2)$ and $\varepsilon_t^Z \overset{\text{iid}}{\sim}{\rm Pareto}(1,1)$ (i.e., $\varepsilon_t^Z$ has a heavier tail than $\varepsilon_t^X, \varepsilon_t^Y$). Then $\Gamma^{time}_{\mathbf{X}\to \mathbf{Y}}(1)=\Gamma^{time}_{\mathbf{Y}\to \mathbf{X}}(1)=1$ \footnote{We do not provide a rigorous proof of this equality, but such a proof follows similar steps as the proof of Theorem \ref{Theorem 2.1}. We can use the fact that $\pr(\varepsilon_t^Z + \varepsilon_t^X>v)\sim \pr(\varepsilon_t^Z>v)$ as $v\to\infty$ and appropriately modify Proposition \ref{TentoTheorem} such that $\varepsilon_t^Z$ is the leading term on both sides. The same results can be seen from simulations in Subsection \ref{section hidden confounder}.}, even though $\mathbf{Y}$ does not cause $\mathbf{X}$.

Even within the context of the unconfounded scenario (or when $\varepsilon_t^Z$ possesses tails that are lighter than those of $\varepsilon_t^X$ and $\varepsilon_t^Y$), a distinct challenge emerges when the tail behaviors of $\varepsilon_X$ and $\varepsilon_Y$ differ. Results from Section \ref{chapter 2} remain applicable if $\varepsilon_X$ has lighter tails than $\varepsilon_Y$. However, if $\varepsilon_X$  has heavier tails than $\varepsilon_Y$, then  $\Gamma^{time}_{\mathbf{X}\to \mathbf{Y}}(1)=\Gamma^{time}_{\mathbf{Y}\to \mathbf{X}}(1)=1$ and we can not distinguish between the cause and the effect. A more detailed discussion can be found in Subsection \ref{section hidden confounder}.
\end{example}

\subsection{Estimating the extremal delay $p$}
\label{optimal lag section}

Up to this point, we have assumed prior knowledge of the exact order $q$ of our time series, with $p$ being set equal to $q$. However, what if this order is unknown? In cases where $p$ is too small, accurate causal relationships cannot be obtained (refer to Lemma \ref{minimal lag lemma} below). Conversely, as $p$ becomes larger, $\Gamma^{time}_{\mathbf{Y}\to \mathbf{X}}(p)$ approaches $1$, posing challenges for empirical inference.

One practical approach to address this challenge is to utilize the extremogram \citep{Extremogram}. Similar to the role of correlograms in classical cases, extremograms help us select a \say{reasonable} value for $p$ by examining plots.  However, we can examine the values of $\Gamma^{time}_{\mathbf{Y}\to \mathbf{ X}}(p)$ for a range of values for $p$. An illustrative example can be found in Subsection \ref{section Lag}.

Now, let's delve into another consideration: the problem of estimating the temporal synchronization between two time series $(\mathbf{X}, \mathbf{Y})^\top$, where $\mathbf{X}$ is the cause of $\mathbf{Y}$. The objective is to determine the time required for information originating from $\mathbf{X}$ to impact $\mathbf{Y}$. In the context of an intervention applied to $\mathbf{X}$, the question is: \textit{when} will this intervention influence $\mathbf{Y}$? A practical instance of this scenario might arise in economics, such as when dealing with two time series representing milk and cheese prices over time. Consider a situation where the government imposes a large tax increase on milk prices at a specific moment. In such a case, when can we expect to observe a subsequent rise in cheese prices? Mathematically, we aim to estimate a parameter known as the \textit{minimal delay}.

\begin{definition}[Minimal delay]
Let $(\mathbf{X},\mathbf{Y})^\top$ follow the stable $\VAR(q)$ model specified in
Definition \ref{VAR(q)}.
We call $s\in\mathbb{N}$ the \textit{minimal delay}, if $\gamma_1=\dots = \gamma_{s-1} = \delta_1=\dots = \delta_{s-1}=0$ and either $\delta_s\neq 0$ or $\gamma_s\neq 0$. If such $s$ does not exist, we define the minimal delay as $+\infty$.

\end{definition}

The subsequent lemma establishes connections between the minimal delay and the causal tail coefficient for time series.

\begin{lemma}
\label{minimal lag lemma}
Let $(\mathbf{X},\mathbf{Y})^\top$ follow the $\hVAR(q, \theta)$ model, where $\mathbf{X}$ causes $\mathbf{Y}$. Let $s$ be the minimal delay. Then, $\Gamma^{time}_{\mathbf{X}\to \mathbf{Y}}(r)<1$ for all $r<s$, and $\Gamma^{time}_{\mathbf{X}\to \mathbf{Y}}(r)=1$ for all $r\geq s$. 
\end{lemma}
The proof can be found in \hyperlink{Proof of minimal lag lemma}{Appendix \ref{AppendixB}}. 
An approach to estimating the minimal delay $s$ involves identifying the smallest value of $s$ for which $\Gamma^{time}_{\mathbf{X}\to \mathbf{Y}}(s)=1$.

Estimating the minimal delay $s$ through extreme values holds relevance across various applications. One such application involves the assessment of \textit{synchrony} between two time series. \textit{Synchrony}, as described in \citep{synchrony}, quantifies the interdependence of two time series, seeking to determine an optimal lag that aligns the time series most closely. Conventional methodologies encompass techniques like time-lagged cross-correlation or dynamic time warping. Utilizing the causal tail coefficient for time series can offer an avenue to proceed using extreme values, facilitating the conception of \textit{synchronization of extremes}. This may pave the way for future research and methodological developments.

\section{Estimations and simulations}
\label{chapter 4}

All the methodologies presented in this section have been implemented in the \textsf{R} programming language \citep{R}. The corresponding code can be accessed either in the supplementary package or through the online repository available at \url{https://github.com/jurobodik/Causality-in-extremes-of-time-series.git}.

\subsection{Non-parametric estimator}

We propose an estimator of the causal tail coefficient $\Gamma^{time}_{\mathbf{X}\to \mathbf{Y}}(p)$ based on a finite sample $(x_1, y_1)^\top, \dots, (x_n, y_n)^\top$. The estimator uses only the values of $y_i$ where the corresponding $x_i$ surpasses a given threshold.

\begin{definition}
\label{DEF}
We define
$$\hat\Gamma^{time}_{\mathbf{X}\to \mathbf{Y}}(p):=\frac{1}{k}\sum_{\substack{i\leq n-p \\ x_i\geq\tau_k^X}}\max\{\hat{F}_Y(y_i), \dots, \hat{F}_Y(y_{i+p})\}, $$ where $\tau_k^X=x_{(n-k+1)}$ is the $k$-th largest value of $x_1, \dots, x_n$, and $\hat{F}_Y(t)=\frac{1}{n}\sum_{j=1}^n \1(y_j\leq t)$. 
\end{definition}

The value $k$ signifies the number of extremes considered in the estimator. In the upcoming discussions, $k$ will be contingent upon $n$, leading us to employ the notation $k_n$ instead of $k$. A fundamental condition in extreme value theory is expressed as:
 \begin{equation}
 \label{k_deleno_n}
 k_n\to\infty, \; \frac{k_n}{n}\to 0, \quad \text{ as } n\to\infty.
 \end{equation}

The next theorem shows that such a statistic is \say{reasonable} by showing that it is asymptotically unbiased. This result doesn't necessitate any stringent model assumptions; it only relies on the assumption of a certain rate of convergence for the empirical distribution function. 

\begin{theorem}
\label{Theorem o asymptotic}
Let $(\mathbf{X},\mathbf{Y})^\top=((X_t,Y_t)^\top, t\in\mathbb{Z})$ be a stationary bivariate time series, whose marginal distributions are absolutely continuous with support on some neighborhood of infinity. Let $\Gamma^{time}_{\mathbf{X}\to \mathbf{Y}}(p)$ exist. Let $k_n$ satisfy (\ref{k_deleno_n}) and 
\begin{equation}
\label{zxc}
\frac{n}{k_n}P\left(\frac{n}{k_n} \sup_{x\in\mathbb{R}}|\hat{F}_X(x)-F_X(x)|>\delta\right)\overset{n\to\infty}{\longrightarrow}0, \,\,\, \forall \delta>0.
\end{equation}
Then, $\E \hat\Gamma^{time}_{\mathbf{X}\to \mathbf{Y}}(p)\overset{n\to\infty}{\to}\Gamma^{time}_{\mathbf{X}\to \mathbf{Y}}(p)$ \footnote{$\hat\Gamma^{time}_{\mathbf{X}\to \mathbf{Y}}(p)$ depends on $n$, although we omitted this index for clarification.}.
\end{theorem}

\begin{consequence}\label{Consequence5}
Let $\mathbf{X}\to\mathbf{Y}$. Under the assumptions of Theorem~\ref{Theorem o asymptotic} and Theorem~\ref{Theorem 2.1}, the proposed estimator is consistent; that is,  $\hat\Gamma^{time}_{\mathbf{X}\to \mathbf{Y}}(p)\overset{P}{\to}\Gamma^{time}_{\mathbf{X}\to \mathbf{Y}}(p)$ as $n\to\infty$. 
\end{consequence}
The proofs of Theorem~\ref{Theorem o asymptotic} and Consequence~\ref{Consequence5} can be found in \hyperlink{Proof of asymptotic unbias}{Appendix \ref{AppendixB}}. 
To establish asymptotic properties of $\hat\Gamma^{time}_{\mathbf{X}\to \mathbf{Y}}(p)$, it is reasonable to assume the consistency of $\hat F_X$. However, this assumption is accompanied by an additional requirement for the rate of its convergence, as specified in condition \eqref{zxc}. While this condition is not overly restrictive, we believe that the condition \eqref{zxc} can be improved. For iid random variables, condition \eqref{zxc} holds if $k_n^2/n \to \infty$ (follows from the Dvoretzky--Kiefer--Wolfowitz inequality). The following lemma establishes that condition \eqref{zxc} is also fulfilled within a certain subset of autoregressive models.

 \begin{lemma}
 \label{lemma o concentration inequality}
 Let $\mathbf{X} = (X_t, t\in\mathbb{Z})$ has the form 
 $$
 X_t = \sum_{k=0}^\infty a_k \varepsilon_{t-k},
 $$
where $(\varepsilon_t, t \in \mathbb{Z})$ are iid random variables with density $f_\varepsilon$. Assume $|a_k|\leq \gamma k^{-\beta}$ for some $\beta>1$, $\gamma>0$ and all $k \in \mathbb{N}$.  Let $f^\star = \max(1, |f_\varepsilon|_{\infty}, |f'_{\varepsilon}|_{\infty}) < \infty$, where $|f|_{\infty} = \sup_{x\in\mathbb{R}}|f(x)|$ is the supremum norm. Assume $\E|\varepsilon_0|^q<\infty$ for $q>2$ and $\pr(|\varepsilon_0|>x) \leq L(\log x)^{-r_0}x^{-q}$ for some constants $L>0, r_0>1$ and for every $x>1$. If the sequence $(k_n)$ satisfies \eqref{k_deleno_n} and 
\begin{equation}\label{aaa}
\exists c>\max\left\{ \frac{1}{2},\frac{2}{1+q\beta}\right\}: \frac{k_n}{n^c}\to\infty,  \text{ as }n\to\infty,
\end{equation}
then the condition \eqref{zxc} is satisfied. 

 \end{lemma}

The proof can be found in \hyperlink{proof o concentration inequality}{Appendix \ref{AppendixB}}. It is based on Proposition 13 in \cite{Concentration_inequality}. Several concentration bounds exist that describe the convergence speed of the empirical distribution functions, and different assumptions can be made to guarantee the validity of the condition \eqref{zxc}. According to Theorem 1.2 in \cite{kontorovich}, condition \eqref{zxc} holds for geometrically ergodic Markov chains. Regrettably, this result was originally presented solely for $\mathbb{N}$-valued random variables. The Dvoretzky--Kiefer--Wolfowitz inequality provides a similar outcome, but we did not come across a specific adaptation of this inequality for autoregressive time series.

Lemma \ref{lemma o concentration inequality} requires a finite second moment of $\varepsilon_i$, a technical condition that follows from the findings in \cite{Concentration_inequality}. In cases of heavy-tailed time series characterized by a tail index $\theta$, this condition necessitates $\theta>2$. Notably, the coefficient $\beta$ in Lemma \ref{lemma o concentration inequality} encapsulates the strength of dependence in the time series. Larger values of $\beta$ correspond to faster decay of dependence, thereby allowing for milder assumptions on the intermediate sequence $(k_n)$. Throughout the paper, we opt for the choice $k_n=\sqrt{n}$, a decision briefly explored in Subsection \ref{section Threshold}.

\subsection{Some insight using simulations}
\label{UHNJ}
We will simulate the behavior of the estimates $\hat\Gamma_{\mathbf{X}\to \mathbf{Y}}^{time}$ for a series of models. Initially, we employ the Monte Carlo principle to estimate the distributions of $\hat\Gamma_{\mathbf{X}\to \mathbf{Y}}^{time}$ and $\hat\Gamma_{\mathbf{Y}\to \mathbf{X}}^{time}$ for the following model.
\begin{model}
\label{Time series 1}
Let $(\mathbf{X},\mathbf{Y})^\top$ follow the $\VAR(2)$ model
\begin{align*}
X_t&=0.5X_{t-1}+ \varepsilon_t^X; \qquad
Y_t=0.5Y_{t-1}+ \delta X_{t-2}+\varepsilon_t^Y,
\end{align*}
where $\varepsilon_t^X, \varepsilon_t^Y$ are independent noise variables and $\delta\in\mathbb{R}$. 
\end{model}
Figure~\ref{histogram} presents the histograms of $\hat\Gamma_{\mathbf{X}\to \mathbf{Y}}^{time}(2)$ and $\hat\Gamma_{\mathbf{Y}\to \mathbf{X}}^{time}(2)$ from $1000$ simulated instances of time series, each with a length of $n=5000$, following Model 1. The parameter $\delta$ is set to $0.5$, and $\varepsilon_t^X$ and $\varepsilon_t^Y$ are drawn from the Cauchy distribution. Figure~\ref{histogram} illustrates that within the causal direction, the values of $\hat\Gamma_{\mathbf{X}\to \mathbf{Y}}^{time}$ predominantly fall within the range of $(0.9, 1)$. Conversely, in the non-causal direction, the values of $\hat\Gamma_{\mathbf{Y}\to \mathbf{X}}^{time}$ typically lie below 0.9. This underscores the asymmetry between the cause and the effect.

\begin{figure}[b]
\centering
\includegraphics[scale=0.7]{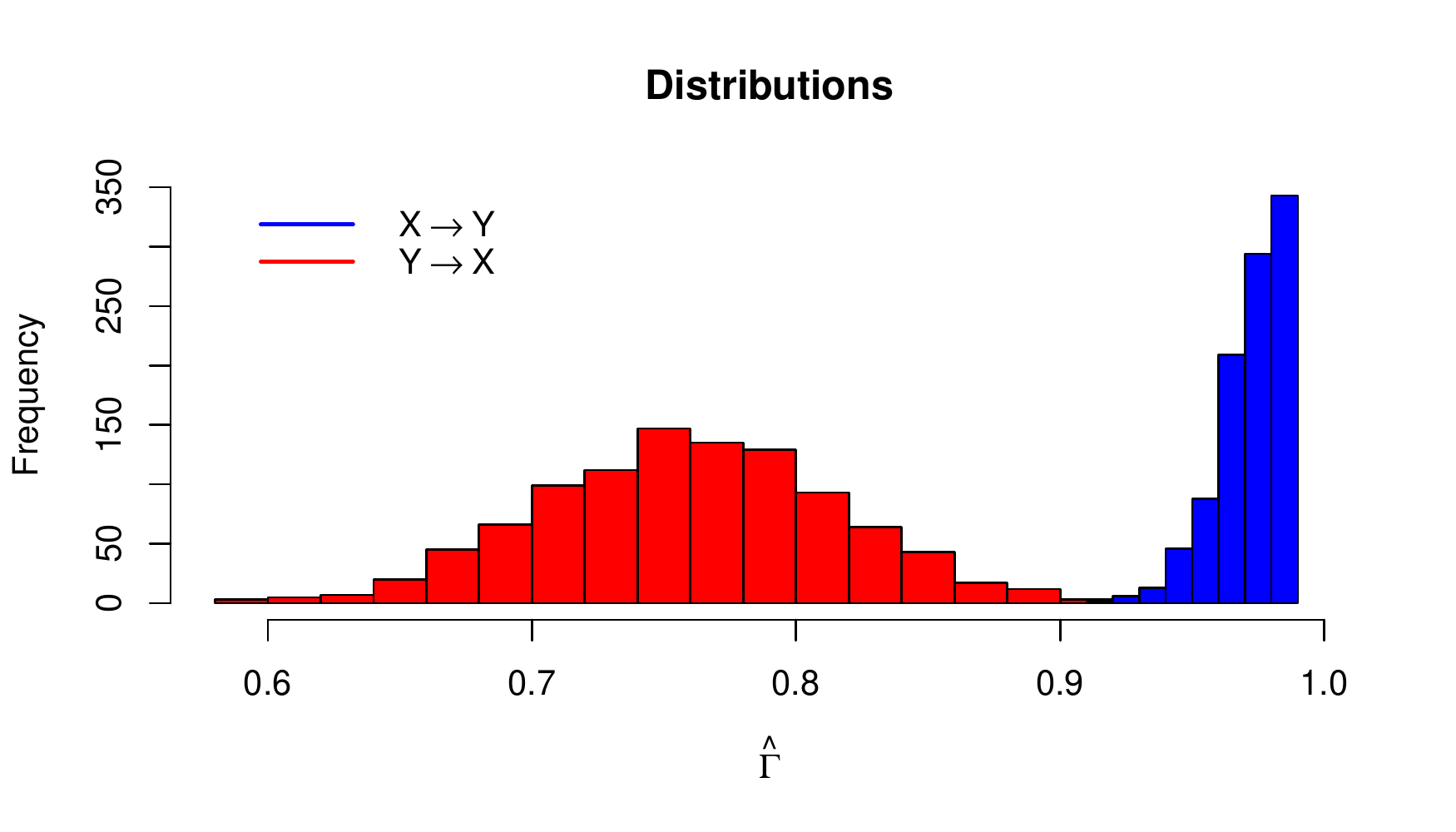}
\caption{The histograms provide an approximate visualization of the distributions of $\hat\Gamma_{\mathbf{X}\to \mathbf{Y}}^{time}(2)$ (in blue) and $\hat\Gamma_{\mathbf{Y}\to \mathbf{X}}^{time}(2)$ (in red) based on Model \ref{Time series 1} with a correct causal direction $\mathbf{X}\to \mathbf{Y}$ and a sample size $n=5000$. The threshold $k_n$ is chosen as $k_n=70=\floor{\sqrt{5000}}$.}
\label{histogram}
\end{figure}

We proceed by simulating $(\mathbf{X},\mathbf{Y})^\top$ according to Model \ref{Time series 1}. We explore three different values for the parameter $\delta$ ($0.1$, $0.5$, and $0.9$) and three distinct sample sizes ($n = 100$, $1000$, and $10000$). The random variables $\varepsilon_t^X$ and $\varepsilon_t^Y$ are generated from either the standard normal distribution (thus violating the RV assumption) or the standard Pareto distribution. For each value of $\delta$, noise distribution, and sample size, we calculate the estimators $\hat{\Gamma}^{time}_{\mathbf{X}\to \mathbf{Y}} := \hat{\Gamma}^{time}_{\mathbf{X}\to \mathbf{Y}}(2)$. This procedure is repeated $200$ times, and we compute the means and quantiles of these estimators. The summarized outcomes are presented in Table \ref{Table1}, where each cell corresponds to a specific model defined by $\delta$, noise distribution, and sample size.

\begin{table*}[]
\centering
\begin{tabular}{clll}
\multicolumn{4}{c}{\textit{\textbf{Errors with standard Pareto distributions}}}                                                                                                                                                                           \\ \hline
\multicolumn{1}{|c|}{}                              & \multicolumn{1}{c|}{$n=100$}                                     & \multicolumn{1}{c|}{$n=1000$}                                     & \multicolumn{1}{c|}{$n=10000$}                                     \\ \hline
\multicolumn{1}{|c|}{\multirow{2}{*}{$\delta=0.1$}} & \multicolumn{1}{l|}{$\hat{\Gamma}_{\mathbf{X}\to \mathbf{Y}}^{time}=0.80 \pm 0.11$} & \multicolumn{1}{l|}{$\hat{\Gamma}_{\mathbf{X}\to \mathbf{Y}}^{time}=0.94 \pm 0.04$}  & \multicolumn{1}{l|}{$\hat{\Gamma}_{\mathbf{X}\to \mathbf{Y}}^{time}=0.98 \pm 0.01$}   \\
\multicolumn{1}{|c|}{}                              & \multicolumn{1}{l|}{$\hat{\Gamma}_{\mathbf{Y}\to \mathbf{X}}^{time}=0.65 \pm 0.15$} & \multicolumn{1}{l|}{$\hat{\Gamma}_{\mathbf{Y}\to \mathbf{X}}^{time}=0.66 \pm 0.16$}  & \multicolumn{1}{l|}{$\hat{\Gamma}_{\mathbf{Y}\to \mathbf{X}}^{time}=0.65 \pm 0.12$}   \\ \hline
\multicolumn{1}{|c|}{\multirow{2}{*}{$\delta=0.5$}} & \multicolumn{1}{l|}{$\hat{\Gamma}_{\mathbf{X}\to \mathbf{Y}}^{time}=0.90\pm 0.05$}  & \multicolumn{1}{l|}{$\hat{\Gamma}_{\mathbf{X}\to \mathbf{Y}}^{time}=0.98 \pm 0.01$}  & \multicolumn{1}{l|}{$\hat{\Gamma}_{\mathbf{X}\to \mathbf{Y}}^{time}=0.994 \pm 0.00$} \\
\multicolumn{1}{|c|}{}                              & \multicolumn{1}{l|}{$\hat{\Gamma}_{\mathbf{Y}\to \mathbf{X}}^{time}=0.71 \pm 0.12$} & \multicolumn{1}{l|}{$\hat{\Gamma}_{\mathbf{Y}\to \mathbf{X}}^{time}=0.75 \pm 0.19$}  & \multicolumn{1}{l|}{$\hat{\Gamma}_{\mathbf{Y}\to \mathbf{X}}^{time}=0.79 \pm 0.11$}   \\ \hline
\multicolumn{1}{|c|}{\multirow{2}{*}{$\delta=0.9$}} & \multicolumn{1}{l|}{$\hat{\Gamma}_{\mathbf{X}\to \mathbf{Y}}^{time}=0.93 \pm 0.05$} & \multicolumn{1}{l|}{$\hat{\Gamma}_{\mathbf{X}\to \mathbf{Y}}^{time}=0.98 \pm 0.01$}  & \multicolumn{1}{l|}{$\hat{\Gamma}_{\mathbf{X}\to \mathbf{Y}}^{time}=0.996 \pm 0.00$} \\
\multicolumn{1}{|c|}{}                              & \multicolumn{1}{l|}{$\hat{\Gamma}_{\mathbf{Y}\to \mathbf{X}}^{time}=0.75 \pm 0.17$} & \multicolumn{1}{l|}{$\hat{\Gamma}_{\mathbf{Y}\to \mathbf{X}}^{time}=0.80 \pm 0.15$}   & \multicolumn{1}{l|}{$\hat{\Gamma}_{\mathbf{Y}\to \mathbf{X}}^{time}=0.84 \pm 0.10$}    \\ \hline
\multicolumn{4}{c}{\textit{\textbf{Errors with standard Gaussian distributions}}}                                                                                                                                                                         \\ \hline
\multicolumn{1}{|c|}{}                              & \multicolumn{1}{c|}{$n=100$}                                     & \multicolumn{1}{c|}{$n=1000$}                                     & \multicolumn{1}{c|}{$n=10000$}                                     \\ \hline
\multicolumn{1}{|c|}{\multirow{2}{*}{$\delta=0.1$}} & \multicolumn{1}{l|}{$\hat{\Gamma}_{\mathbf{X}\to \mathbf{Y}}^{time}=0.68 \pm 0.14$} & \multicolumn{1}{l|}{$\hat{\Gamma}_{\mathbf{X}\to \mathbf{Y}}^{time}=0.68\pm 0.10$}    & \multicolumn{1}{l|}{$\hat{\Gamma}_{\mathbf{X}\to \mathbf{Y}}^{time}=0.69 \pm 0.07$}   \\
\multicolumn{1}{|c|}{}                              & \multicolumn{1}{l|}{$\hat{\Gamma}_{\mathbf{Y}\to \mathbf{X}}^{time}=0.63 \pm 0.19$} & \multicolumn{1}{l|}{$\hat{\Gamma}_{\mathbf{Y}\to \mathbf{X}}^{time}=0.63\pm 0.13$}   & \multicolumn{1}{l|}{$\hat{\Gamma}_{\mathbf{Y}\to \mathbf{X}}^{time}=0.62\pm 0.08$}    \\ \hline
\multicolumn{1}{|c|}{\multirow{2}{*}{$\delta=0.5$}} & \multicolumn{1}{l|}{$\hat{\Gamma}_{\mathbf{X}\to \mathbf{Y}}^{time}=0.83 \pm 0.11$} & \multicolumn{1}{l|}{$\hat{\Gamma}_{\mathbf{X}\to \mathbf{Y}}^{time}=0.86\pm 0.06$}   & \multicolumn{1}{l|}{$\hat{\Gamma}_{\mathbf{X}\to \mathbf{Y}}^{time}=0.90\pm 0.03$}    \\
\multicolumn{1}{|c|}{}                              & \multicolumn{1}{l|}{$\hat{\Gamma}_{\mathbf{Y}\to \mathbf{X}}^{time}=0.64\pm 0.20$}   & \multicolumn{1}{l|}{$\hat{\Gamma}_{\mathbf{Y}\to \mathbf{X}}^{time}=0.65\pm 0.13$}   & \multicolumn{1}{l|}{$\hat{\Gamma}_{\mathbf{Y}\to \mathbf{X}}^{time}=0.66\pm 0.06$}    \\ \hline
\multicolumn{1}{|c|}{\multirow{2}{*}{$\delta=0.9$}} & \multicolumn{1}{l|}{$\hat{\Gamma}_{\mathbf{X}\to \mathbf{Y}}^{time}=0.88\pm 0.07$}  & \multicolumn{1}{l|}{$\hat{\Gamma}_{\mathbf{X}\to \mathbf{Y}}^{time}=0.93\pm 0.03$}   & \multicolumn{1}{l|}{$\hat{\Gamma}_{\mathbf{X}\to \mathbf{Y}}^{time}=0.96\pm 0.01$}    \\
\multicolumn{1}{|c|}{}                              & \multicolumn{1}{l|}{$\hat{\Gamma}_{\mathbf{Y}\to \mathbf{X}}^{time}=0.64\pm 0.19$}  & \multicolumn{1}{l|}{$\hat{\Gamma}_{\mathbf{Y}\to \mathbf{X}}^{time}=0.65\pm 0.13$} & \multicolumn{1}{l|}{$\hat{\Gamma}_{\mathbf{Y}\to \mathbf{X}}^{time}=0.66\pm 0.09$}    \\ \hline
\multicolumn{4}{c}{\textit{\textbf{$X$ with Pareto error, $Y$ with Gaussian error}}}                                                                                                                                                                      \\ \hline
\multicolumn{1}{|c|}{}                              & \multicolumn{1}{c|}{$n=100$}                                     & \multicolumn{1}{c|}{$n=1000$}                                     & \multicolumn{1}{c|}{$n=10000$}                                     \\ \hline
\multicolumn{1}{|c|}{\multirow{2}{*}{$\delta=0.5$}} & \multicolumn{1}{l|}{$\hat{\Gamma}_{\mathbf{X}\to \mathbf{Y}}^{time}=0.96\pm 0.02$}  & \multicolumn{1}{l|}{$\hat{\Gamma}_{\mathbf{X}\to \mathbf{Y}}^{time}=0.98\pm 0.01$} & \multicolumn{1}{l|}{$\hat{\Gamma}_{\mathbf{X}\to \mathbf{Y}}^{time}=0.997\pm 0.00$}  \\
\multicolumn{1}{|c|}{}                              & \multicolumn{1}{l|}{$\hat{\Gamma}_{\mathbf{Y}\to \mathbf{X}}^{time}=0.80\pm 0.11$}   & \multicolumn{1}{l|}{$\hat{\Gamma}_{\mathbf{Y}\to \mathbf{X}}^{time}=0.92\pm 0.04$}   & \multicolumn{1}{l|}{$\hat{\Gamma}_{\mathbf{Y}\to \mathbf{X}}^{time}=0.98\pm 0.01$}   \\ \hline
\multicolumn{4}{c}{\textit{\textbf{$X$ with Gaussian error, $Y$ with Pareto error}}}                                                                                                                                                                      \\ \hline
\multicolumn{1}{|c|}{}                              & \multicolumn{1}{c|}{$n=100$}                                     & \multicolumn{1}{c|}{$n=1000$}                                     & \multicolumn{1}{c|}{$n=10000$}                                     \\ \hline
\multicolumn{1}{|c|}{\multirow{2}{*}{$\delta=0.5$}} & \multicolumn{1}{l|}{$\hat{\Gamma}_{\mathbf{X}\to \mathbf{Y}}^{time}=0.65\pm 0.15$}  & \multicolumn{1}{l|}{$\hat{\Gamma}_{\mathbf{X}\to \mathbf{Y}}^{time}=0.67\pm 0.12$}    & \multicolumn{1}{l|}{$\hat{\Gamma}_{\mathbf{X}\to \mathbf{Y}}^{time}=0.68\pm 0.05$}    \\
\multicolumn{1}{|c|}{}                              & \multicolumn{1}{l|}{$\hat{\Gamma}_{\mathbf{Y}\to \mathbf{X}}^{time}=0.62\pm 0.20$}  & \multicolumn{1}{l|}{$\hat{\Gamma}_{\mathbf{Y}\to \mathbf{X}}^{time}=0.63\pm 0.13$}   & \multicolumn{1}{l|}{$\hat{\Gamma}_{\mathbf{Y}\to \mathbf{X}}^{time}=0.63\pm 0.08$}    \\ \hline
\end{tabular}
\caption{The results obtained from $200$ simulated time series following Model \ref{Time series 1} where $\mathbf{X}$ causes $\mathbf{Y}$. Each cell in the table corresponds to a distinct coefficient $\delta$, a specific number of data-points $n$, and a particular noise distribution. The value $\hat{\Gamma}_{\mathbf{X}\to \mathbf{Y}}^{time}=\cdot \pm \cdot$ represents the mean of all $200$ estimated coefficients $\hat{\Gamma}_{\mathbf{X}\to \mathbf{Y}}^{time}$, along with the difference between the $95\%$ empirical quantile and the mean, based on all $200$ simulations.  }
\label{Table1}
\end{table*}

\begin{Results}
The method's performance is unexpectedly robust even when the assumption of regular variation is violated. Although we lack theoretical justification for this observation, particularly with Gaussian noise where the true theoretical value is $\Gamma_{\mathbf{X}\to \mathbf{Y}}^{time}(p) = \Gamma_{\mathbf{Y}\to \mathbf{X}}^{time}(p) = 1$, empirical results suggest that $\hat\Gamma_{\mathbf{X}\to \mathbf{Y}}^{time}(p)$ converges faster than $\hat\Gamma_{\mathbf{Y}\to \mathbf{X}}^{time}(p)$. This sub-asymptotic behavior reveals an inherent asymmetry. Further investigation is required to determine whether this observation holds as a general result.

Conversely, when the tails of the noise variables differ between the cause and the effect, our method does not perform well. If $\varepsilon_t^X$ possesses heavier tails than $\varepsilon_t^Y$, for large $n$, both estimates tend to converge close to $1$. Alternatively, if $\varepsilon_t^X$ has lighter tails than $\varepsilon_t^Y$, both $\hat\Gamma_{\mathbf{X}\to \mathbf{Y}}^{time}$ and $\hat\Gamma_{\mathbf{Y}\to \mathbf{X}}^{time}$ tend to be very small, significantly deviating from $1$. Consequently, different tail behaviors of the noise variables appear to be the primary challenge. This issue is explored in greater detail in Subsection \ref{section hidden confounder}.
\end{Results}
\subsection{Choice of a threshold}
\label{section Threshold}

A common challenge in extreme value theory is the selection of an appropriate threshold. In our case, this corresponds to the choice of the parameter $k$ in Definition \ref{DEF}. There exists a trade-off between bias and variance; smaller values of $k$ lead to reduced bias (but increased variance), and larger values of $k$ result in higher bias (but lower variance). Unfortunately, there is no universally applicable method for threshold selection.

To provide insight into this behavior, consider Model \ref{Time series 1} with $n=1000$ and $\delta = 0.5$. Figure \ref{Threshold} illustrates the estimators $\hat\Gamma^{time}_{\mathbf{X}\to \mathbf{Y}}(2)$ and $\hat\Gamma^{time}_{\mathbf{Y}\to \mathbf{X}}(2)$ using various values of $k$.  The variance of $\hat\Gamma^{time}_{\mathbf{Y}\to \mathbf{X}}(2)$ for small $k$ is very large. Conversely, as $k$ increases, the bias of  $\hat\Gamma^{time}_{\mathbf{X}\to \mathbf{Y}}(2)$ will increase. 

Based on this example and several others, the choice $k=\sqrt{n}$ appears to be a reasonable option. It is crucial to emphasize that while this choice may be pragmatic, it might not necessarily be optimal.

By visually evaluating Figure \ref{Threshold}, it is possible to infer the values of $\hat\Gamma^{time}_{\mathbf{X}\to \mathbf{Y}}(2)$ and $\hat\Gamma^{time}_{\mathbf{Y}\to \mathbf{X}}(2)$. The observation seems to be that $\hat\Gamma^{time}_{\mathbf{X}\to \mathbf{Y}}(2)$ approaches one as $k$ decreases, while the same trend does not hold for $\hat\Gamma^{time}_{\mathbf{Y}\to \mathbf{X}}(2)$.

\begin{figure}[b]
\centering
\includegraphics[scale=0.43]{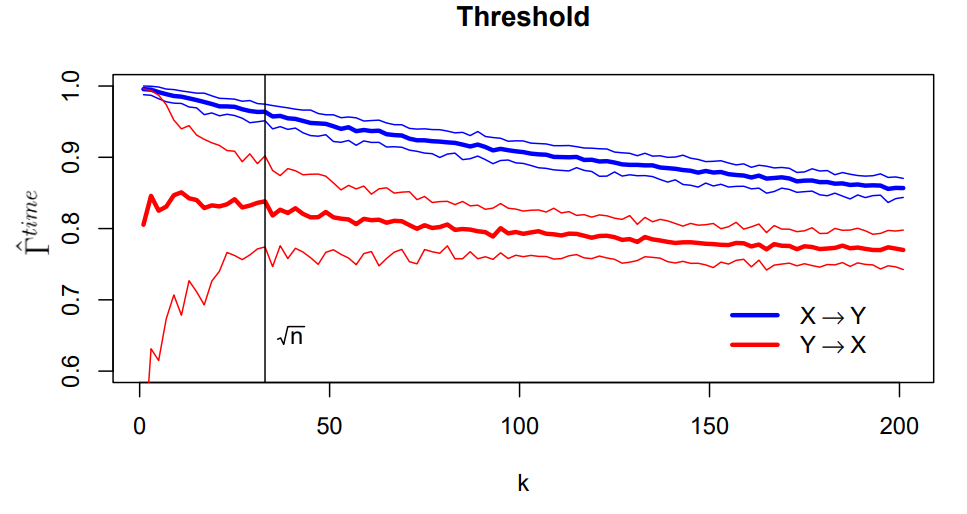}
\caption{The figure illustrates the behavior of the estimators $\hat\Gamma^{time}_{\mathbf{X}\to \mathbf{Y}}(2)$ (in blue) and $\hat\Gamma^{time}_{\mathbf{Y}\to \mathbf{X}}(2)$ (in red) under varying selections of the parameter $k$. The time series are generated based on Model \ref{Time series 1}, utilizing $n=1000$ and $\delta=0.5$. The thick line denotes the mean value across $100$ realizations, while the thin lines correspond to the $5\%$ and $95\%$ empirical pointwise quantiles. }
\label{Threshold}
\end{figure}

\subsection{Choice of the extremal delay}
\label{section Lag}

To provide an example of how $\Gamma^{time}_{\mathbf{X}\to \mathbf{Y}}(p)$ behaves for various selections of $p$, we consider the following model.

\begin{model}\label{Model2}
Let $(\mathbf{X},\mathbf{Y})^\top$ follow the $\hVAR(6,1)$ model 
\begin{align*}
X_t&=0.5X_{t-1}+ \varepsilon_t^X; \qquad
Y_t=0.5Y_{t-1}+ 0.5 X_{t-6}+\varepsilon_t^Y,
\end{align*}
where $\varepsilon_t^X, \varepsilon_t^Y\overset{\text{iid}}{\sim}{\rm Cauchy}$. 
\end{model}
Notice that the minimal delay is equal to $6$. Similar to the approach in Subsection \ref{section Threshold}, we simulate data from Model \ref{Model2} with a size of $n=1000$. Subsequently, we compute $\hat\Gamma^{time}_{\mathbf{X}\to \mathbf{Y}}(p)$ and $\hat\Gamma^{time}_{\mathbf{Y}\to \mathbf{X}}(p)$ for different values of $p$. This process is repeated 100 times. Figure \ref{Lag choice} shows the mean, $5\%$ and $95\%$ empirical quantiles of these estimates.

\begin{figure}[b]
\centering
\includegraphics[scale=0.47]{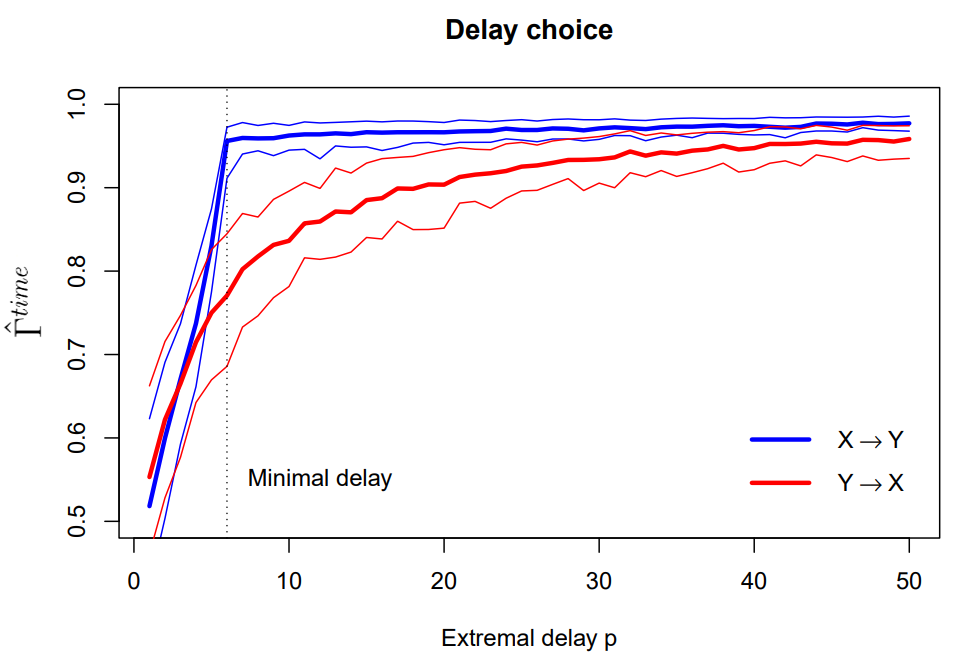}
\caption{The figure visualizes the behavior of the estimators $\hat\Gamma^{time}_{\mathbf{X}\to \mathbf{Y}}(p)$ (in blue) and $\hat\Gamma^{time}_{\mathbf{Y}\to \mathbf{X}}(p)$ (in red) for various selections of the extremal delay $p$. This is done using the  $\VAR(6)$ model (Model \ref{Model2}) with $n=1000$. The thick line represents the mean value across $100$ realizations, while the thin lines correspond to the $5\%$ and $95\%$ empirical quantiles of these estimates}
\label{Lag choice}
\end{figure}

The coefficient $\hat\Gamma^{time}_{\mathbf{X}\to \mathbf{Y}}(p)$ in the correct causal direction rises much faster than in the other direction until it reaches the \say{correct} minimal delay. Following this, the coefficient approaches $1$, remaining in close proximity even for larger $p$, aligning with theoretical expectations. Conversely, $\hat\Gamma^{time}_{\mathbf{Y}\to \mathbf{X}}(p)$ experiences a slower ascent, gradually converging towards $1$.

Nevertheless, practical scenarios often demand the selection of a specific $p$. In cases involving a small number of time series, like the application in Subsection \ref{Application}, it's feasible to calculate $\hat{\Gamma}^{time}_{\mathbf{X}\to\mathbf{Y}}(p)$ for several $p$ choices and create visualizations, as exemplified in Figure \ref{Space lag} further ahead. In instances involving a greater number of time series, such as the application in Subsection \ref{Hydrometeorology}, a pragmatic decision for $p$ should align with the maximum expected physical time delay within the system. The challenge of selecting a suitable time delay is not exclusive to the extreme value context and is encountered in non-extreme scenarios as well \citep{LagSelection,Runge}.

\subsection{Hidden confounder and different tail behavior of the noise}
\label{section hidden confounder}

The aforementioned examples have exclusively addressed scenarios where $\mathbf{X}$ causes $\mathbf{Y}$. However, it is important to investigate how the methodology fares when the correlation stems from a shared confounder. Furthermore, how does it respond when the tail indexes of the noise variables differ? To address these inquiries, we intend to analyze the following time series.

\begin{model}
\label{Timeseries3definicia}
Let $(\mathbf{X,Y,Z})^\top$ follow the $\VAR(3)$ model
\begin{align*}
Z_t&=0.5Z_{t-1}+\varepsilon_t^Z,\\
X_t&=0.5X_{t-1}+ 0.5Z_{t-2}+ \delta_YY_{t-3}+\varepsilon_t^X,\\
Y_t&=0.5Y_{t-1}+ 0.5Z_{t-1}+ \delta_XX_{t-3}+\varepsilon_t^Y,
\end{align*}
with independent noise variables $\varepsilon_t^X\sim {\rm t}_{\theta_{X}}, \varepsilon_t^Y\sim {\rm t}_{\theta_{Y}}, \varepsilon_t^Z\sim {\rm t}_{\theta_{Z}}$ (notation ${\rm t}_{\theta}$ corresponds to Student's $t$-distribution with $\theta$ degrees of freedom) and constants $\delta_X, \delta_Y\in\mathbb{R}$. 
\end{model}
The process $\mathbf{Z}$ represents an unobserved common cause, illustrating various scenarios: non-causal ($\delta_X=\delta_Y=0$), $\mathbf{X}$ causing $\mathbf{Y}$ ($\delta_X\neq 0$), $\mathbf{Y}$ causing $\mathbf{X}$ ($\delta_Y\neq 0$), and bidirectional causality ($\delta_X,\delta_Y\neq 0$, notation $\mathbf{X}\leftrightarrow \mathbf{Y}$). Additionally, the tail indexes $\theta_X$, $\theta_Y$, and $\theta_Z$ characterize the tail index of the time series. If $\theta_X$ is large, then the distribution of $\varepsilon_t^X$ is close to Gaussian. The case $\theta_X < \theta_Y$ indicates that $\varepsilon_t^X$ exhibits \textit{heavier} tails than $\varepsilon_t^Y$. 

Simulating time series according to Model \ref{Timeseries3definicia} with $n=1000$, $\delta_X, \delta_Y\in\{0,0.5, 1\}$, and various tail index combinations, we compute $\hat{\Gamma}^{time}_{\mathbf{X}\to \mathbf{Y}} := \hat{\Gamma}^{time}_{\mathbf{X}\to \mathbf{Y}}(3)$. This procedure is repeated 500 times to compute means and empirical quantiles of the estimators. The difference $\hat{\Gamma}^{time}_{\mathbf{X}\to \mathbf{Y}} - \hat{\Gamma}^{time}_{\mathbf{Y}\to \mathbf{X}}$ is also calculated alongside $95\%$ empirical quantiles. The outcomes are displayed in Table \ref{Table3}.

The results indicate that whenever either $\hat{\Gamma}^{time}_{\mathbf{X}\to \mathbf{Y}}$ or $\hat{\Gamma}^{time}_{\mathbf{Y}\to \mathbf{X}}$ is smaller than one, we can detect the correct causal direction. When $\mathbf{X}$ does not cause $\mathbf{Y}$ and vice versa, $\hat{\Gamma}^{time}_{\mathbf{X}\to \mathbf{Y}}$ and $\hat{\Gamma}^{time}_{\mathbf{Y}\to \mathbf{X}}$ are very similar and smaller than $1$.  When $\mathbf{X}$ causes $\mathbf{Y}$, $\hat{\Gamma}^{time}_{\mathbf{X}\to \mathbf{Y}}$ is generally much larger than $\hat{\Gamma}^{time}_{\mathbf{Y}\to \mathbf{X}}$.  When $\mathbf{X}\leftrightarrow \mathbf{Y}$, both $\hat{\Gamma}^{time}_{\mathbf{X}\to \mathbf{Y}}$ and $\hat{\Gamma}^{time}_{\mathbf{Y}\to \mathbf{X}}$ closely approach $1$.  However, if $\varepsilon_t^Z$ exhibits heavier tails than $\varepsilon_t^X$ and $\varepsilon_t^Y$, both $\hat{\Gamma}^{time}_{\mathbf{X}\to \mathbf{Y}}$ and $\hat{\Gamma}^{time}_{\mathbf{Y}\to \mathbf{X}}$ are in close proximity to $1$, regardless of the true causal relationship between $\mathbf{X}$ and $\mathbf{Y}$. Consequently, if  $\hat{\Gamma}^{time}_{\mathbf{X}\to \mathbf{Y}}$ and $\hat{\Gamma}^{time}_{\mathbf{Y}\to \mathbf{X}}$ are close to $1$, we can not distinguish whether this is due  $\mathbf{X}\leftrightarrow \mathbf{Y}$, or due to the presence of a hidden confounder with heavier tails. 

An interesting scenario arises when $\theta_X\neq\theta_Y$. Theory suggests that the method should work well as long as $\theta_X\geq\theta_Y$ (heavier tails of $\varepsilon_t^Y$) and should no longer work if  $\theta_X<\theta_Y$ (a supporting argument for this claim can be found in \hyperref[Claim]{Appendix \ref{AppendixB}},  Claim~\ref{Claim}). However, it's important to note that these results pertain to a population (limiting) context. For small sample sizes, the situation where $\theta_X<\theta_Y$ appears to \say{work well}, as evident from the blue values in Table~\ref{Table3} and Table~\ref{Table2}. There is an evident asymmetry between the cause and the effect, but it vanishes for large $n$. Both cases ($\theta_X<\theta_Y$ and  $\theta_X>\theta_Y$) can create problems for the estimation part, and further investigation is needed to understand better the method's behavior under different tail indices.

To summarize, the most challenging situation occurs when a heavier-tailed noise variable acts as a common cause. This situation necessitates careful consideration, particularly when both $\hat{\Gamma}^{time}_{\mathbf{X}\to \mathbf{Y}},\hat{\Gamma}^{time}_{\mathbf{Y}\to \mathbf{X}}\approx 1$. This ambiguity can be attributed to either true bidirectional causality or the influence of a hidden confounder.

\begin{table}[]
\centering
\begin{tabular}{cccc}
\multicolumn{4}{c}{\cellcolor[HTML]{EFEFEF}Confounded case with no causality between $\mathbf{X}$, $\mathbf{Y}$ 
  (case $\delta_X=0, \delta_Y=0$)}                                                                                                                                                                                           \\ \hline
\multicolumn{1}{|c|}{}                                                            & \multicolumn{1}{c|}{$\theta_X = \theta_Y =\theta_Z= 1$} & \multicolumn{1}{c|}{$\theta_X = 1, \theta_Y =1, \theta_Z= 2$} & \multicolumn{1}{c|}{$\theta_X = 2, \theta_Y =2, \theta_Z= 1$} \\ \hline
\multicolumn{1}{|c|}{$\hat{\Gamma}_{\mathbf{X}\to \mathbf{Y}}^{time}$}                              & \multicolumn{1}{c|}{$0.79\pm 0.09$}                     & \multicolumn{1}{c|}{$0.65\pm 0.07$}                           & \multicolumn{1}{c|}{$0.94\pm 0.03$}                           \\ \hline
\multicolumn{1}{|c|}{$\hat{\Gamma}_{\mathbf{Y}\to \mathbf{X}}^{time}$}                              & \multicolumn{1}{c|}{$0.81\pm 0.08$}                     & \multicolumn{1}{c|}{$0.65\pm 0.07$}                           & \multicolumn{1}{c|}{$0.95\pm 0.03$}                           \\ \hline
\multicolumn{1}{|c|}{$\hat{\Gamma}_{\mathbf{X}\to \mathbf{Y}}^{time}-\hat{\Gamma}_{\mathbf{Y}\to \mathbf{X}}^{time}$} & \multicolumn{1}{c|}{$-0.02 (-0.11, 0.7)$}              & \multicolumn{1}{c|}{$-0.01 (-0.10, 0.10)$}                    & \multicolumn{1}{c|}{$-0.02 (-0.5, 0.3)$}                     \\ \hline
\multicolumn{4}{c}{\cellcolor[HTML]{EFEFEF}Confounded case with $\mathbf{X}\to \mathbf{Y}$ (case $\delta_X=1, \delta_Y=0$)}                                                                                                                                                                                               \\ \hline
\multicolumn{1}{|c|}{}                                                            & \multicolumn{1}{c|}{$\theta_X = \theta_Y =\theta_Z= 1$} & \multicolumn{1}{c|}{$\theta_X = 1, \theta_Y =1, \theta_Z= 2$} & \multicolumn{1}{c|}{$\theta_X = 2, \theta_Y =2, \theta_Z= 1$} \\ \hline
\multicolumn{1}{|c|}{$\hat{\Gamma}_{\mathbf{X}\to \mathbf{Y}}^{time}$}                              & \multicolumn{1}{c|}{$0.97\pm 0.01$}                     & \multicolumn{1}{c|}{$0.97\pm 0.01$}                           & \multicolumn{1}{c|}{$0.97\pm 0.01$}                           \\ \hline
\multicolumn{1}{|c|}{$\hat{\Gamma}_{\mathbf{Y}\to \mathbf{X}}^{time}$}                              & \multicolumn{1}{c|}{$0.82\pm 0.07$}                     & \multicolumn{1}{c|}{$0.76\pm 0.7$}                           & \multicolumn{1}{c|}{$0.91\pm 0.05$}                           \\ \hline
\multicolumn{1}{|c|}{$\hat{\Gamma}_{\mathbf{X}\to \mathbf{Y}}^{time}-\hat{\Gamma}_{\mathbf{Y}\to \mathbf{X}}^{time}$} & \multicolumn{1}{c|}{$0.13 (0.05, 0.25)$}               & \multicolumn{1}{c|}{$0.20 (0.13, 0.26)$}                      & \multicolumn{1}{c|}{$0.06 (0.00, 0.13)$}                      \\ \hline
\multicolumn{4}{c}{\cellcolor[HTML]{EFEFEF}Confounded case with $\mathbf{X}\to \mathbf{Y}$ (case $\delta_X=1, \delta_Y=0$)}                                                                                                                                                                                               \\ \hline
\multicolumn{1}{|c|}{}                                                            & \multicolumn{1}{c|}{$\theta_X = \theta_Y =\theta_Z=9$}  & \multicolumn{1}{c|}{$\theta_X = 1, \theta_Y =2, \theta_Z= 3$} & \multicolumn{1}{c|}{$\theta_X = 2, \theta_Y =1, \theta_Z= 3$} \\ \hline
\multicolumn{1}{|c|}{$\hat{\Gamma}_{\mathbf{X}\to \mathbf{Y}}^{time}$}                              & \multicolumn{1}{c|}{$0.93\pm 0.02$}                     & \multicolumn{1}{c|}{\textcolor{blue}{$0.98\pm 0.01$}}                           & \multicolumn{1}{c|}{\textcolor{blue}{$0.87\pm 0.05$}}                           \\ \hline
\multicolumn{1}{|c|}{$\hat{\Gamma}_{\mathbf{Y}\to \mathbf{X}}^{time}$}                              & \multicolumn{1}{c|}{$0.78\pm 0.05$}                     & \multicolumn{1}{c|}{\textcolor{blue}{$0.79\pm 0.10$}}                           & \multicolumn{1}{c|}{\textcolor{blue}{$0.68\pm 0.05$}}                           \\ \hline
\multicolumn{1}{|c|}{$\hat{\Gamma}_{\mathbf{X}\to \mathbf{Y}}^{time}-\hat{\Gamma}_{\mathbf{Y}\to \mathbf{X}}^{time}$} & \multicolumn{1}{c|}{$0.10 (0.01, 0.20)$}                & \multicolumn{1}{c|}{$0.13 (0.00, 0.27)$}                      & \multicolumn{1}{c|}{$0.16 (-0.02, 0.32)$}                     \\ \hline
\multicolumn{4}{c}{\cellcolor[HTML]{EFEFEF}Confounded loop case where $\mathbf{X}\leftrightarrow \mathbf{Y}$ (case $\delta_X=0.5, \delta_Y=0.5$)}                                                                                                                                 \\ \hline
\multicolumn{1}{|c|}{}                                                            & \multicolumn{1}{c|}{$\theta_X = \theta_Y =\theta_Z=1$}  & \multicolumn{1}{c|}{$\theta_X = 1, \theta_Y =2, \theta_Z= 3$} & \multicolumn{1}{c|}{$\theta_X = 3, \theta_Y =2, \theta_Z= 1$} \\ \hline
\multicolumn{1}{|c|}{$\hat{\Gamma}_{\mathbf{X}\to \mathbf{Y}}^{time}$}                              & \multicolumn{1}{c|}{$0.97\pm 0.01$}                     & \multicolumn{1}{c|}{$0.98\pm 0.01$}                           & \multicolumn{1}{c|}{$0.98\pm 0.01$}                           \\ \hline
\multicolumn{1}{|c|}{$\hat{\Gamma}_{\mathbf{Y}\to \mathbf{X}}^{time}$}                              & \multicolumn{1}{c|}{$0.97\pm 0.01$}                     & \multicolumn{1}{c|}{$0.98\pm 0.01$}                           & \multicolumn{1}{c|}{$0.98\pm 0.01$}                           \\ \hline
\multicolumn{1}{|c|}{$\hat{\Gamma}_{\mathbf{X}\to \mathbf{Y}}^{time}-\hat{\Gamma}_{\mathbf{Y}\to \mathbf{X}}^{time}$} & \multicolumn{1}{c|}{$0.00 (-0.01, 0.01)$}               & \multicolumn{1}{c|}{$0.00 (-0.01, 0.01)$}                     & \multicolumn{1}{c|}{$0.00 (-0.01, 0.01)$}                     \\ \hline
\end{tabular}
\caption{The estimation of $\hat{\Gamma}^{time}(3)$ is performed on time series generated from Model \ref{Timeseries3definicia}. The tail indexes $\theta_X$, $\theta_Y$, and $\theta_Z$ correspond to the tail index of the respective time series, where a large $\theta_X$ signifies lighter tails and a small $\theta_X$ denotes heavy tails in $\varepsilon_X$. Each value of $\hat{\Gamma}^{time}_{\mathbf{X}\to \mathbf{Y}}=\cdot \pm \cdot$ represents the mean of the estimated coefficients $\hat{\Gamma}_{\mathbf{X}\to \mathbf{Y}}^{time}$, along with the difference between this mean and the $95\%$ empirical quantile out of all $500$ simulations.}
\label{Table3}
\end{table}

\subsection{Testing}
\label{SectionTesting}
One approach to establish a formal methodology for detecting the causal direction between two time series involves the utilization of a threshold. Specifically, we say that $\mathbf{X}$ causes $\mathbf{Y}$ if and only if  $\hat\Gamma^{time}_{\mathbf{X}\to \mathbf{Y}}(p)\geq \tau$, where $\tau=0.9$ or $0.95$. Our goal is to evaluate the hypothesis $H_0: \Gamma^{time}_{\mathbf{X}\to \mathbf{Y}}(p) < 1$ against the alternative $H_A: \Gamma^{time}_{\mathbf{X}\to \mathbf{Y}}(p) = 1$. An understanding of the distribution of $\hat\Gamma^{time}_{\mathbf{X}\to \mathbf{Y}}(p)$ is imperative. However, computing this distribution lies beyond the purview of our current study. We considered estimating confidence intervals using the block bootstrap technique (Section 5.4 in \cite{bootstrap} and the accompanying code in the supplementary package). However, we opted not to incorporate this methodology within the confines of this paper, as its performance in our specific scenario proved to be suboptimal. Future research can reveal a better-working methodology. Additionally, visual aids such as Figures \ref{Threshold} and \ref{Lag choice} can serve as valuable tools for enhancing understanding. Subsequently, we will delve into the scenario wherein we infer $\mathbf{X}\to \mathbf{Y}$ if and only if the value of $\hat\Gamma^{time}_{\mathbf{X}\to \mathbf{Y}}(p)$ meets or exceeds the designated threshold $\tau$.

In the following, we provide an insight into the process of selecting the appropriate threshold value $\tau$. This choice  should depend on the sample size $n$ and the extremal delay $p$. Opting for a small $\tau$ runs the risk of yielding erroneous conclusions, while selecting a large $\tau$ can curtail the statistical power. Analogous methodologies are commonly employed, particularly within the realm of extreme value theory. For instance, in contexts such as distinguishing between extremal dependence and independence \citep{asymptotic_dependence}. To develop an intuitive understanding, we will focus on Model  \ref{Time series 1} with $\delta=0.5$ and $\varepsilon_t^X, \varepsilon_t^Y\overset{\text{iid}}{\sim}{\rm Pareto}(1,1)$.  

We will observe two objectives as a function of $\tau$. The first is the percentage of correctly concluding that $\mathbf{X}\to \mathbf{Y}$ and the other is the percentage of correctly concluding that $\mathbf{Y}\not\to \mathbf{X}$ (that is, concluding that $\mathbf{Y}$ does not cause $\mathbf{X}$). 

Figures \ref{tauP} and \ref{tauN} illustrate the sensitivity of the results to the choice of the threshold $\tau=\tau(n,p)$ across a range of $(n,p)$ values. It is reasonable to expect that a larger value of $p$ should correspond to a larger value of $\tau$. To ensure that $\hat\Gamma^{time}_{\mathbf{X}\to \mathbf{Y}}(p)< \tau$ occurs in less than $5\%$ of cases, simulations suggest values of $\tau(500, 6)=0.90$ and $\tau(500, 11)=0.93$. For larger values of $p>11$, $\hat\Gamma^{time}_{\mathbf{X}\to \mathbf{Y}}(p)$ and $\hat\Gamma^{time}_{\mathbf{Y}\to \mathbf{X}}(p)$ exhibit substantial similarity.

It appears that as $n$ increases, the accuracy of $\hat\Gamma^{time}_{\mathbf{Y}\to \mathbf{X}}(p)$ improves. The $5\%$ significance level seems consistent across all $n$ values, with the main enhancement lying in the power to detect causality.

In conclusion, the appropriate selection of $\tau$ exhibits minimal variation with changes in $n$ and $p$. In all cases, opting for $\tau \approx 0.9$ appears reasonable, unless $p$ becomes exceedingly large.

\begin{figure}[]
\centering
\includegraphics[scale=0.5]{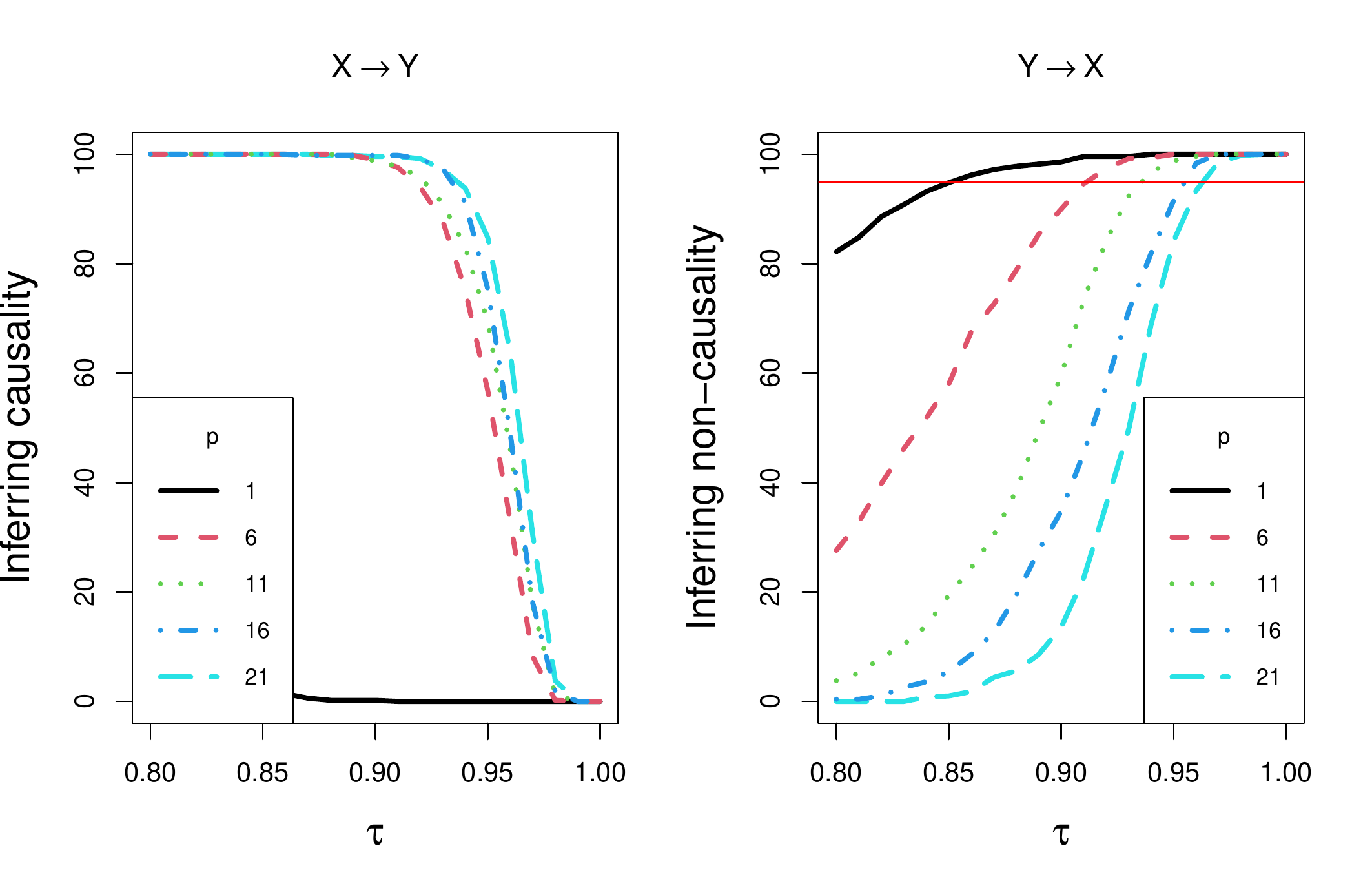}
\caption{The figures depict the performance of our methodology for various $\tau(500,p)$ selections in Model \ref{Timeseries3definicia}. In the left figure, we illustrate the percentage of instances where we accurately infer the relation $\mathbf{X}\to \mathbf{Y}$. The right figure represents the percentage of cases in which we correctly identify the absence of the relation $\mathbf{Y}\not\to \mathbf{X}$. If $\tau$ is large, we rarely find the correct relation $\mathbf{X}\to \mathbf{Y}$, but we almost always correctly find $\mathbf{Y}\not\to \mathbf{X}$. If $\tau$ is small, we almost always find the correct relation $\mathbf{X}\to \mathbf{Y}$, but we almost never correctly find $\mathbf{Y}\not\to \mathbf{X}$. A red horizontal line marks the $95\%$ success rate. Note that the scenario where $p=1$  signifies an erroneous selection of the extremal delay—specifically, a situation where $p$ is smaller than the minimum delay. }
\label{tauP}
\end{figure}

\begin{figure}[]
\centering
\includegraphics[scale=0.5]{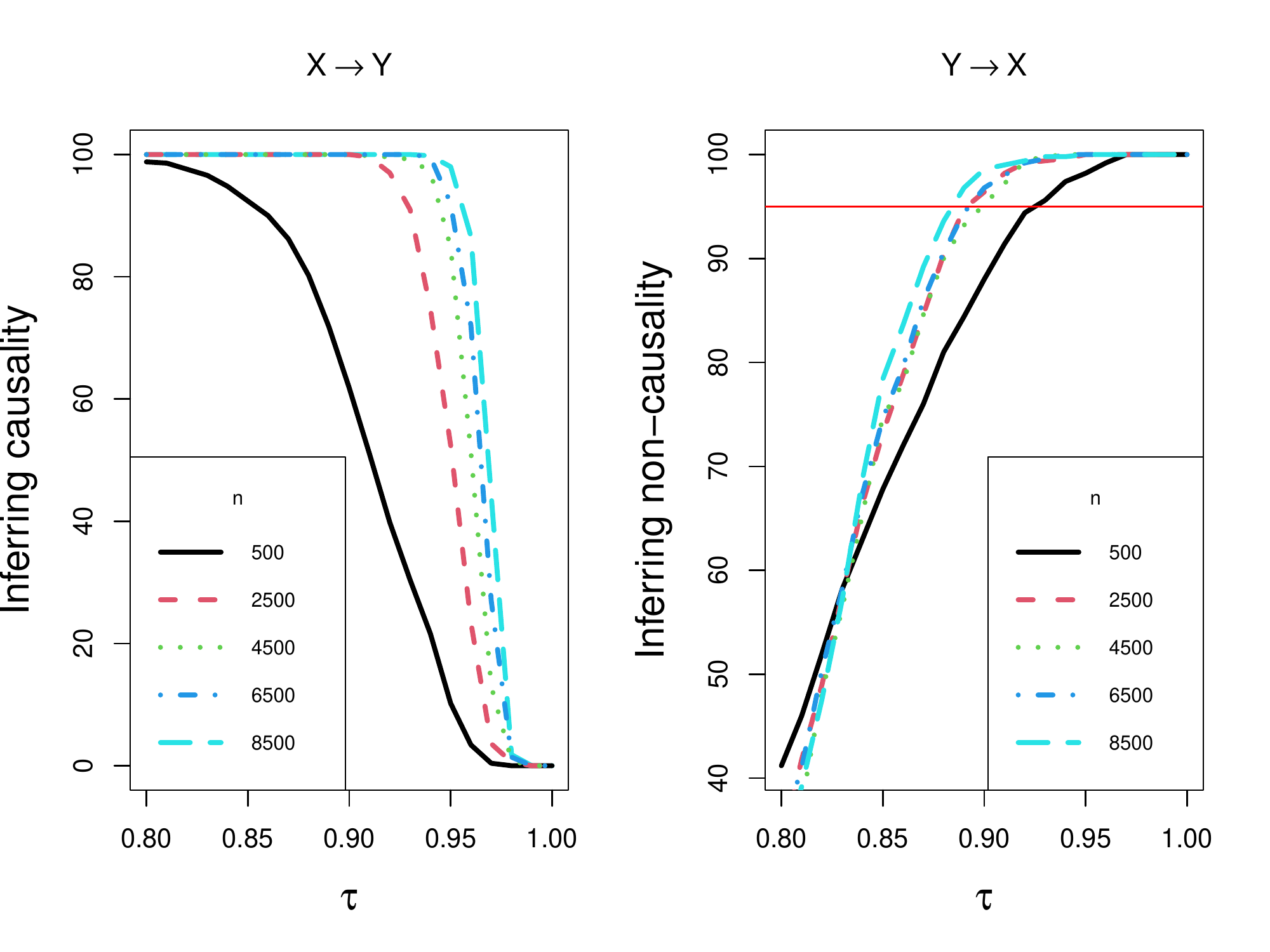}
\caption{Choice of the threshold as a function of $n$ and fixed $p=6$. The left figure represents the percentage of cases when we correctly inferred  $\mathbf{X}\to \mathbf{Y}$. The right figure represents the percentage of cases when we correctly inferred  $\mathbf{Y}\not\to \mathbf{X}$. A red horizontal line marks the $95\%$ success rate.}
\label{tauN}
\end{figure}

\begin{table}[]
\centering
\resizebox{\textwidth}{!}{%
\begin{tabular}{|c|cccccccc|}
\hline
 &
  \multicolumn{8}{c|}{{\ul \textbf{VAR(2) model}}} \\
 &
  \multicolumn{2}{c|}{$\hat\Gamma^{time}_{\mathbf{X}\to \mathbf{Y}}\geq \tau=0.9$} &
  \multicolumn{2}{c|}{Granger test} &
  \multicolumn{2}{c|}{PCMCI} &
  \multicolumn{2}{c|}{LPCMCI} \\
 &
  $n=500$ &
  \multicolumn{1}{c|}{$n=5000$} &
  $n=500$ &
  \multicolumn{1}{c|}{$n=5000$} &
  $n=500$ &
  \multicolumn{1}{c|}{$n=5000$} &
  $n=500$ &
  $n=5000$ \\ \hline
$\mathbf{X}\to \mathbf{Y}$ &
  95\% &
  \multicolumn{1}{c|}{100\%} &
  96\% &
  \multicolumn{1}{c|}{97\%} &
  95\% &
  \multicolumn{1}{c|}{100\%} &
  100\% &
  100\% \\
$\mathbf{Y}\to \mathbf{X}$ &
  99\% &
  \multicolumn{1}{c|}{100\%} &
  97\% &
  \multicolumn{1}{c|}{98\%} &
  80\% &
  \multicolumn{1}{c|}{70\%} &
  97\% &
  97\% \\ \hline
 &
  \multicolumn{8}{c|}{{\ul \textbf{NAAR(3) model with hidden confounder}}} \\
 &
  \multicolumn{2}{c|}{$\hat\Gamma^{time}_{\mathbf{X}\to \mathbf{Y}}\geq \tau=0.9$} &
  \multicolumn{2}{c|}{Granger test} &
  \multicolumn{2}{c|}{PCMCI} &
  \multicolumn{2}{c|}{LPCMCI} \\
 &
  $n=500$ &
  \multicolumn{1}{c|}{$n=5000$} &
  $n=500$ &
  \multicolumn{1}{c|}{$n=5000$} &
  $n=500$ &
  \multicolumn{1}{c|}{$n=5000$} &
  $n=500$ &
  $n=5000$ \\ \hline
$\mathbf{X}\to \mathbf{Y}$ &
  97\% &
  \multicolumn{1}{c|}{100\%} &
  74\% &
  \multicolumn{1}{c|}{92\%} &
  88\% &
  \multicolumn{1}{c|}{100\%} &
  94\% &
  100\% \\
$\mathbf{Y}\to \mathbf{X}$ &
  89\% &
  \multicolumn{1}{c|}{100\%} &
  56\% &
  \multicolumn{1}{c|}{43\%} &
  5\% &
  \multicolumn{1}{c|}{0\%} &
  5\% &
  0\% \\ \hline
\end{tabular}%
}
\caption{A comparison of four methods for the causal inference on two time series models, a simple $\VAR(2)$ model and a more complex nonlinear model with a hidden common cause. The percentage obtained after 100 repetitions indicates how many times the estimation was correct.}
\label{Table2}
\end{table}

\subsection{Comparison with state-of-the-art methods}
We conduct a comparative analysis of our method with three classical techniques found in the literature. The initial approach is the classical Granger test utilizing default parameters and a significance level of $\alpha=0.05$. The second and third approaches encompass the PCMCI and LPCMCI methods, both employing default parameters and featuring a robust partial correlation independence test. For comprehensive insights into their implementations, please refer to the supplementary materials. Our method corresponds to estimating $\hat\Gamma^{time}_{\mathbf{X}\to \mathbf{Y}}(3)$ and concluding that $\mathbf{X}$ causes $\mathbf{Y}$ if and only if $\hat\Gamma^{time}_{\mathbf{X}\to \mathbf{Y}}(3)\geq \tau$ with the choice $\tau=0.9$. 

We examine two distinct models. The first is the VAR Model \ref{Time series 1}, characterized by $\delta=0.5$ and noise following a ${\rm Pareto}(1,1)$ distribution. Notably, this noise exhibits an infinite expected value, which can potentially pose challenges for classical methodologies. The second is more complex NAAR model with a hidden confounder, as defined in the subsequent description.  
\begin{model}
\label{Timeseries4definicia}
Let $(\mathbf{X,Y,Z})^\top$ follow the $\hNAAR(3,1)$ model
\begin{align*}
Z_t&=0.5Z_{t-1}+\varepsilon_t^Z,\\
X_t&=0.5X_{t-1}+ 0.5Z_{t-2} +\varepsilon_t^X,\\
Y_t&=0.5Y_{t-1}+ 0.5Z_{t-1}+ f_X(X_{t-3})+\varepsilon_t^Y,
\end{align*}
with independent noise variables $\varepsilon_t^X, \varepsilon_t^Y, 2\varepsilon_t^Z\sim {\rm Pareto}(1,1)$ and $f_X(x) = x^{\frac{3}{4}}\1(x>50)$. 
\end{model}

The function $f_X$ in Model \ref{Timeseries4definicia} represents a case when the causality presents itself only in the tail, not in the body of the distribution. 

We conduct simulations on the aforementioned time series, considering two data set sizes: $n=500$ and $n=5000$. Subsequently, employing the four methodologies, we compute the count of simulations (out of 100) where the causal relationships $\mathbf{X}\to \mathbf{Y}$ and $\mathbf{Y}\not\to \mathbf{X}$ are accurately inferred. The outcome percentages are presented in Table \ref{Table2}.

The results suggest that state-of-the-art methods behave well within the context of the simple model. Although these methods inherently assume finite expected values, it appears that the ${\rm Pareto}(1,1)$ noise distribution does not introduce significant challenges. Our method demonstrates satisfactory performance but exhibits slightly reduced power when compared with other approaches.

On the other hand, state-of-the-art methods encounter difficulties when confronted with the more complex model. PCMCI and LPCMCI frequently miscalculate the directionality, erroneously estimating that $\mathbf{Y}$ causes $\mathbf{X}$. This misjudgment arises from the presence of a hidden common cause $\mathbf{Z}$, which creates a spurious effect that even LPCMCI can not handle well. In contrast, our method successfully discerns genuine causality from causal connections attributed to a shared cause.
To conclude, the outcomes presented in Table \ref{Table2} imply that our method is particularly well-suited for large sample sizes, where it outperforms state-of-the-art techniques.

\section{Applications}
\label{APPLICATIONS}
We apply our methodology to two real datasets. The first one pertains to geomagnetic storms in the field of space weather science, while the second one is related to earth science and explores causal connections between hydro-meteorological phenomena. All data and a detailed \textsf{R} code are available in the supplementary package. 
\subsection{Space weather}
\label{Application}
In the subsequent sections, we address a problem within the realm of space weather studies. The term \say{space weather} refers to the variable conditions on the Sun and in space that can influence the performance of technology we use on Earth. Extreme space weather could potentially cause damage to critical infrastructures -- especially the electric grid. In order to protect people and systems that might be at risk from space weather effects, we need to understand the causes of space weather\footnote{Text taken from a webpage https://www.ready.gov/space-weather, accessed 18.5.2021.}.

\begin{figure*}[!]
\centering
\includegraphics[scale=0.65]{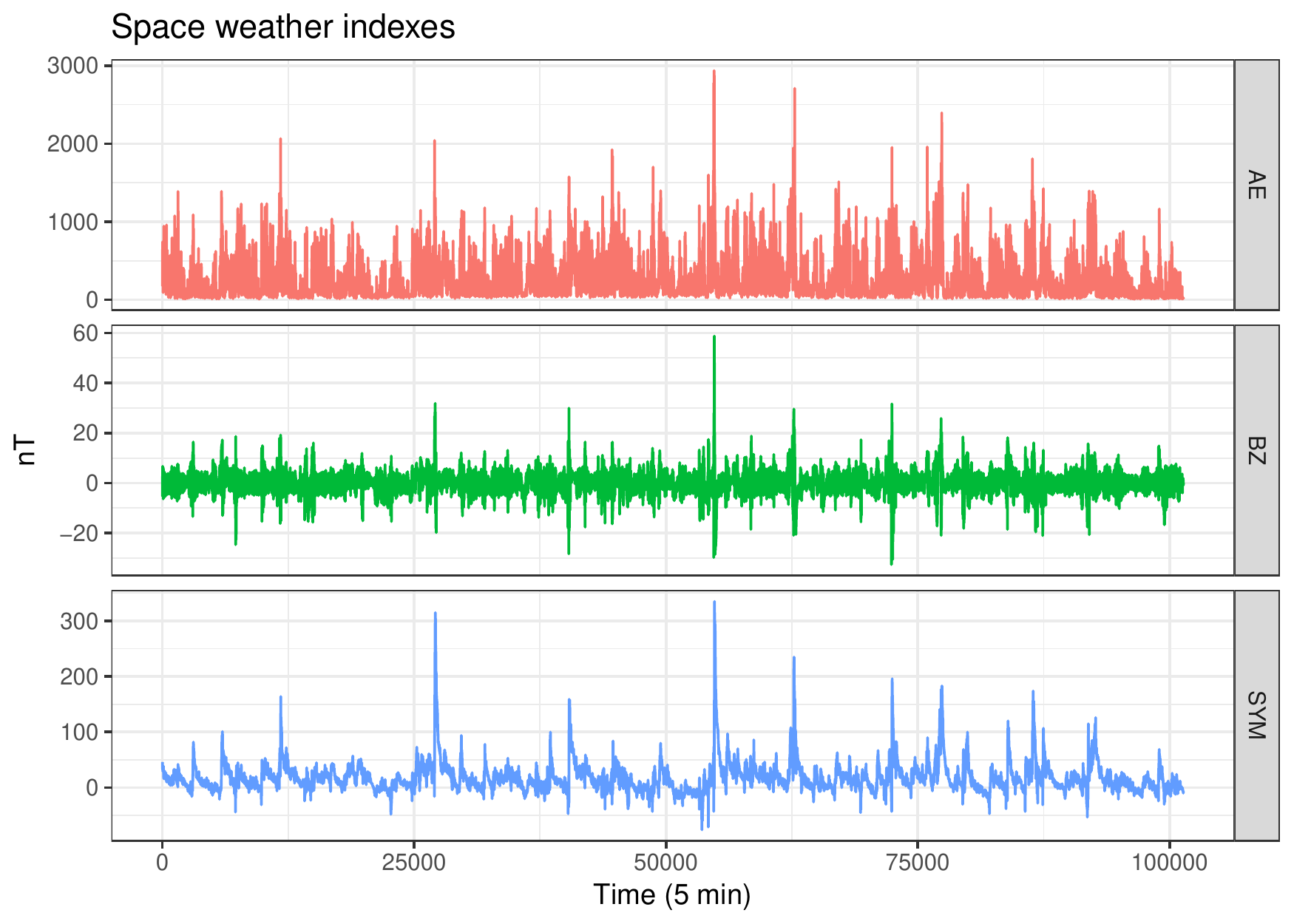}
\caption{Space weather. The first plot represents AE (substorm index), the second one BZ (vertical component of an interplanetary magnetic field) and the last one SYM (magnetic storm index). Data were measured in 5-minute intervals for the year 2000 by NASA. The unit of measurement is nanotesla (nT).}
\label{Space figure}
\end{figure*}

Geomagnetic storms and substorms serve as indicators of geomagnetic activity. A substorm manifests as a sudden intensification and heightened motion of auroral arcs, potentially leading to magnetic field disturbances in the auroral zones of up to 1000 nT (Tesla units). The origin of this geomagnetic activity lies within the Sun itself. More precisely, a substantial correlation exists between this geomagnetic activity, the solar wind (a flow of negatively charged particles emitted by the Sun), and the interplanetary magnetic field (which constitutes a portion of the solar magnetic field carried outward by the solar wind).

A fundamental challenge in this field revolves around the determination and prediction of specific characteristics: the magnetic storm index (SYM) and the substorm index (AE). It may seem that AE is a driving factor (cause) of SYM because the accumulation of successive substorms usually precedes the occurrence of magnetic storms. However, a recent article \citep{application} suggests otherwise. It proposes that a vertical component of the interplanetary magnetic field (BZ) serves as a shared trigger for both indices. Our approach involves applying our methodology to validate this finding and ascertain whether this causal influence becomes evident in extreme scenarios.

Our dataset comprises three distinct time series (SYM, AE, BZ), encompassing approximately $100\,000$ measurements, each captured at 5-minute intervals over the entire year 2000. These data, in addition to being available in the supplementary package, can also be accessed online through the NASA website\footnote{NASA webpage https://cdaweb.gsfc.nasa.gov, accessed 18.5.2021.}. A visual representation of the data is presented in Figure \ref{Space figure}.  From the nature of the data, we focus on analyzing extremes characterized by  extremely small SYM, extremely large AE, and extremely small BZ  (that is, we consider $-$SYM, $+$AE, $-$BZ and compare maxima). Given the inherent characteristics of the data, an appropriate delay will be smaller than $p=24$ ($2$ hours).  

First, we discuss whether the assumptions for our method are fulfilled. We estimate the tail indexes of our data\footnote{We employ the \textsf{HTailIndex} function from the \textsf{ExtremeRisks} package in \textsf{R} with variable $k=500$ \citep{Extreme_risks_package}. The confidence intervals are computed using the normal approximation from Theorem 3.2.5 in \cite{deHaan} and page 1288 in \cite{Drees}. Details can be found in  the corresponding \textsf{R} documentation.}. Results are the following: SYM has the estimated tail index $0.25 \,\,(0.015, 0.5)$,   AE has $0.18 \,\,(0.08, 0.28)$ and BZ has  $0.30 \,\,(0.12, 0.46)$.  
Therefore, the assumption of regular variation characterized by the same tail index appears reasonable. Furthermore, none of the confidence intervals encompass the zero value, although SYM comes quite close. This observation suggests that our time series can be deemed as regularly varying.
We also require a stationarity assumption of our time series, which seems to hold. It is believed that our variables do not contain a seasonal pattern or a trend (at least not in a horizon of a few years). The Dickey--Fuller test also suggests stationarity, although it is recognized that tests for stationarity aren't very dependable \citep{Dickey}. 
In summation, all the assumptions seem to be reasonably satisfied.

Lastly, we proceed to calculate the causal tail coefficient across varying extreme delays $p$ and a range of considered extremes $k$. The numerical outcomes are depicted in Figure \ref{Space lag}. These findings impart the following insights:

\begin{itemize}
\item BZ and SYM have strong asymmetry ($\hat{\Gamma}_{BZ\to SYM}^{time}(p)\approx 1$, and $\hat{\Gamma}_{SYM\to  BZ}^{time}(p)\ll 1$, see blue lines), 
\item BZ and AE have an asymmetry ($\hat{\Gamma}_{BZ\to AE}^{time}(p)\approx 1$, and $\hat{\Gamma}_{SYM\to  AE}^{time}(p)<1$, see green lines) 
\item SYM and AE have no asymmetry  ($\hat{\Gamma}_{AE\to SYM}^{time}(p),\hat{\Gamma}_{SYM\to  AE}^{time}(p)<1$ and both $\hat{\Gamma}_{SYM\to  AE}^{time}(p)$, $\hat{\Gamma}_{AE\to SYM}^{time}(p)$ are very similar, see brown lines). 
\end{itemize}

\begin{figure}[]
\centering
\includegraphics[scale=0.284]{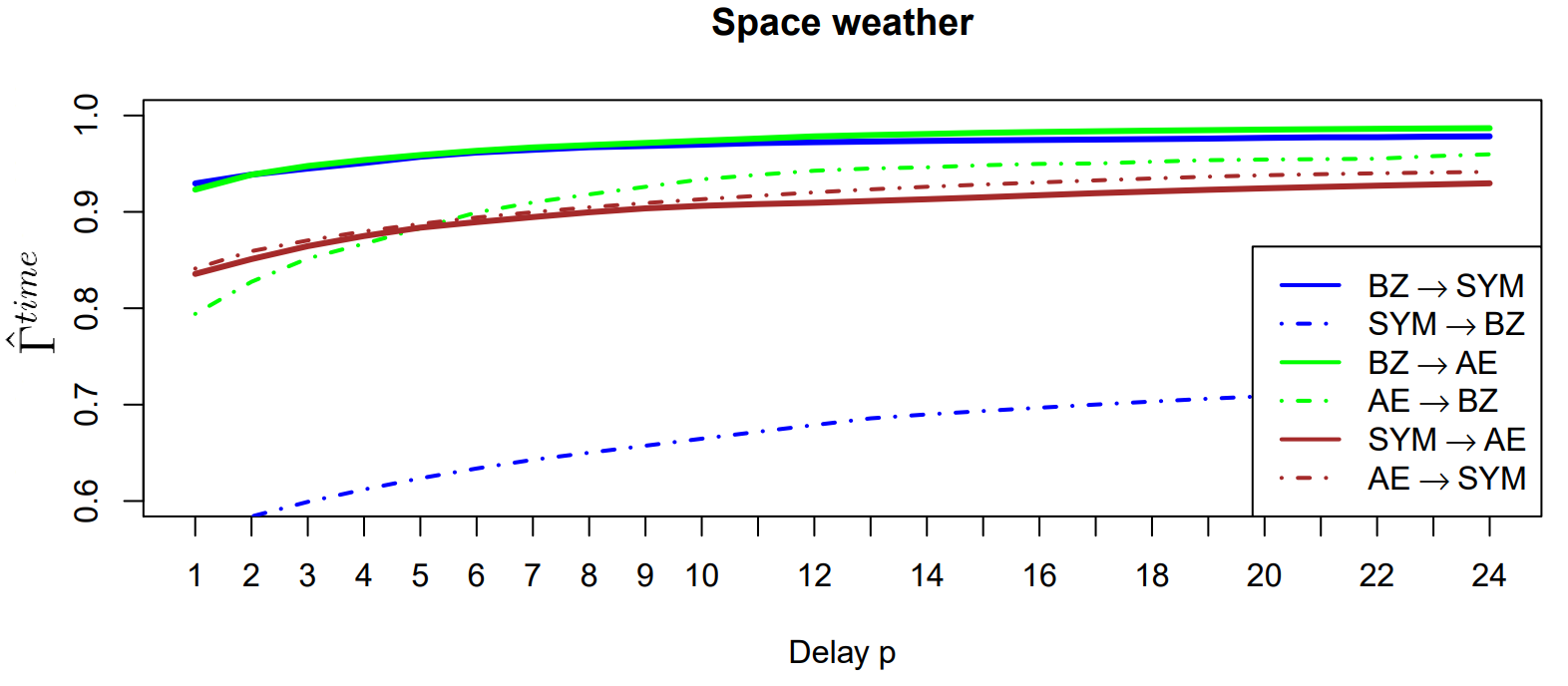}
\includegraphics[scale=0.3]{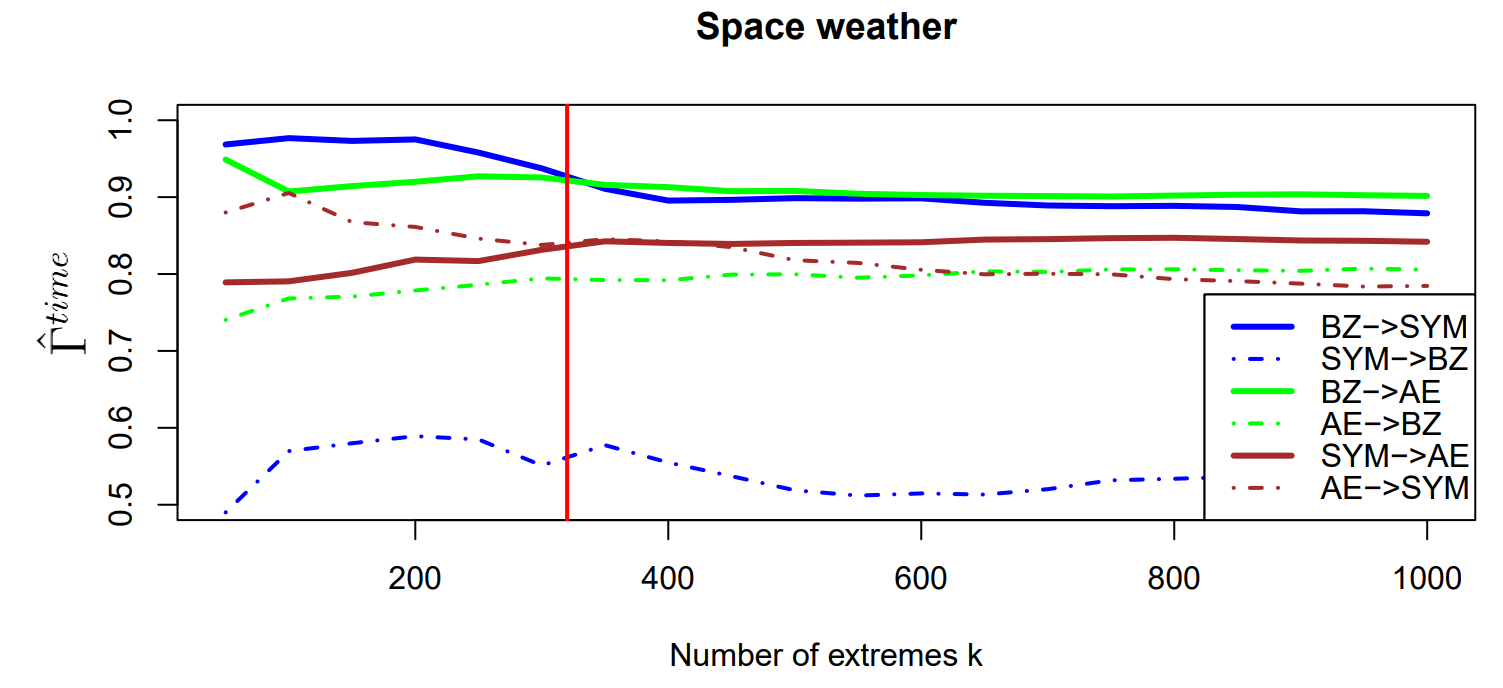}
\caption{The first figure represents all values of $\hat{\Gamma}^{time}_{\cdot\to  \cdot}(p)$ across a range of extreme delays $p\in [1, 24]$ with $k$ set to $k=\sqrt{n}$, for all possible pairs of:  SYM (magnetic storm index), AE (substorm index) and BZ (interplanetary magnetic field). An evident asymmetry in the causal influence between the time series BZ-SYM and BZ-AE is visible from the plot. The second figure represents all values of $\hat{\Gamma}^{time}_{\cdot\to  \cdot}(1)$ for different number of extremes $k$. The red vertical line corresponds to $k=\sqrt{n}$.}
\label{Space lag}
\end{figure}

These results align with the hypothesis that BZ is a common cause of SYM and AE, with no causal relation between them. Note that our methodology has the capability to accommodate scenarios involving a common cause, at least in theory. While conventional methods might suggest a causal relationship from AE to SYM, our analysis reveals instances of extreme events where AE exhibits extremity, yet SYM does not.

\subsection{Hydrometeorology}
\label{Hydrometeorology}

Numerous significant phenomena have an impact on weather and climate, even though their importance and effects are not yet well understood. Within this application, we focus on a selection of the most intriguing phenomena to explore potential causal relationships among them. By delving into these relationships, we aim to gain insights into the intricate interplay among these variables, contributing to a deeper understanding of their influence on weather and climate dynamics.

\begin{itemize}
\item \textbf{El Niño Southern Oscillation (ENSO)} is a climate pattern that describes the unusual warming of surface waters in the eastern tropical Pacific Ocean. It has an impact on ocean temperatures, the speed and strength of ocean currents, the health of coastal fisheries, and local weather from Australia to South America and beyond. 
\item \textbf{North Atlantic Oscillation (NAO)} is a weather phenomenon over the North Atlantic Ocean of fluctuations in the difference of atmospheric pressure. 
\item \textbf{Indian Dipole Index (DMI)} (sometimes referred to as the Dipole Mode Index) is commonly measured by an index describing the difference between sea surface temperature anomalies in two specific regions of the tropical Indian Ocean.
\item \textbf{Pacific Decadal Oscillation (PDO)} is a recurring pattern of ocean-atmosphere climate variability centred over the mid-latitude Pacific basin. 
\item \textbf{East Asian Summer Monsoon Index (EASMI)} is defined as an area-averaged seasonally dynamically normalized seasonality at 850 hPa within the East Asian monsoon domain.
\item \textbf{Amount of rainfall in a region in India (AoR)} is a data set consisting of the amount of rainfall in the region in the centre of India. 
\end{itemize}

For all six variables, we possess monthly measurements starting from 1.1.1953. Following preliminary data manipulation, such as mitigating seasonality and ensuring data stationarity (for comprehensive specifics, please refer to the supplementary package), we turn our attention to assessing the assumption of heavy-tailness.

It seems that the assumption of identical tail indexes across all time series does not hold. Employing the same methodology as in the preceding subsection, we calculate the estimated tail indexes along with their corresponding confidence intervals, as presented in Table \ref{Table5}. It is evident that ENSO and AoR exhibit notably larger tail indexes in comparison to the other variables. On the other hand, variables NAO, DMI, PDO, and EASMI appear to share similar tail indexes (with NAO possibly having a slightly smaller tail index).  Still, we will proceed keeping in mind that these relations between the first and second group may be misleading. 

Finally, we proceed to compute the causal tail coefficients for each pair of time series. In this application, we conclude that $\mathbf{X}$ causes $\mathbf{Y}$ if and only if both $\hat{\Gamma}_{\mathbf{X}\to \mathbf{Y}}^{time}(p)>0.9$ and $\hat{\Gamma}_{\mathbf{Y}\to \mathbf{X}}^{time}(p)<0.8$. This selection is motivated by the outcomes discussed in Subsection \ref{section hidden confounder}. Simulation results indicate that if both $\hat{\Gamma}_{\mathbf{X}\to \mathbf{Y}}^{time}(p), \hat{\Gamma}_{\mathbf{Y}\to \mathbf{X}}^{time}(p)$ are close to $1$, we can not distinguish between the case when $\mathbf{X}\leftrightarrow \mathbf{Y}$ and when there exists a confounder with heavier tails causing both $\mathbf{X}$ and $\mathbf{Y}$. In this application, the absence of a confounder with heavier tails \textit{cannot} be assumed, given the distinct tail indexes observed in the relevant time series. Hence, if both $\hat{\Gamma}_{\mathbf{X}\to \mathbf{Y}}^{time}(p), \hat{\Gamma}_{\mathbf{Y}\to \mathbf{X}}^{time}(p)$ are large, we can not conclude any causal relation. Therefore, we are only interested in asymmetric extremal behavior. We chose the threshold $0.9$ with the same reasoning discussed in Subsection \ref{SectionTesting}. A similar line of thought guides the choice of the threshold 0.8: a suitable threshold ought to reside within the interval $(0.7, 0.9)$, given that most coefficients $ \hat{\Gamma}_{\mathbf{Y}\to \mathbf{X}}^{time}(p)$ from simulations (and the results of this application in Table \ref{Hydro}) lie within this interval. Furthermore, the results remain the same if we chose any threshold in the interval $(0.77, 0.86)$ since no arrows in Figure \ref{Sipky} would change. Consequently, the chosen threshold logically fits within this gap in the computed coefficients.

We opt to use an extremal delay of $p=12$, equivalent to one year, for all pairs in our analysis\footnote{In words, we answer the following question: When an extreme event happens in one time series, will an extreme event happen in the next $12$ months in the other time series? And is this not valid for the opposite direction? We answer yes if $\hat{\Gamma}_{\mathbf{X}\to \mathbf{Y}}^{time}>0.9$ and $\hat{\Gamma}_{\mathbf{Y}\to \mathbf{X}}^{time}<0.8$.}. The resulting causal relationships are outlined in Table \ref{Hydro}.
Subsequently, upon constructing the corresponding graph, we identify specific extremal causal relationships, as depicted in Figure \ref{Sipky}. Here, directed arrows indicate the pairs with an apparent asymmetry in the corresponding $\hat{\Gamma}^{time}(p)$ coefficients (bold numbers in Table \ref{Hydro}). The red arrows correspond to a pair of time series with significantly different tail indexes. 

Concluding that the discovered causal relations are surely true would be a very bold statement. These highly complex datasets may not conform to a straightforward time series model. Their tail indexes are also different. Additionally, the absence of certain causal relationships doesn't imply their nonexistence; for instance, a causal connection might exist without exerting significant influence on extreme events. while this methodology doesn't offer definitive conclusions, it does identify asymmetries in extreme behavior and provides a preliminary notion of directions that might warrant further exploration.

\begin{table}[]
\centering
\begin{tabular}{|c|c|c|c|c|c|c|}
\hline
               & ENSO      & NAO          & DMI          & PDO          & EASMI        & AoR          \\ \hline
$\hat{\theta}$ & 0.44         & 0.19         & 0.28         & 0.25         & 0.25         & 0.44         \\ \hline
CI             & (0.30, 0.56) & (0.14, 0.25) & (0.19, 0.33) & (0.17, 0.33) & (0.17, 0.32) & (0.30, 0.57) \\ \hline
\end{tabular}
\caption{Estimations of the tail indexes and their $95\%$ confidence intervals.}
\label{Table5}
\end{table}

\begin{table}[]
\centering
\begin{tabular}{|l|c|c|c|c|c|c|}
\hline
$\hat{\Gamma}^{time}(p)$ & ENSO             & NAO                    & DMI                    & PDO                    & EASMI                  & AoR                    \\ \hline
ENSO        & $\times$ & 0.88                   & 0.94                   & 0.87                   & \textit{\textbf{0.92}} & \textit{\textbf{0.95}} \\ \hline
NAO            & 0.71                   & $\times$ & 0.83                   & 0.77                   & 0.90                   & 0.94                   \\ \hline
DMI            & 0.86                   & 0.89                   & $\times$ & \textit{\textbf{0.77}} & \textit{\textbf{0.94}} & \textit{\textbf{0.95}} \\ \hline
PDO            & 0.92                   & 0.88                   & \textit{\textbf{0.92}} & $\times$ & \textit{\textbf{0.94}} & \textit{\textbf{0.94}} \\ \hline
EASMI          & \textit{\textbf{0.65}} & 0.92                   & \textit{\textbf{0.74}} & \textit{\textbf{0.72}} & $\times$ & 0.91                   \\ \hline
AoR            & \textit{\textbf{0.69}} & 0.92                   & \textit{\textbf{0.77}} & \textit{\textbf{0.72}} & 0.89                   & $\times$ \\ \hline
\end{tabular}
\caption{Estimation of the causal tail coefficient for time series for each pair of variables. Concerning the order of subscripts, the column comes first, then the row. For example, $\hat{\Gamma}_{DMI\to  NAO}^{time}(p)=0.89$ and $\hat{\Gamma}_{NAO\to  DMI}^{time}(p)=0.83$.  }
\label{Hydro}
\end{table}

\begin{figure}[]
\centering
\includegraphics[scale=0.7]{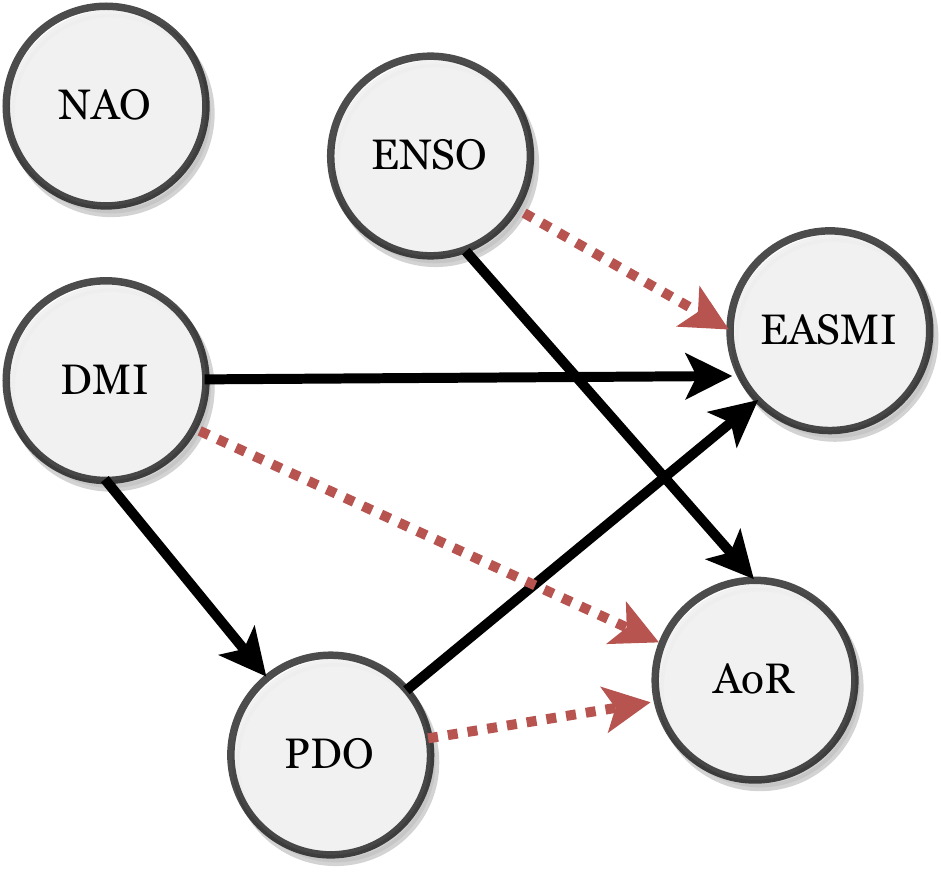}
\caption{The estimated causal directions, determined using the causal tail coefficient for time series, involve the following variables: AoR (Amount of rainfall in India), DMI (Indian Dipole Index), PDO (Pacific Decadal Oscillation), NAO (North Atlantic Oscillation), EASMI (East Asian Summer Monsoon Index) and ENSO (El Niño Southern Oscillation). The presence of arrows denotes asymmetries in extreme events. The presence of arrows denotes asymmetries in extreme events. In cases where the corresponding tail indexes significantly differ (indicating that the assumptions are not satisfied), the arrow is depicted in red.}
\label{Sipky}
\end{figure}

As a means of comparison with existing research in the field of earth science, we offer a selection of pertinent results and articles that tend to concur with our findings. It's important to note that this list is not exhaustive.

\begin{itemize}
\item \textbf{ENSO} $\to$  \textbf{EASMI} \citep{ENSOtoEASMI} (using information theory approach), \citep{PDOtoEASMI},
\item \textbf{ENSO}  $\to$ \textbf{AoR} \citep{ElNinoToAoR} (although rigorous causal methodology was not used),
\item \textbf{ENSO}  $\to$ \textbf{NAO} \citep{ENSOtoNAO} (using a phase dynamics modeling which can give different results than looking at extremes),
\end{itemize}
\begin{itemize}
\item \textbf{DMI}  $\to$ \textbf{ENSO} \citep{DMItoENSO},
\item \textbf{DMI}  $\to$  \textbf{EASMI} \citep{ENSOtoEASMI} (using information theory approach),
\end{itemize}

\begin{itemize}
\item \textbf{PDO}  $\to$ \textbf{AoR}  \citep{ElNinoToAoR} (although rigorous causal methodology was not used),
\item \textbf{PDO}  $\to$ \textbf{EASMI} \citep{PDOtoEASMI} (although rigorous causal methodology was not used),
\end{itemize}

\begin{itemize}
\item \textbf{NAO} is known to influence air temperature and precipitation in Europe, Northern America and a part of Asia \citep{NAOtoEUROPE}. However, whether there are relations between NAO and other investigated variables is yet not known. 
\end{itemize}

Our findings align with the majority of results found in the literature. However, there are three exceptions worth noting:  our results suggest DMI $\to$ PDO  (we couldn't find any reason for this causal relation in the literature). Also, our results do \textit{not} suggest DMI $\to$ ENSO and ENSO $\to$ NAO, although these causal relations seem to hold \citep{ENSOtoNAO, DMItoENSO}.

\section{Conclusion}
\label{CONCLUSION}

Causal inference in time series presents a recurring challenge within the literature. In numerous real-world scenarios, the causal mechanisms become apparent during extreme periods rather than within the bulk of the distribution. We introduce an innovative method that leverages extreme events for detecting causal relationships. This method is formalized within a mathematical framework, and several of its properties are systematically demonstrated under specific model assumptions.

This work sheds light on some connections between causality and extremes. Many scientific disciplines use causal inference as a baseline of their work. A method that can detect a causal direction in complex heavy-tailed datasets can be very useful in such domains.

This subject offers an extensive array of prospects for future research. For instance, what is the underlying distribution of  $\hat\Gamma^{time}_{\mathbf{X}\to \mathbf{Y}}(p)$? Could we devise a better (and consistent) statistic that mitigates the negative bias inherent in $\hat\Gamma^{time}_{\mathbf{X}\to \mathbf{Y}}(p)$? Is there a means to establish a rigorous testing methodology for the hypothesis $\Gamma^{time}_{\mathbf{X}\to \mathbf{Y}}(p) = 1$? Furthermore, can we incorporate observed covariates into the causal inference between $\mathbf{X}$ and $\mathbf{Y}$, moving beyond reliance solely on pairwise discovery? Does this method retain its efficacy even when applied to light-tailed time series? Can a similar approach be adapted to a more generalized framework? These questions can lead to potential future research and important results.

\section*{Conflict of interest and data availability}
\textsf{R} code and the data are available in an online \href{  repository}{https://github.com/jurobodik/Causality-in-extremes-of-time-series.git}, or at a request to an author. 

The authors declare that they have no known competing financial interests or personal relationships that could have appeared to influence the work reported in this paper.

\section*{Acknowledgements}
We are grateful to the editorial team and the referees, for knowledgeable comments that improved the paper. We also wish to thank Valérie Chavez-Demoulin and Katerina Schindler for helpful discussions.  This project was supported by the Czech Science Foundation (Project No. GA19-16066S) and by the Czech Academy of Sciences (Praemium Academiae awarded to M. Palu\v{s}).

\appendix
\section{Auxiliary propositions}
\label{chapter 5}

We repeatedly use the following property of regularly varying random variables  \citep{breiman1965limit}. If $Z\sim \RV(\theta)$, then for every $\alpha>0$ we have $$\pr(\alpha Z>x)\sim \alpha^\theta \pr(Z>x).$$

\begin{proposition}
\label{PrvaLemma}
Let $X,Y, (\varepsilon_i, i\in\mathbb{N})$ be independent continuous random variables with support on some neighborhood of infinity. Let $a_i, b_i\geq 0$, $i\in\mathbb{N}$ and $\lambda_1,\lambda_2\in\mathbb{R}$ be constants. Then $$\pr(X+Y>\lambda_1\mid a_1X+a_2Y>\lambda_2)\geq \pr(X+Y>\lambda_1),$$or more generally, 
\begin{equation*}
    \pr(\sum_{i=1}^{\infty}a_i\varepsilon_i>\lambda_1\mid \sum_{i=1}^{\infty}b_i\varepsilon_i>\lambda_2)\geq \pr(\sum_{i=1}^{\infty}a_i\varepsilon_i>\lambda_1),
\end{equation*}
provided that the sums are a.s. summable, non-trivial. 
\end{proposition}

\begin{proof}
Let $\mathbf{\varepsilon}=(\varepsilon_1, \dots, \varepsilon_n)$, then for any non-decreasing functions $f,g:\mathbb{R}^n\to \mathbb{R}$ we have $$\cov(f(\mathbf{\varepsilon}), g(\mathbf{\varepsilon}))\geq 0.$$  
This is a well-known result from the theory of associated random variables, see, e.g., \cite[Theorem 2.1]{cov(f(X)g(x))>0}. The functions $f(x_1,\dots, x_n)=\1(\sum_{i=1}^n a_ix_i>\lambda_1)$ and $g(x_1,\dots, x_n)=\1(\sum_{i=1}^n b_ix_i>\lambda_2)$ are non-decreasing because $a_i, b_i\geq 0$. Therefore, we obtain
\begin{equation*}\label{NahodnyLabel}
\begin{split}
0&\leq \cov(f(\varepsilon), g(\varepsilon))\\&
=\pr(\sum_{i=1}^n a_i\varepsilon_i>\lambda_1, \sum_{i=1}^n b_i\varepsilon_i>\lambda_2) - \pr(\sum_{i=1}^n a_i\varepsilon_i>\lambda_1)\pr(\sum_{i=1}^n b_i\varepsilon_i>\lambda_2). 
\end{split}
\end{equation*}
Dividing both sides by $\pr(\sum_{i=1}^n b_i\varepsilon_i>\lambda_2) $ (which is positive), we obtain the inequality 
\begin{equation*}
\label{umbalagunga}
    \pr(\sum_{i=1}^{n}a_i\varepsilon_i>\lambda_1\mid \sum_{i=1}^{n}b_i\varepsilon_i>\lambda_2)\geq \pr(\sum_{i=1}^{n}a_i\varepsilon_i>\lambda_1)
\end{equation*}
for arbitrary $n\in\mathbb{N}$. The assertion of the proposition follows by taking the limits as $n \to \infty$.
\end{proof}

\begin{proposition} 
\label{TentoTheorem}
For $i=0,1,2,\dots,$ let
\begin{itemize}
\item $ \varepsilon_i^X, \varepsilon_i^Y\overset{iid}{\sim} \RV(\theta)$ be continuous,
\item $a_i, b_i, c_i\geq 0$ be constants and $\exists\delta>0:$ $\sum_{i=0}^\infty a_i^{\theta-\delta}<\infty, \sum_{i=0}^\infty b_i^{\theta-\delta}<\infty, \sum_{i=0}^\infty c_i^{\theta-\delta}<\infty$,
\item denote  $A=\sum_{i=0}^\infty a_i^\theta, B=\sum_{i=0}^\infty b_i^\theta, C=\sum_{i=0}^\infty c_i^\theta$, for which it holds that $A,B,C\in (0, \infty)$,
\item let $\Phi=\{i\in\mathbb{N}\cup \{0\}: b_i>0=a_i\}$.    
\end{itemize}
Then,    
$$\lim_{u\to\infty}\pr(\sum_{i=0}^\infty a_i \varepsilon_i^X<\lambda\mid \sum_{i=0}^\infty b_i\varepsilon_i^X + \sum_{i=0}^\infty c_i\varepsilon_i^Y>u)= \pr(\sum_{i=0}^\infty a_i \varepsilon_i^X<\lambda)\frac{C+\sum_{i\in\Phi}b_i^\theta}{C+B}$$ 
for all $\lambda\in\mathbb{R}$. 
\end{proposition}

We consider only those $\lambda\in\mathbb{R}$ such that $\pr(\sum_{i=0}^\infty a_i \varepsilon_i^X<\lambda)>0$, otherwise the statement is trivial. Note that the first condition in Proposition \ref{TentoTheorem} implies that all $\sum_{i=0}^\infty a_i\varepsilon_i^X, \sum_{i=0}^\infty b_i\varepsilon_i^X, \sum_{i=0}^\infty c_i\varepsilon_i^Y$ are a.s. summable \citep{Theorem1.7}.We prove this proposition using the following series of lemmas. 

\begin{lemma}
\label{jedna}
Let $X,Y\sim \RV(\theta)$ be independent random variables. Then
$$\lim_{u\to\infty}\pr(X<\lambda\mid X+Y>u)=\pr(X<\lambda)\lim_{u\to\infty}\frac{\pr(Y>u)}{\pr(Y>u)+\pr(X>u)},$$
 for every $ \lambda\in\mathbb{R}$. 
\end{lemma}
\begin{lemma}
\label{dva}
Under the conditions from Proposition \ref{TentoTheorem}, 
$$
\lim_{u\to\infty}\pr(\sum_{i=0}^n a_i\varepsilon_i^X<\lambda\mid \sum_{i=0; i\notin\Phi}^n b_i\varepsilon_i^X>u)=0
$$
for all $n\in\mathbb{N}$.
\end{lemma}
\begin{lemma}
\label{tri}
Let $Z\sim \RV(\theta)$ be a random variable independent of $(\varepsilon_i^X, i\in\mathbb{Z})$. Under the conditions from Proposition \ref{TentoTheorem},
\begin{equation*}
\begin{split}
&\lim_{u\to\infty}\pr(\sum_{i=0}^n a_i \varepsilon_i^X<\lambda\mid \sum_{i=0; i\notin \Phi}^n b_i\varepsilon_i^X + Z>u)\\&
= \pr(\sum_{i=0}^n a_i \varepsilon_i^X<\lambda)\lim_{u\to\infty}\frac{\pr(Z>u)}{\pr(Z>u) + \pr(\sum_{i=0; i\notin \Phi}^n b_i\varepsilon_i^X>u)}
\end{split}
\end{equation*}
for all $n\in\mathbb{N}$.
\end{lemma}
\begin{proof}[Proof of Lemma \ref{jedna}]

Using Bayes theorem, we obtain
\begin{equation*}
\begin{split}
\pr(X<\lambda\mid X+Y>u)&=\pr(X+Y>u\mid X<\lambda)\frac{\pr(X<\lambda)}{\pr(X+Y>u)}.
\end{split}
\end{equation*}
For the denominator we use the sum-equivalence $\pr(X+Y>u)\sim \pr(X>u)+\pr(Y>u)$. Therefore, it is sufficient to show that $\pr(X+Y>u\mid X<\lambda)\sim \pr(Y>u)$.

Now, let $W$ be a random variable independent of $Y$ with a distribution satisfying $\pr(W\leq t)=\pr(X\leq t\mid X<\lambda)$ for all $t\in\mathbb{R}$. Then, $\pr(X+Y>u\mid X<\lambda)=\pr(W+Y>u)$. We obviously have $\lim_{u\to\infty}\frac{\pr(W>u)}{\pr(Y>u)}=0$ and we can use, e.g., Theorem 2.1 in \cite{clanok_o_regularly_varying_densities} to obtain $\lim_{u\to\infty}\frac{\pr(X+Y>u\mid X<\lambda)}{\pr(Y>u)}=\lim_{u\to\infty}\frac{\pr(Y+W>u)}{\pr(Y>u)}=1$. Therefore $\lim_{u\to\infty}\pr(X+Y>u\mid X<\lambda) = \lim_{u\to\infty}\pr(Y>u)$, which we wanted to prove. 
\end{proof}
\begin{proof}[Proof of Lemma \ref{dva}]
Without loss of generality $\Phi=\emptyset$, otherwise we have only lower $n$. Denote $\omega=\min_{i\leq n}a_i$, it holds that $\omega>0$. In this proof only, we denote $B=\sum_{i=0}^nb_i$, and $A=\sum_{i=0}^na_i$. The following events relation are valid:
\begin{equation*}
\begin{split}
\{\sum_{i=0}^n a_i\varepsilon_i^X<\lambda;\sum_{i=0}^n b_i\varepsilon_i^X>u\}&\subseteq\{\exists j\leq n: \varepsilon_j^X>\frac{u}{B}, \sum_{i=0}^n a_i\varepsilon_i^X<\lambda\}\\&
\subseteq\{\exists i,j\leq n: \varepsilon_j^X>\frac{u}{B}, \varepsilon_i^X<\frac{\lambda-\frac{\omega u}{B}}{A}\}. 
\end{split}
\end{equation*}
(Simply put, there needs to be one large and one small $\varepsilon^X_\cdot$). Therefore, we can rewrite
\begin{equation*}
\begin{split}
&\lim_{u\to\infty}\pr(\sum_{i=0}^n a_i\varepsilon_i^X<\lambda\mid \sum_{i=0; i\notin\Phi}^n b_i\varepsilon_i^X>u)\\&
=\lim_{u\to\infty}\frac{\pr(\sum_{i=0}^n a_i\varepsilon_i^X<\lambda;\sum_{i=0}^n b_i\varepsilon_i^X>u)}{\pr(\sum_{i=0}^n b_i\varepsilon_i^X>u)}\\&
\leq \lim_{u\to\infty} \frac{\pr(\exists i,j\leq n: \varepsilon_i^X<\frac{\lambda-\frac{\omega u}{B}}{A}, \varepsilon_j^X>\frac{u}{B}\})}{\pr(\sum_{i=0}^n b_i\varepsilon_i^X>u)}\\&
\leq \lim_{u\to\infty} \frac{n(n+1)\pr(\varepsilon_1^X<\frac{\lambda-\frac{\omega u}{B}}{A}, \varepsilon_2^X>\frac{u}{B})}{\pr(\sum_{i=0}^n b_i\varepsilon_i^X>u)}\\&
=n(n+1) \lim_{u\to\infty} \frac{\pr(\varepsilon_1^X<\frac{\lambda-\frac{\omega u}{B}}{A})\pr(\varepsilon_2^X>\frac{u}{B})}{\sum_{i=0}^nb_i^\theta\cdot \pr(\varepsilon_2^X>u)}\\&
=n(n+1) \lim_{u\to\infty} \pr(\varepsilon_1^X<\frac{\lambda-\frac{\omega u}{B}}{A})\frac{B^\theta}{\sum_{i=0}^nb_i^\theta}=0.
\end{split}
\end{equation*}
\end{proof}
\begin{proof}[Proof of Lemma \ref{tri}]
Without loss of generality $\Phi=\emptyset$, otherwise we have only lower $n$. The proof is very similar to that of Proposition \ref{PrvaLemma}. In this proof only, we denote $B=\sum_{i=0}^nb_i$, and $A=\sum_{i=0}^na_i$. 

Let $W$ be a random variable independent of $Z$ with a distribution satisfying $\pr(W\leq t)=\pr(\sum_{i=0}^n b_i\varepsilon_i^X\leq t\mid \sum_{i=0}^n a_i\varepsilon_i^X<\lambda)$ for all $t\in\mathbb{R}$. Then, $\pr(X+Y>u\mid X<\lambda)=\pr(W+Y>u)$. Using Bayes theorem, we have
\begin{equation*}
\begin{split}
&\lim_{u\to\infty}\pr(\sum_{i=0}^n a_i \varepsilon_i^X<\lambda\mid \sum_{i=0}^n b_i\varepsilon_i^X + Z>u) \\&
=\lim_{u\to\infty}\pr(\sum_{i=0}^n b_i \varepsilon_i^X+Z>u\mid \sum_{i=0}^n a_i\varepsilon_i^X<\lambda)\frac{\pr(\sum_{i=0}^n a_i \varepsilon_i^X<\lambda)}{\pr(\sum_{i=0}^n b_i\varepsilon_i^X + Z>u)}\\&
=\pr(\sum_{i=0}^n a_i \varepsilon_i^X<\lambda)\lim_{u\to\infty}\frac{\pr(W+Z>u)}{\pr(\sum_{i=0}^n b_i\varepsilon_i^X+Z>u)}\\&
=\pr(\sum_{i=0}^n a_i \varepsilon_i^X<\lambda)\lim_{u\to\infty}\frac{\pr(W>u)+\pr(Z>u)}{\pr(\sum_{i=0}^n b_i\varepsilon_i^X>u)+\pr(Z>u)}.
\end{split}
\end{equation*}
In the last equality, we used the fact that $W$ does not have a heavier tail than $Z$ and therefore we can use the sum-equivalence.

All we need to prove is that $\lim_{u\to\infty}\frac{\pr(W>u)}{\pr(\sum_{i=0}^n b_i\varepsilon_i^X>u)}=0$. Again, using Bayes theorem, we obtain 
\begin{equation*}
\begin{split}
&\lim_{u\to\infty}\frac{\pr(W>u)}{\pr(\sum_{i=0}^n b_i\varepsilon_i^X>u)}= \lim_{u\to\infty}\frac{\pr(\sum_{i=0}^n b_i\varepsilon_i^X>u\mid \sum_{i=0}^n a_i\varepsilon_i^X<\lambda)}{\pr(\sum_{i=0}^n b_i\varepsilon_i^X>u)}\\&
=\lim_{u\to\infty}\frac{\pr(\sum_{i=0}^n a_i\varepsilon_i^X<\lambda\mid \sum_{i=0}^n b_i\varepsilon_i^X>u) \frac{\pr(\sum_{i=0}^n b_i\varepsilon_i^X>u)}{\pr(\sum_{i=0}^n a_i\varepsilon_i^X<\lambda)}}{\pr(\sum_{i=0}^n b_i\varepsilon_i^X>u)}\\&
=\frac{1}{\pr(\sum_{i=0}^n a_i\varepsilon_i^X<\lambda)} \lim_{u\to\infty} \pr(\sum_{i=0}^n a_i\varepsilon_i^X<\lambda\mid \sum_{i=0}^n b_i\varepsilon_i^X>u). 
\end{split}
\end{equation*}
The rest follows from Lemma \ref{dva}. 
\end{proof}

\begin{proof}[Proof of Proposition \ref{TentoTheorem}]
Let $\delta>0$, define $\zeta=1-\sqrt{1-\delta}>0$\footnote{$1-\sqrt{1-\delta}$ is a solution of $1-(1-\zeta)(1-\zeta)=\delta$. When $\delta\to 0$ then also $\zeta\to 0$.} and choose large $n_0\in \mathbb{N}$ such that the following conditions hold:
\begin{itemize}
\item $\pr(|\sum_{i=n_0+1}^\infty a_i\varepsilon_i^X|>\delta)<\delta$,
\item $\frac{\sum_{i=0}^{n_0} b_i^\theta + C}{\sum_{i=0}^\infty b_i^\theta + C}>1-\zeta$,
\item $\pr(|\sum_{i=n_0+1; i\notin\Phi}^\infty b_i\varepsilon_i^X| <\delta)>1-\zeta$.
\end{itemize}

Denote 
\begin{itemize}
\item $E=\sum_{i=0}^{n_0} a_i\varepsilon_i^X, F=\sum_{i=n_0+1}^\infty a_i\varepsilon_i^X$,
\item $G=\sum_{i=0; i\notin\Phi}^{n_0} b_i\varepsilon_i^X, H=\sum_{i=n_0+1; i\notin\Phi}^{\infty}  b_i\varepsilon_i^X$,
\item $Z=\sum_{i=0}^\infty c_i\varepsilon_i^Y + \sum_{i\in\Phi} b_i\varepsilon_i^X$. 
\end{itemize}

Then, $E,F,Z$ and also $G,H,Z$ are pairwise independent. 
Sum-max equivalence gives us $\pr(Z>u)\sim [\sum_{i=0}^\infty c_i^\theta + \sum_{i\in\Phi} b_i^\theta]  \pr(\varepsilon_1^X>u)$ and  
$\pr(G+H+Z>u)\sim [\sum_{i=0}^\infty c_i^\theta + \sum_{i=0}^\infty b_i^\theta]  \pr(\varepsilon_1^X>u)$. Hence, with our notation, we want to prove that 
\begin{equation*}
\begin{split}
\lim_{u\to\infty}&\pr(E+F<\lambda\mid G+H+Z>u) \\&
= \pr(E+F<\lambda)\lim_{u\to\infty}\frac{\pr(Z>u)}{\pr(G+H>u)+\pr(Z>u)}.
\end{split}
\end{equation*}

First, due to Lemma \ref{jedna},  
\begin{equation*}
\begin{split}
&\lim_{u\to\infty} \pr(H>\delta\mid H+(G+Z)>u)=1-\lim_{u\to\infty}\pr(H\leq\delta\mid H+(G+Z)>u)\\&
=1-\pr(H\leq\delta)\lim_{u\to\infty}\frac{\pr(G+Z>u)}{\pr(G+Z+H>u)}=1-\pr(H\leq\delta)\frac{\sum_{i=0}^{n_0} b_i^\theta + C}{\sum_{i=0}^\infty b_i^\theta + C}\\&<1-(1-\zeta)(1-\zeta)=\delta.
\end{split}
\end{equation*}

Second, using previous results and independence of $F$ and $(G,Z)$, we obtain
\begin{equation*}
\begin{split}
&\lim_{u\to\infty} \pr(F>\delta\mid G+H+Z>u)\\&
=\lim_{u\to\infty}\frac{\pr(F>\delta, G+H+Z>u, H>\delta)}{\pr(G+H+Z>u)} +\frac{\pr(F>\delta, G+H+Z>u, H\leq \delta)}{\pr(G+H+Z>u)} \\&
\leq\lim_{u\to\infty} \frac{\pr(G+H+Z>u, H>\delta)}{\pr(G+H+Z>u)} +\frac{\pr(F>\delta, G+Z>u-\delta)}{\pr(G+H+Z>u)} \\&
= \lim_{u\to\infty} \pr(H>\delta\mid G+H+Z>u) + \frac{\pr(F>\delta)\pr(G+Z>u-\delta)}{\pr(G+H+Z>u)}\\&
<\delta + \pr(F>\delta)\lim_{u\to\infty}\frac{\pr(G+Z>u)}{\pr(G+Z>u)+\pr(H>u)}\leq \delta +  \pr(F>\delta) <2\delta.
\end{split}
\end{equation*}

Finally, we obtain (the first inequality is trivial; the second uses the identity $\pr(A\cap B)\geq \pr(A)-\pr(B^c)$; the third uses the previous result; the equality follows from Lemma \ref{tri}; the next two inequalities follow from the sum-equivalence and trivial  $\pr(H>u)\geq 0$; the next is trivial and the last inequality follows from $\pr(F+\delta>0)>1-\delta$ and independence of $E,F$):
\begin{equation*}
\begin{split}
&\lim_{u\to\infty} \pr(E+F<\lambda\mid G+H+Z>u)\\&
\geq \lim_{u\to\infty} \pr(E+\delta<\lambda; F\leq \delta\mid G+H+Z>u)\\&
\geq \lim_{u\to\infty} \pr(E+\delta<\lambda\mid G+H+Z>u) - \pr(F>\delta\mid G+H+Z>u)\\&
\geq \lim_{u\to\infty} \pr(E<\lambda-\delta \mid G+(H+Z)>u) -2\delta \\&
= \pr(E<\lambda-\delta)\lim_{u\to\infty}\frac{\pr(H+Z>u)}{\pr(H+Z>u) + \pr(G>u)}-2\delta\\&
\geq \pr(E<\lambda-\delta)\lim_{u\to\infty}\frac{\pr(H>u)+\pr(Z>u)}{\pr(Z>u) + \pr(H>u)+\pr(G>u)}-2\delta\\&
\geq \pr(E<\lambda-\delta)\lim_{u\to\infty}\frac{\pr(Z>u)}{\pr(Z>u) + \pr(G+H>u)}-2\delta\\&
\geq \pr(E+(F+\delta)<\lambda - \delta; (F+\delta)> 0)\lim_{u\to\infty}\frac{\pr(Z>u)}{\pr(Z>u) + \pr(G+H>u)}-2\delta\\&
\geq (1-\delta)\pr(E+F+\delta<\lambda - \delta)\lim_{u\to\infty}\frac{\pr(Z>u)}{\pr(Z>u) + \pr(G+H>u)}-2\delta.
\end{split}
\end{equation*}
When we send $\delta \to 0$, we finally obtain 
\begin{equation*}
\begin{split}
&\pr(E+F<\lambda\mid G+H+Z>u)\\&
\geq \pr(E+F<\lambda)\lim_{u\to\infty}\frac{\pr(Z>u)}{\pr(Z>u) + \pr(G+H>u)},
\end{split}
\end{equation*}
which we wanted to show. The inequality in the opposite direction can be done analogously. 
\end{proof}

\begin{consequence}
Under the conditions from Proposition \ref{TentoTheorem},
$$\lim_{u\to\infty}\pr(\sum_{i=0}^\infty a_i | \varepsilon_i^X|<\lambda\mid \sum_{i=0}^\infty b_i\varepsilon_i^X + \sum_{i=0}^\infty c_i\varepsilon_i^Y>u)= \pr(\sum_{i=0}^\infty a_i |\varepsilon_i^X|<\lambda)\frac{C+\sum_{i\in\Phi}b_i^\theta}{C+B}.$$

\end{consequence}

\begin{proof}
The proof is analogous as that of Proposition \ref{TentoTheorem}. Modified Lemma \ref{jedna} and Lemma \ref{tri} are still valid, just with $|\varepsilon_i^X|$ instead of $\varepsilon_i^X$ in the equations. Modification for Lemma \ref{dva} is trivial, because  
\begin{equation*}
    \begin{split}
   &\lim_{u\to\infty}\pr(\sum_{i=0}^n a_i|\varepsilon_i^X|<\lambda\mid \sum_{i=0; i\notin\Phi}^n b_i\varepsilon_i^X>u)\\&
   \leq  \lim_{u\to\infty}\pr(\sum_{i=0}^n a_i\varepsilon_i^X<\lambda\mid \sum_{i=0; i\notin\Phi}^n b_i\varepsilon_i^X>u)=0.     
    \end{split}
\end{equation*}
The limiting argument for $n\to\infty$ remains the same. 
 \end{proof}

\begin{proposition}
\label{Proposition 3}

Let $(\mathbf{X},\mathbf{Y})^\top$ follow the $\NAAR(q)$ model, specified by 
\begin{align*}
X_t&=f(X_{t-1}) + \varepsilon_t^X,\\
Y_t&=g_{1}(Y_{t-1}) + g_{2}(X_{t-q}) + \varepsilon_t^Y,
\end{align*}
where $f, g_1, g_2$ are continuous and satisfy $\lim_{x\to\infty}h(x)=\infty$ and $\lim_{x\to\infty}\frac{h(x)}{x}<1$, $h=f, g_1, g_2$. Moreover, let $\varepsilon, \varepsilon_t^X, \varepsilon_t^Y\overset{\text{iid}}{\sim} \RV(\theta)$ be non-negative. If $(\mathbf{X},\mathbf{Y})^\top$ is stationary, then
\begin{equation*}
    \lim_{u\to\infty}\frac{\pr(Y_t>u)}{\pr(\varepsilon>u)}<\infty.
\end{equation*}
\end{proposition}

\begin{lemma}
\label{prop 3 lemma}
Under the assumptions from Proposition \ref{Proposition 3},
\begin{equation*}
    \lim_{u\to\infty}\frac{\pr(X_t>u)}{\pr(\varepsilon>u)}<\infty.
\end{equation*}
\end{lemma}

\begin{proof}[Proof of Lemma \ref{prop 3 lemma}]
Let $c=\lim_{x\to\infty}\frac{f(x)}{x}\in [0,1)$. First, notice that   
\begin{equation*}
    \lim_{u\to\infty}\frac{\pr(f(X_t)>u)}{\pr(X_t>u)}= c^\theta.
\end{equation*}
Compute 
\begin{equation*}
\begin{split}
   & \lim_{u\to\infty}\frac{\pr(X_t>u)}{\pr(\varepsilon>u)}=\lim_{u\to\infty}\frac{\pr(f(X_{t-1})+\varepsilon_t^X>u)}{\pr(\varepsilon>u)} \\&
   =1+\lim_{u\to\infty}\frac{\pr(f(X_{t-1})>u)}{\pr(\varepsilon>u)}\leq 1+c^\theta\lim_{u\to\infty}\frac{\pr(X_{t-1}>u)}{\pr(\varepsilon>u)}\\&
   = 1+c^\theta\lim_{u\to\infty}\frac{\pr(X_t>u)}{\pr(\varepsilon>u)}.
\end{split}
\end{equation*}
We have used the sum-equivalence, independence and the previous equation. Therefore, we have $\lim_{u\to\infty}\frac{\pr(X_t>u)}{\pr(\varepsilon>u)} \leq \frac{1}{1-c^\theta}<\infty$.
\end{proof}

\begin{proof}[Proof of Proposition \ref{Proposition 3}]
Find $c<1, K\in\mathbb{R}$ such that for all $x>0$ we have
\begin{equation*}
       f(x)<K+cx, \; g_1(x)<K+cx, \; g_2(x)<K+cx. 
\end{equation*}
 
Note that $f(x+y)\leq (K+cx)+(K+cy)$. Then, the following holds a.s. 
\begin{equation*}
    \begin{split}
Y_0&=\varepsilon_0^Y+g_2(X_{-q}) + g_1(Y_{-1})\leq  \varepsilon_0^Y+g_2(X_{-q}) + K + cY_{-1}\\&
\leq \varepsilon_0^Y+g_2(X_{-q}) + K + c(\varepsilon_{-1}^Y+g_2(X_{-q-1}) + K + cY_{-2})\\&
\leq  (\varepsilon_0^Y + c\varepsilon_{-1}^Y + c^2\varepsilon_{-2}^Y+\dots) \\&\quad +( g_2(X_{-q})+cg_2(X_{-q-1}) +c^2g_2(X_{-q-2})+\dots) +(K+cK+c^2K+\dots)\\&
=\sum_{i=0}^\infty c^i\varepsilon_{-i}^Y+\sum_{i=0}^\infty c^iK + \sum_{i=0}^\infty c^ig_2(X_{q-i})\leq \sum_{i=0}^\infty c^i\varepsilon_{-i}^Y + \frac{2K}{1-c} + \sum_{i=0}^\infty c^{i+1}X_{q-i}. 
    \end{split}
\end{equation*}

Finally, because $X_i$ and $\varepsilon_j^Y$ are all independent,
\begin{equation*}
\begin{split}
   &\lim_{u\to\infty}\frac{\pr(Y_t>u)}{\pr(\varepsilon>u)}\leq\lim_{u\to\infty}\frac{\pr(\sum_{i=0}^\infty c^i\varepsilon_{-i}^Y + \frac{2K}{1-c} + \sum_{i=0}^\infty c^{i+1}X_{q-i}>u)}{\pr(\varepsilon>u)}\\&
   =\lim_{u\to\infty}\frac{\pr(\sum_{i=0}^\infty c^i\varepsilon_{-i}^Y>u) + \pr(\sum_{i=0}^\infty c^{i+1}X_{q-i}>u)}{\pr(\varepsilon>u)}\\&
   =\sum_{i=0}^\infty c^{i\theta} + \lim_{u\to\infty}\frac{\pr(\sum_{i=0}^\infty c^{i+1}X_{q-i}>u)}{\pr(\varepsilon>u)}<\infty,
\end{split}
\end{equation*}
where we used regular variation property, sum-equivalence, and Lemma \ref{prop 3 lemma}.
\end{proof}

\begin{remark}
We proved a stronger claim. We showed that for every $\hNAAR(q, \theta)$ model, there exist a stable $\VAR(q)$ sequence which is a.s. larger. Note that the $\VAR(q)$ process defined by 
\begin{align*}
X_t&=aX_{t-1}+ \varepsilon_t^X; \,\,\,\,\,\,\,
Y_t=bY_{t-1}+dX_{t-q}+ \varepsilon_t^Y,
\end{align*}
with $0<a,b,d<1$, is stable. 
\end{remark}

\section{Proofs}
\label{AppendixB}

\textbf{Observation:} Let $X,Y$ be continuous random variables with support on some neighborhood of infinity, and $F_X, F_Y$ their distribution functions. Then,  $$ \lim_{u\to 1^-}\E[F_Y(Y)\mid F_X(X)>u]=1$$ if and only if  $\lim_{v\to\infty}\pr(Y>\lambda\mid X>v)=1$ for every $\lambda\in \mathbb{R}$. 


\begin{customthm}{\ref{Theorem 2.1}}
Let  $(\mathbf{X},\mathbf{Y})^\top$ be a bivariate time series which follows either the $\hVAR(q, \theta)$ model or the $\hNAAR(q, \theta)$ model. If $\mathbf{X}$ causes $\mathbf{Y}$, then $\Gamma^{time}_{\mathbf{X}\to \mathbf{Y}}(q)=1$. 
\end{customthm}

\begin{proof}[\hypertarget{Proof of Theorem 2.1.}{\textit{Proof for  $\hVAR(q, \theta)$ model}}]

Since $\mathbf{X}$ causes $\mathbf{Y}$, we get $\delta_r>0$ for some $r\leq q$. 
Then
\begin{equation*}
\begin{split}
\Gamma^{time}_{\mathbf{X}\to \mathbf{Y}}(q)& =\lim_{u\to 1^-}\E[\max\{F_Y(Y_0), \dots, F_Y(Y_q)\}\mid F_X(X_0)>u]\\
&\geq \lim_{u\to 1^-}\E[F_Y(Y_r)\mid F_X(X_0)>u] =\lim_{v\to \infty}\E[F_Y(Y_r)\mid X_0>v].
\end{split}
\end{equation*}

Now, if we prove that $\lim_{v\to\infty}\pr(Y_r>\lambda\mid X_0>v)=1$ for all 
$\lambda \in \mathbb{R}$, it will imply that $\lim_{v\to \infty}\E[F_Y(Y_r)\mid X_0>v]=1$ (see \hyperref[AppendixB]{Observation}). Rewrite 
\begin{equation*}
\begin{split}
&\pr(Y_r>\lambda\mid X_0>v)\\&
=\pr(\delta_rX_0 + \sum_{i=1}^q\beta_iY_{r-i}+\sum_{i=1; i\neq p}^q\delta_iX_{r-i} + \varepsilon_r^Y>\lambda\mid X_0>v)\\
&\geq \pr(\delta_rv + \sum_{i=1}^q\beta_iY_{r-i}+\sum_{i=1; i\neq r}^q\delta_iX_{r-i}+ \varepsilon_r^Y>\lambda\mid X_0>v).
\end{split}
\end{equation*}

Now, using causal representation, we can write
\begin{align*}
X_0&=\sum_{i=0}^\infty a_i\varepsilon^X_{-i}+\sum_{i=0}^\infty c_i\varepsilon^Y_{-i},\\
\sum_{i=1}^q\beta_iY_{r-i}+\sum_{i=1; i\neq r}^q\delta_iX_{r-i} + \varepsilon_r^Y&=\sum_{i=0}^\infty \phi_i\varepsilon^X_{r-i}+\sum_{i=0}^\infty \psi_i\varepsilon^Y_{r-i}
\end{align*}
for some $\phi_i, \psi_i\geq 0$. 

We obtain 
\begin{equation*}
\begin{split}
&\lim_{v\to\infty}\pr(\delta_rv + \sum_{i=1}^q\beta_iY_{r-i}+\sum_{i=1; i\neq r}^q\delta_iX_{r-i}>\lambda\mid X_0>v)\\
&=\lim_{v\to\infty}\pr(\sum_{i=0}^\infty \phi_i\varepsilon^X_{q-i}+\sum_{i=0}^\infty \psi_i\varepsilon^Y_{q-i}>\lambda-\delta_rv\mid \sum_{i=0}^\infty a_i\varepsilon^X_{-i}+\sum_{i=0}^\infty c_i\varepsilon^Y_{-i}>v)\\
&\geq \lim_{v\to\infty} \pr(\sum_{i=0}^\infty \phi_i\varepsilon^X_{q-i}+\sum_{i=0}^\infty \psi_i\varepsilon^Y_{q-i}>\lambda-\delta_rv)=1,
\end{split}
\end{equation*}
where we used Proposition \ref{PrvaLemma} in the last step. Therefore, $\lim_{v\to\infty}\pr(Y_r>\lambda\mid X_0>v)\geq 1$, which proves the theorem. 

\bigskip

\noindent
\textit{Proof for  $\hNAAR(q, \theta)$ model}.
We proceed very similarly as in the proof of $\hVAR(q, \theta)$ model. We rewrite $\Gamma^{time}_{\mathbf{X}\to \mathbf{Y}}(q) \geq \lim_{v\to \infty}\E[F_Y(Y_q)\mid X_0>v]$, which is equal to 1 if $\lim_{v\to\infty}\pr(Y_q>\lambda\mid X_0>v)=1$ for all $\lambda\in\mathbb{R}$. We rewrite 
\begin{equation*}
\lim_{v\to\infty}\pr(Y_q>\lambda\mid X_0>v) = \lim_{v\to\infty}\pr(g_1(Y_{q-1}) + g_2(X_0)+\varepsilon_q^Y>\lambda\mid X_0>v).  
\end{equation*}
Since $\mathbf{X}$ causes $\mathbf{Y}$, it holds that $g_2$ is not constant and $\lim_{x\to\infty}g_2(x)=\infty$. 
This implies that there exists $x_0\in\mathbb{R}$ such that $g_2(x)>\lambda$ for all $\lambda\in\mathbb{R}$. Therefore, for all $u>x_0$,
$$
\pr(g_2(X_0)>\lambda\mid X_0>v)=1.
$$
Finally, we only use the fact that $\varepsilon_t^Y$ and $g_1$ are non-negative. Hence,
\begin{equation*}
\begin{split}
&\lim_{v\to\infty}\pr(g_1(Y_{q-1}) + g_2(X_0)+\varepsilon_t^Y>\lambda\mid X_0>v)\\&\geq \lim_{v\to\infty}\pr(g_2(X_0)>\lambda\mid X_0>v)=1,
\end{split}
\end{equation*}
which we wanted to prove. 
\end{proof}


\begin{customthm}{\ref{Theorem 2.2.}}
Let  $(\mathbf{X},\mathbf{Y})^\top$ be a bivariate time series which follows either the $\hVAR(q, \theta)$ model or the $\hNAAR(q, \theta)$ model. If $\mathbf{Y}$ does not cause $\mathbf{X}$, then $\Gamma^{time}_{\mathbf{Y}\to \mathbf{X}}(p)<1$ for all $p\in\mathbb{N}$. 
\end{customthm}

\begin{proof}[\hypertarget{Proof of Theorem 2.2.}{Proof for  $\hVAR(q, \theta)$ model}]
Let $\lambda\in\mathbb{R}$ be such that $\pr(X_0<\lambda)>0$. We show that $$\lim_{v\to\infty}\pr(\max (X_0, \dots, X_p)<\lambda\mid Y_0>v)>0,$$from which it follows that $\lim_{v\to\infty}\E[\max (F_X(X_0), \dots, F_X(X_p))\mid Y_0>v]<1$.

Rewrite 
\begin{equation*}
\begin{split}
\pr(\max (X_0, \dots, X_p)<\lambda\mid Y_0>v)&
= \pr(X_0<\lambda, \dots, X_p<\lambda\mid Y_0>v) \\&
\geq \pr(|X_0|+|X_1|+\dots + |X_p|<\lambda\mid Y_0>v).
\end{split}
\end{equation*}

Now, we use the causal representation of the time series, which, because we know that $\mathbf{Y}$ does not cause $\mathbf{X}$, can be written in the form
\begin{align*}
X_t&=\sum_{i=0}^\infty a_i\varepsilon^X_{t-i}; \,\,\,\,\,\,\,\,\,
Y_t=\sum_{i=0}^\infty b_i\varepsilon^Y_{t-i}+\sum_{i=0}^\infty d_i\varepsilon^X_{t-i}.
\end{align*}
We obtain 
\begin{equation*}
    \begin{split}
  &\pr(\sum_{t=0}^p|X_t|<\lambda\mid Y_0>v)=  \pr(\sum_{t=0}^p|\sum_{i=0}^\infty a_i\varepsilon^X_{t-i}|<\lambda\mid Y_0>v)\\&
  \geq  \pr(\sum_{t=0}^p\sum_{i=0}^\infty a_i|\varepsilon^X_{t-i}|<\lambda\mid Y_0>v)\\&
  = \pr(\sum_{i=0}^\infty \phi_i|\varepsilon^X_{p-i}|<\lambda\mid \sum_{i=0}^\infty b_i\varepsilon^Y_{-i}+\sum_{i=0}^\infty d_i\varepsilon^X_{-i}>v),
    \end{split}
\end{equation*}
for $\phi_i=a_i+\dots + a_{i-p}$ (we define $a_j=0$ for $j<0$). Finally, it follows from the consequence of Proposition \ref{TentoTheorem} that $$\lim_{v\to\infty} \pr(\sum_{i=0}^\infty \phi_i|\varepsilon^X_{p-i}|<\lambda\mid \sum_{i=0}^\infty b_i\varepsilon^Y_{-i}+\sum_{i=0}^\infty d_i\varepsilon^X_{-i}>v)>0,$$ which we wanted to prove (Proposition \ref{TentoTheorem} requires non-trivial sums, but if $\forall i: d_i=0$ then the series are independent and this inequality holds trivially).  

\bigskip

\noindent
\textit{Proof for $\hNAAR(q, \theta)$ model}.
We have 
\begin{align*}
X_t&=f(X_{t-1}) + \varepsilon_t^X; \,\,\,\,\,\,\,
Y_t=g_1(Y_{t-1})+g_2(X_{t-q}) + \varepsilon_t^Y.
\end{align*}
Choose large $\lambda\in\mathbb{R}$, such that $\sup_{x<\lambda}f(x)<\lambda$ and such that\footnote{This is possible from the assumptions on continuity and the limit behavior of $f$.} $$\pr\big(\varepsilon_0^X<\lambda-\sup_{x<\lambda}f(x)\big)>0.$$ Denote $\lambda^\star=\sup_{x<\lambda}f(x)$. Rewrite
\begin{equation*}
\begin{split}
&\pr(\max (X_0, \dots,  X_q)<\lambda\mid Y_0>v)
= \pr(X_0<\lambda, \dots, X_q<\lambda\mid Y_0>v)\\&
=\prod_{i=0}^q  \pr(X_i<\lambda\mid X_0<\lambda, \dots, X_{i-1}<\lambda, Y_0>v).
\end{split}
\end{equation*}
Then, as in the proof for the Heavy-tailed VAR model case, if we show that this is strictly larger than 0, it will imply that $\Gamma_{\mathbf{Y}\to \mathbf{X}}^{time}(q)<1$.  We know that for every $i\geq 1$ the following holds
\begin{equation*}
\begin{split}
&\lim_{v\to\infty} \pr(X_i<\lambda\mid X_0<\lambda, \dots, X_{i-1}<\lambda, Y_0>v)\\&
=\lim_{v\to\infty} \pr(f(X_{i-1})+\varepsilon_i^X<\lambda\mid X_0<\lambda, \dots, X_{i-1}<\lambda, Y_0>v)\\&
\geq\lim_{v\to\infty} \pr(\lambda^\star+\varepsilon_i^X<\lambda\mid X_0<\lambda, \dots, X_{i-1}<\lambda, Y_0>v)\\&
=\pr(\lambda^\star+\varepsilon_i^X<\lambda)>0. 
\end{split}
\end{equation*} 

We only need to show for the case when $i=0$ that $\lim_{v\to\infty} \pr(X_0>\lambda\mid Y_0>v) <1$. Let $Z=g_1(Y_{-1})+ g_2(X_{-q})$, $Z$ is independent with $\varepsilon_0^Y$. After rewriting, we obtain 
\begin{equation*}
\begin{split}
\pr(X_0>\lambda\mid Y_0>v) &= \pr(X_0>\lambda\mid \varepsilon_0^Y + Z>v) 
= \frac{\pr(X_0>\lambda;\varepsilon_0^Y + Z>v) }{\pr(\varepsilon_0^Y + Z>v)}.
\end{split}
\end{equation*}
Let $\frac{1}{2}<\delta<1$ (we will later send $\delta\to 1$). The following events-relation is valid:
\begin{equation*}
\begin{split}
&\{X_0>\lambda;\varepsilon_0^Y + Z>v\}  \\&\subseteq
\{X_0>\lambda;\varepsilon_0^Y>\delta v\} \cup \{Z>\delta v\} \cup \{Z>(1-\delta)v; \varepsilon_0^Y>(1-\delta)v\}.
\end{split}
\end{equation*}
Applying it to the previous equation, we obtain
\begin{equation*}
\begin{split}
&\lim_{v\to\infty} \frac{\pr(X_0>\lambda;\varepsilon_0^Y + Z>v) }{\pr(\varepsilon_0^Y + Z>v)}\\&
\leq \lim_{v\to\infty} \frac{\pr(X_0>\lambda;\varepsilon_0^Y>\delta v) + \pr(Z>\delta v) + \pr(Z>(1-\delta)v; \varepsilon_0^Y>(1-\delta)v) }{\pr(\varepsilon_0^Y + Z>v)}\\&
=\lim_{v\to\infty} \frac{\pr(X_0>\lambda)\pr(\varepsilon_0^Y>\delta v)}{\pr(\varepsilon_0^Y + Z>v)} + \frac{ \pr(Z>\delta v)}{\pr(\varepsilon_0^Y + Z>v)}\\&
\quad +\lim_{v\to\infty}\pr(Z>(1-\delta))\frac{(\frac{1}{1-\delta})^\theta \pr(\varepsilon_0^Y>v)}{\pr(\varepsilon_0^Y + Z>v)}\\&
=\frac{1}{\delta^\theta}\lim_{v\to\infty} \frac{\pr(X_0>\lambda)\pr(\varepsilon_0^Y>v)}{\pr(\varepsilon_0^Y + Z>v)}+ \frac{ \pr(Z>\delta v )}{\pr(\varepsilon_0^Y + Z>v)} + 0.
\end{split}
\end{equation*}
The last summand is $0$ because $\lim_{v\to\infty}\pr(Z>(1-\delta)v)=0$ and $\frac{\pr(\varepsilon_0^Y>v)}{\pr(\varepsilon_0^Y + Z>v)}\leq 1$ (simply because $Z$ is a non-negative random variable).

Now, we will use the result from Proposition \ref{Proposition 3}.  In the case when\\ $\lim_{v\to\infty} \frac{\pr(Z>v)}{\pr(\varepsilon_0^Y>v)}=0$, we obtain (see, e.g., Lemma 1.3.2 in \cite{Heavy_tailed_time_series}) $\lim_{v\to\infty}\frac{\pr(\varepsilon_0^Y>v)}{\pr(\varepsilon_0^Y+Z>v)}=1$ and $\lim_{v\to\infty}\frac{\pr(Z>v)}{\pr(\varepsilon_0^Y+Z>v)}=0$. Therefore,
\begin{equation*}
\lim_{v\to\infty} \frac{ \frac{1}{\delta^\theta}  \pr(X_0>\lambda)\pr(\varepsilon_0^Y>v) + \pr(Z>\delta v)}{\pr(\varepsilon_0^Y + Z>v)} = \frac{1}{\delta^\theta}\pr(X_0>\lambda)<1,
\end{equation*}
for $\delta$ close enough to 1. 

On the other hand, if $\lim_{v\to\infty} \frac{\pr(Z>v)}{\pr(\varepsilon_0^Y>v)}=c\in\mathbb{R}^+$, we also have that $Z\sim \RV(\theta)$ (this follows trivially from the definition of regular variation, tails behavior is the same up to a constant). Therefore, we can apply the sum-equivalence and we obtain
\begin{equation*}
    \begin{split}
    &\lim_{v\to\infty} \frac{\frac{\pr(X_0>\lambda)}{\delta^\theta}\pr(\varepsilon_0^Y>v) + \pr(Z>\delta v)}{\pr(\varepsilon_0^Y + Z>v)} \\&
    = \frac{1}{\delta^\theta}\lim_{v\to\infty} \frac{\pr(X_0>\lambda)\pr(\varepsilon_0^Y>v) + \pr(Z>v)}{\pr(\varepsilon_0^Y>v)+\pr(Z>v)}\\&
    = \frac{1}{\delta^\theta}\lim_{v\to\infty} \frac{\pr(X_0>\lambda)\pr(\varepsilon_0^Y>v) + c\pr(\varepsilon_0^Y>v)}{\pr(\varepsilon_0^Y>v)+c\pr(\varepsilon_0^Y>v)}\\&
    =\frac{1}{\delta^\theta}\frac{\pr(X_0>\lambda)+c}{1+c},
    \end{split}
\end{equation*}
which is less than 1 for $\delta$ close enough to 1. Therefore, we obtain \\ $\lim_{v\to\infty} \pr(X_0>\lambda\mid Y_0>v) <1$, which we wanted to prove.  
\end{proof}

\begin{customlem}{\ref{lemma Extremal causal condition}}
The extremal causal condition holds in the $\hVAR(q, \theta)$ model (i.e., where the coefficients are non-negative) when $\mathbf{X}$ causes $\mathbf{Y}$. 
\end{customlem}

\begin{proof}
\hypertarget{proof of lemma Extremal causal condition}{}
In the notion of Definition \ref{heavy-tailed-VAR} and Theorem \ref{Theorem 2.2.}, if $\delta_p>0$, then
\begin{equation*}
\begin{split}
        \sum_{i=0}^\infty d_i\varepsilon^X_{p-i}+\sum_{i=0}^\infty b_i\varepsilon^Y_{p-i}&=Y_p=\delta_pX_0+\dots=\delta_p(\sum_{i=0}^\infty a_i\varepsilon^X_{-i}+\sum_{i=0}^\infty c_i\varepsilon^Y_{-i})+\dots.
\end{split}
\end{equation*}
 Therefore, if $a_i>0$, then $d_{i+p}\geq \delta_pa_i>0$. 
\end{proof}

\begin{customthm}{\ref{Absolute value theorem}}
Let  $(\mathbf{X},\mathbf{Y})^\top$ be a time series which follows the $\hVAR(q, \theta)$ model, with possibly negative coefficients, satisfying the extremal causal condition. Moreover, let $\varepsilon_t^X, \varepsilon_t^Y$ have full support on $\mathbb{R}$, and $|\varepsilon_t^X|, |\varepsilon_t^Y|\sim \RV(\theta)$.  If $\mathbf{X}$ causes $\mathbf{Y}$, but $\mathbf{Y}$ does not cause $\mathbf{X}$, then $\Gamma^{time}_{|\mathbf{X}|\to |\mathbf{Y}|}(q)=1$, and $\Gamma^{time}_{|\mathbf{Y}|\to |\mathbf{X}|}(q)<1$. 
\end{customthm}

\begin{proof}
\hypertarget{Proof Absolute value theorem}{}
First, we  show that if $\mathbf{Y}$ does not cause $\mathbf{X}$, then  $\Gamma^{time}_{|\mathbf{Y}|\to |\mathbf{X}|}(q)<1$. This holds even without the extremal causal condition. Similarly as in the proof of Theorem \ref{Theorem 2.2.}, it is sufficient to show that $\exists \lambda>0:$ 
\begin{equation*}
\begin{split}
&\lim_{v\to\infty}\pr(|X_t|>\lambda\mid |Y_0|>v) \\&
= \lim_{v\to\infty}\pr(|\sum_{i=0}^\infty a_i\varepsilon^X_{t-i}|>\lambda\mid |\sum_{i=0}^\infty b_i\varepsilon^Y_{-i}+\sum_{i=0}^\infty d_i\varepsilon^X_{-i}|>v)<1 
\end{split}
\end{equation*}
for $t\leq q$.  

We  use the following fact. Since we assumed that $\varepsilon_i^X$ and $|\varepsilon_i^X|$ are $\RV(\theta)$~ \footnote{This also implies that they satisfy the \textit{tail balance condition} that is defined in \cite{RegularlyVaryingFunctiosMikosch}}, the following holds: 
$$
\pr(|\sum_{i=0}^\infty a_i\varepsilon^X_{t-i}|>v)\sim [\sum_{i=0}^\infty |a_i|^\theta] \pr(|\varepsilon^X_{0}|>v)\sim \pr(\sum_{i=0}^\infty |a_i||\varepsilon^X_{t-i}|>v),
$$
see, e.g., \citep[Lemma 3.5]{RegularlyVaryingFunctiosMikosch}. The second step follows simply from the max-sum equivalence. Finally, we use this fact and the triangle inequality to obtain the following relations
\begin{equation*}
\begin{split}
&\pr(|\sum_{i=0}^\infty a_i\varepsilon^X_{t-i}|>\lambda\mid |\sum_{i=0}^\infty b_i\varepsilon^Y_{-i}+\sum_{i=0}^\infty d_i\varepsilon^X_{-i}|>v)\\&
\leq \frac{\pr(\sum_{i=0}^\infty |a_i||\varepsilon^X_{t-i}|>\lambda; \sum_{i=0}^\infty |b_i||\varepsilon^Y_{-i}|+\sum_{i=0}^\infty |d_i||\varepsilon^X_{-i}|>v)}{\pr(|\sum_{i=0}^\infty b_i\varepsilon^Y_{-i}+\sum_{i=0}^\infty d_i\varepsilon^X_{-i}|>v)} \\&
\sim \frac{\pr(\sum_{i=0}^\infty |a_i||\varepsilon^X_{t-i}|>\lambda; \sum_{i=0}^\infty |b_i||\varepsilon^Y_{-i}|+\sum_{i=0}^\infty |d_i||\varepsilon^X_{-i}|>v)}{\pr(\sum_{i=0}^\infty |b_i||\varepsilon^Y_{-i}|+\sum_{i=0}^\infty |d_i||\varepsilon^X_{-i}|>v)}\\&
=\pr(\sum_{i=0}^\infty |a_i||\varepsilon^X_{t-i}|>\lambda\mid \sum_{i=0}^\infty |b_i||\varepsilon^Y_{-i}|+\sum_{i=0}^\infty |d_i||\varepsilon^X_{-i}|>v).
\end{split}
\end{equation*}

This is for $v\to\infty$ less than $1$ due to the classical non-negative case from Proposition \ref{TentoTheorem} (for any $\lambda\in\mathbb{R}$ such that $\pr(|\sum_{i=0}^\infty a_i\varepsilon^X_{t-i}|>\lambda)<1$).

Second, we show that if $\mathbf{X}$ causes $\mathbf{Y}$, then $\Gamma^{time}_{|\mathbf{X}|\to |\mathbf{Y}|}(q)=1$. Similarly, as in the proof of Theorem \ref{Theorem 2.1}, it is sufficient to show that
$$
\lim_{v\to\infty} \pr(|Y_r |<\lambda \mid  |X_0|>u)=0
$$ 
for every $\lambda\in\mathbb{R}$.
Here, $r\leq q$ is some index with $\delta_r\neq 0$. Using the causal representation with the same notation as in the proof of Theorem \ref{Theorem 2.1}, 
\begin{equation*}
\begin{split}
&\lim_{v\to\infty} \pr(|\sum_{i=0}^\infty b_i\varepsilon^Y_{r-i}+\sum_{i=0}^\infty d_i\varepsilon^X_{r-i}|<\lambda \mid |\sum_{i=0}^\infty a_i\varepsilon^X_{-i}|>v)\\&
\leq \lim_{v\to\infty} \pr(\sum_{i=0}^\infty |b_i||\varepsilon^Y_{r-i}|+\sum_{i=0}^\infty |d_i||\varepsilon^X_{r-i}|<\lambda \mid  |\sum_{i=0}^\infty |a_i||\varepsilon^X_{-i}|>v),
\end{split}
\end{equation*}
where we used the same argument as in the first part of the proof. Therefore, we simplified our model and obtained the classical non-negative case. The result follows from the previous theory. Using Lemma \ref{dva} we obtain the result for finite $n$, 
$$
\lim_{v\to\infty} \pr(\sum_{i=0}^n |b_i||\varepsilon^Y_{r-i}|+\sum_{i=0}^n |d_i||\varepsilon^X_{r-i}|<\lambda \mid  |\sum_{i=0}^n |a_i||\varepsilon^X_{-i}|>v)=0,
$$
because $\Phi=\emptyset$ due to the extremal causal condition. The argument for limiting case $n\to\infty$ follows the same steps as those in the proof of Proposition \ref{TentoTheorem}.  
\end{proof}


\noindent
\textbf{Theorem \ref{Common cause theorem}.} Let $(\mathbf{X,Y,Z})^\top = ((X_t,Y_t,Z_t)^\top, t\in\mathbb{Z})$  follow the three-dimensional stable VAR$(q)$ model, with iid regularly varying noise variables. Let $\mathbf{Z}$ be a common cause of both $\mathbf{X}$ and $\mathbf{Y}$, and neither $\mathbf{X}$ nor $\mathbf{Y}$ cause $\mathbf{Z}$. If $\mathbf{Y}$ does not cause $\mathbf{X}$, then $\Gamma^{time}_{\mathbf{Y}\to \mathbf{X}}(p)<1$ for all $p\in\mathbb{N}$.

\begin{proof}
\hypertarget{Proof of Common cause theorem}{}
Let our series have the following representation: 
\begin{align*}
Z_t&=\sum_{i=0}^{\infty} a_i\varepsilon_{t-i}^Z,\\
X_t&=\sum_{i=0}^{\infty} b_i\varepsilon_{t-i}^X+\sum_{i=0}^{\infty} c_i\varepsilon_{t-i}^Z,\\
Y_t&=\sum_{i=0}^{\infty} d_i\varepsilon_{t-i}^X+\sum_{i=0}^{\infty} e_i\varepsilon_{t-i}^Y+\sum_{i=0}^{\infty} f_i\varepsilon_{t-i}^Z.
\end{align*}
Just as in the proof of Theorem \ref{Theorem 2.2.}, it is sufficient to show that $\lim_{v\to\infty}\pr(X_p>\lambda|Y_0>v)<1$ for some $\lambda>0$. After rewriting, 
$$
\lim_{v\to\infty}\pr(\sum_{i=0}^{\infty} b_i\varepsilon_{p-i}^X+\sum_{i=0}^{\infty} c_i\varepsilon_{p-i}^Z>\lambda\mid \sum_{i=0}^{\infty} d_i\varepsilon_{-i}^X+\sum_{i=0}^{\infty} e_i\varepsilon_{-i}^Y+\sum_{i=0}^{\infty} f_i\varepsilon_{-i}^Z>v)<1,  
$$
which follows from Proposition \ref{TentoTheorem} (two countable sums can be written as one countable sum). 
\end{proof}

\begin{customlem}{\ref{minimal lag lemma}}
Let $(\mathbf{X},\mathbf{Y})^\top$ follow the $\hVAR(q, \theta)$ model, where $\mathbf{X}$ causes $\mathbf{Y}$. Let $s$ be the minimal delay. Then, $\Gamma^{time}_{\mathbf{X}\to \mathbf{Y}}(r)<1$ for all $r<s$, and $\Gamma^{time}_{\mathbf{X}\to \mathbf{Y}}(r)=1$ for all $r\geq s$. 
\end{customlem}

\begin{proof}
\hypertarget{Proof of minimal lag lemma}{}
Proving that $\Gamma^{time}_{\mathbf{X}\to \mathbf{Y}}(r)=1$ for all $r\geq s$, is a trivial consequence of the proof of Theorem \ref{Theorem 2.1} (in the first row of the proof, instead of choosing \textit{some} $s\leq q : \delta_s>0$, we choose $s$ to be the minimal delay). 

Concerning the first part, we only need to prove that $\Gamma^{time}_{\mathbf{X}\to \mathbf{Y}}(s-1)<1$, because then also $\Gamma^{time}_{\mathbf{X}\to \mathbf{Y}}(s-i)\leq \Gamma^{time}_{\mathbf{X}\to \mathbf{Y}}(s-1)<1$. As in the proof of Theorem \ref{Theorem 2.2.}, we only need to show that $\lim_{v\to\infty} \pr(Y_{s-1}<\lambda|X_0>v)>0$ for some $\lambda>0$. By rewriting to its causal representation, we obtain the following relation
\begin{equation*}
\lim_{v\to\infty} \pr(\sum_{i=0}^\infty b_i\varepsilon^Y_{s-1-i}+\sum_{i=0}^\infty d_i\varepsilon^X_{s-1-i}<\lambda\mid \sum_{i=0}^\infty a_i\varepsilon^X_{-i}+\sum_{i=0}^\infty c_i\varepsilon^Y_{-i}>v)>0. 
\end{equation*}
We only need to realize that $d_i=0$ for $i\in\{1, \dots, s-1\}$ because $s$ is the minimal delay. Therefore, $\varepsilon^X_0$ is independent of $Y_{s-1}$ and the rest follows from Proposition \ref{TentoTheorem} (where we deal with the two sums as one, and single $\varepsilon^X_0$ is the second \say{sum}). 
\end{proof}

\begin{customthm}{\ref{Theorem o asymptotic}}
Let $(\mathbf{X},\mathbf{Y})^\top=((X_t,Y_t)^\top, t\in\mathbb{Z})$ be a stationary bivariate time series, whose marginal distributions are absolutely continuous with support on some neighborhood of infinity. Let $\Gamma^{time}_{\mathbf{X}\to \mathbf{Y}}(p)$ exist. Let $k_n$ satisfy (\ref{k_deleno_n}) and 
\begin{equation*}
\frac{n}{k_n}P\left(\frac{n}{k_n} \sup_{x\in\mathbb{R}}|\hat{F}_X(x)-F(x)|>\delta\right)\overset{n\to\infty}{\longrightarrow}0, \,\,\, \forall \delta>0.
\end{equation*}
Then, $\E \hat\Gamma^{time}_{\mathbf{X}\to \mathbf{Y}}(p)\overset{n\to\infty}{\to}\Gamma^{time}_{\mathbf{X}\to \mathbf{Y}}(p)$.
\end{customthm}

\begin{proof}  
\hypertarget{Proof of asymptotic unbias}{}
Throughout the proof, we use the fact that for a continuous $X_1$ always holds $\pr(F_X(X_1)\leq t)=t$ for $t\in [0,1]$ and the fact that follows from the stationarity $\pr(\hat{F}_X(X_1)\leq \frac{k}{n})=\pr(X_1\leq X_{(k)})=\frac{k}{n}$, for $k\leq n, k\in \N$. Please note that $X_{(k)}$ is always meant with respect to (not written) index $n$. 

First, notice the following (the third equation follows from the linearity of expectation and stationarity of our series; the fourth equation follows from the definition of conditional expectation; the fifth is quite trivial):
\begin{equation*}
\begin{split}
\E\hat\Gamma^{time}_{\mathbf{X}\to \mathbf{Y}}(p)&= \E\frac{1}{k_n}\sum_{i: X_i\geq\tau_{k_n}^X}\max\{\hat{F}_Y(Y_i), \dots, \hat{F}_Y(Y_{i+p})\}\\&
=\E\frac{1}{n}\sum_{i=1}^n\frac{n}{k_n}\max\{\hat{F}_Y(Y_i), \dots, \hat{F}_Y(Y_{i+p})\}\1(\hat{F}_X(X_i)>1-\frac{k_n}{n})\\&
=\frac{n}{k_n}\E [\hat{F}_Y(\max\{Y_1, \dots, Y_{p+1}\})\1(\hat{F}_X(X_1)>1-\frac{k_n}{n})]\\&
=\frac{n}{k_n} \pr(\hat{F}_X(X_1)>1-\frac{k_n}{n})\cdot\\&\;\;\;\;\;\;\;\;\cdot\E [\hat{F}_Y(\max\{Y_1, \dots, Y_{p+1}\})\mid \hat{F}_X(X_1)>1-\frac{k_n}{n}]\\&
=\E [\hat{F}_Y(\max\{Y_1, \dots, Y_{p+1}\})\mid \hat{F}_X(X_1)>1-\frac{k_n}{n}].
\end{split}
\end{equation*}
Now, use $\hat{F}=F+ \hat{F}-F$ to obtain
\begin{equation*}
\begin{split}
&\E [\hat{F}_Y(\max\{Y_1, \dots, Y_{p+1}\})\mid \hat{F}_X(X_1)>1-\frac{k_n}{n}]\\&
=\E [F_Y(\max\{Y_1, \dots, Y_{p+1}\})\mid \hat{F}_X(X_1)>1-\frac{k_n}{n}] \\&
\quad +\E [(\hat{F}_Y-F_Y)(\max\{Y_1, \dots, Y_{p+1}\})\mid \hat{F}_X(X_1)>1-\frac{k_n}{n}].
\end{split}
\end{equation*}
The second term is less than $\E[\sup_{x\in\mathbb{R}}|\hat{F}_Y(x)-F_Y(x)|]\to 0$ as $n\to\infty$ from the assumptions. All we need to show is that the first term converges to $\Gamma^{time}_{\mathbf{X}\to \mathbf{Y}}(p)$. Rewrite
\begin{equation*}
\begin{split}
&\E [F_Y(\max\{Y_1, \dots, Y_{p+1}\})\mid \hat{F}_X(X_1)>1-\frac{k_n}{n}]\\&
=\E [F_Y(\max\{Y_1, \dots, Y_{p+1}\})\mid X_1>X_{(n-k_n)}]. 
\end{split}
\end{equation*}
Therefore, all we \textit{need} to show is the following 
\begin{equation}
\label{equa}
\begin{split}
\Gamma^{time}_{\mathbf{X}\to \mathbf{Y}}(p)&= \lim_{u\to \infty}\E [F_Y(\max\{Y_1, \dots, Y_{p+1}\})\mid X_1>u]\\&
\overset{}{=}\lim_{n\to\infty}\E [F_Y(\max\{Y_1, \dots, Y_{p+1}\})\mid X_1>X_{(n-k_n)}]. 
\end{split}
\end{equation}
Denote $Z=F_Y(\max\{Y_1, \dots, Y_{p+1}\})$. Denote $u_n\in\mathbb{R}$ as  $1-\frac{k_n}{n}$ quantiles of $X_1$, that is, numbers fulfilling $\pr(X_1>u_n)=\frac{k_n}{n}$. Because $u_n\to\infty$, \eqref{equa} is equivalent to
\begin{equation}
\label{pes}
\lim_{n\to\infty}\E[Z\mid X_1>u_n]\overset{}{=} \lim_{n\to\infty}\E[Z\mid X_1>X_{(n-k_n)}].
\end{equation}
Hence, if we prove \eqref{pes}, our proof will be complete. Rewrite (using identity $\1(a>b) = \1(c>a>b) + \1(a>c>b) + \1(a>b>c)$ when no ties are present):
\begin{equation*}
\begin{split}
&\E[Z\mid X_1>u_n]=\frac{1}{\pr(X_1>u_n)}\int_{\Omega}Z\cdot \1(X_1>u_n)\,\diff\pr
=\frac{n}{k_n}\int_{\Omega}Z\cdot \1(X_1>u_n)\,\diff\pr \\&= \frac{n}{k_n}\int_{\Omega}Z\cdot \1(X_{(n-k_n)}>X_1>u_n)\,\diff\pr + \frac{n}{k_n}\int_{\Omega}Z\cdot \1(X_1>X_{(n-k_n)}>u_n)\,\diff\pr \\&\quad + \frac{n}{k_n}\int_{\Omega}Z\cdot \1(X_1>u_n>X_{(n-k_n)})\,\diff\pr.
\end{split}
\end{equation*}

On the other hand, rewrite also
\begin{equation*}
\begin{split}
&\E[Z\mid X_1>X_{(n-k_n)}]=\frac{1}{\pr(X_1>X_{(n-k_n)})}\int_{\Omega}Z\cdot \1(X_1>X_{(n-k_n)})\,\diff\pr\\&
=\frac{n}{k_n}\int_{\Omega}Z\cdot \1(X_1>X_{(n-k_n)})\,\diff\pr = \frac{n}{k_n}\int_{\Omega}Z\cdot \1(u_n>X_1>X_{(n-k_n)})\,\diff\pr\\& + \frac{n}{k_n}\int_{\Omega}Z\cdot \1(X_1>X_{(n-k_n)}>u_n)\,\diff\pr + \frac{n}{k_n}\int_{\Omega}Z\cdot \1(X_1>u_n>X_{(n-k_n)})\,\diff\pr.
\end{split}
\end{equation*}
Note that these two equations differ only in the first term. Therefore, to show \eqref{pes}, we only need to show that 
$$
\frac{n}{k_n}\int_{\Omega}Z\cdot \1(X_{(n-k_n)}>X_1>u_n)\,\diff\pr - \frac{n}{k_n}\int_{\Omega}Z\cdot \1(u_n>X_1>X_{(n-k_n)})\,\diff\pr 
\overset{n\to\infty}{\to} 0.
$$
We show that the first term goes to $0$. Analogously, the second term can be shown to converge to $0$.

We know that $0\leq Z \leq 1$ and we have for the first term:
\begin{equation*}
\begin{split}
&\frac{n}{k_n}\int_{\Omega}Z\cdot \1(X_{(n-k_n)}>X_1>u_n)\,\diff\pr\leq \frac{n}{k_n} \pr(X_{(n-k_n)}>X_1>u_n)\\&
=\pr(X_{(n-k_n)}>X_1\mid X_1>u_n)=\pr(X_{(n-k_n)}>X_1\mid  F_X(X_1)>1-\frac{k_n}{n})\\&=1-\pr(X_1\geq X_{(n-k_n)}\mid F_X(X_1)>1-\frac{k_n}{n})\\&
=1-\pr(\hat{F}_X(X_1)\geq 1-\frac{k_n}{n}\mid F_X(X_1)>1-\frac{k_n}{n})\\&
=1-\pr(F_X(X_1) + (\hat{F}_X(X_1) - F_X(X_1))\geq1-\frac{k_n}{n}\mid F_X(X_1)>1-\frac{k_n}{n})\\&
\leq 1- \pr(F_X(X_1) - \sup_{x\in\mathbb{R}}|\hat{F}_X(x) - F_X(x)|\geq1-\frac{k_n}{n}\mid F_X(X_1)>1-\frac{k_n}{n}).
\end{split}
\end{equation*}
Denote $S_n:= \sup_{x\in\mathbb{R}}|\hat{F}_X(x) - F_X(x)|$. It is sufficient for our proof to show that  $$\pr(F_X(X_1) - S_n\geq 1-\frac{k_n}{n}\mid F_X(X_1)>1-\frac{k_n}{n})\overset{n\to\infty}{\to} 1.$$ 
Choose $\varepsilon>1$, define $\delta=1-\frac{1}{\varepsilon}$. Rewrite
\begin{equation*}
\begin{split}
&\pr(F_X(X_1) - S_n\geq 1-\frac{k_n}{n}\mid F_X(X_1)>1-\frac{k_n}{n}) \\&
=\frac{n}{k_n}\pr(F_X(X_1) - S_n\geq 1-\frac{k_n}{n}; F_X(X_1)>1-\frac{k_n}{n}) \\&
\geq \frac{n}{k_n}\pr(F_X(X_1) - S_n\geq 1-\frac{k_n}{n} ; F_X(X_1)>1-\frac{k_n/\varepsilon}{n})\\&
\geq \frac{n}{k_n}\pr( S_n\leq \frac{k_n - k_n/\varepsilon}{n} ; F_X(X_1)>1-\frac{k_n/\varepsilon}{n})\\&
= \frac{n}{k_n}\pr( \frac{n}{k_n}S_n\leq \delta ; F_X(X_1)>1-\frac{k_n/\varepsilon}{n}).
\end{split}
\end{equation*}
Use the identity $\pr(A\cap B) = 1-\pr(A^c) - \pr(B^c) + \pr(A^c\cap B^c)$  and continue 
\begin{equation*}
\begin{split}
&\frac{n}{k_n}\pr( \frac{n}{k_n}S_n\leq \delta ; F_X(X_1)>1-\frac{k_n/\varepsilon}{n})\\&
=\frac{n}{k_n} [1-\pr( \frac{n}{k_n}S_n>\delta) - \pr(F_X(X_1)\leq 1-\frac{k_n/\varepsilon}{n})\\&\quad + \pr( \frac{n}{k_n}S_n>\delta ; F_X(X_1)\leq 1-\frac{k_n/\varepsilon}{n})) ]\\&
\geq\frac{n}{k_n} [1-\pr( \frac{n}{k_n}S_n>\delta) - (1-\frac{k_n/\varepsilon}{n}) +0]\\&
=\frac{n}{k_n}[\frac{k_n/\varepsilon}{n} - \pr( \frac{n}{k_n}S_n>\delta) ]=\frac{1}{\varepsilon} -\frac{n}{k_n}\pr( \frac{n}{k_n}S_n>\delta) \overset{n\to\infty}{\to}\frac{1}{\varepsilon}  \overset{\varepsilon\to 1}{\to}1.
\end{split}
\end{equation*}
Altogether, we proved that $
\lim_{n\to\infty}\E[Z\mid X_1>u_n]=\lim_{n\to\infty}\E[Z\mid X_1>X_{(n-k_n)}],
$
 from which the theorem follows. 
\end{proof}
\begin{customconsequence}{\ref{Consequence5}}
Let $\mathbf{X}\to\mathbf{Y}$. Under the assumptions of Theorem~\ref{Theorem o asymptotic} and Theorem~\ref{Theorem 2.1}, the proposed estimator is consistent; that is,  $\hat\Gamma^{time}_{\mathbf{X}\to \mathbf{Y}}(p)\overset{P}{\to}\Gamma^{time}_{\mathbf{X}\to \mathbf{Y}}(p)$ as $n\to\infty$. 
\end{customconsequence}
\begin{proof}
Since  $\Gamma^{time}_{\mathbf{X}\to \mathbf{Y}}(p)=1$ from Theorem~\ref{Theorem 2.1} and trivially  $\hat\Gamma^{time}_{\mathbf{X}\to \mathbf{Y}}(p)\overset{a.s.}{\leq}1$, Theorem~\ref{Theorem o asymptotic} implies that $\E \hat\Gamma^{time}_{\mathbf{X}\to \mathbf{Y}}(p)\overset{n\to\infty}{\to}1$ and $\var\big(\hat\Gamma^{time}_{\mathbf{X}\to \mathbf{Y}}(p)\big)\overset{n\to\infty}{\longrightarrow} 0$. This implies consistency. 
\end{proof}

\begin{customlem}{\ref{lemma o concentration inequality}}
Let $\mathbf{X} = (X_t, t\in\mathbb{Z})$ has the form 
 $$
 X_t = \sum_{k=0}^\infty a_k \varepsilon_{t-k},
 $$
where $(\varepsilon_t, t \in \mathbb{Z})$ are iid random variables with density $f_\varepsilon$. Assume $|a_k|\leq \gamma k^{-\beta}$ for some $\beta>1$, $\gamma>0$ and all $k \in \mathbb{N}$. Let $f^\star = \max(1, |f_\varepsilon|_{\infty}, |f'_{\varepsilon}|_{\infty}) < \infty$, where $|f|_{\infty} = \sup_{x\in\mathbb{R}}|f(x)|$ is the supremum norm. Assume $\E|\varepsilon_0|^q<\infty$ for $q>2$ and $\pr(|\varepsilon_0|>x) \leq L(\log x)^{-r_0}x^{-q}$ for some constants $L>0, r_0>1$ and for every $x>1$. Then, if a sequence $(k_n)$ satisfies \eqref{k_deleno_n} and 
\begin{equation} \tag{\ref{aaa}}
\exists c>\max\left\{ \frac{1}{2},\frac{2}{1+q\beta}\right\}: \frac{k_n}{n^c}\to\infty,  \text{ as }n\to\infty,
\end{equation}
then the condition \eqref{zxc} is satisfied. 
\end{customlem}

\begin{proof}\hypertarget{proof o concentration inequality}{}
The result is a slight modification of Proposition 13 in \cite{Concentration_inequality}. The proposition states that, under our conditions, there exist $\alpha>1/2$ and $n_0\in\mathbb{N}$ such that for all $n>n_0$ and for all $z\geq \gamma \sqrt{n}(\log n)^\alpha$ holds

\begin{equation}
\label{bbb}
\pr( n\sup_{x\in\mathbb{R}}|\hat{F}_X(x)-F_X(x)|/f^\star > z) \leq constant\frac{n}{z^{q\beta}(\log z)^{r_0}}. 
\end{equation}

Choose $\delta>0$ and take $z=\frac{f^\star}{\delta} k_n$. Note that $z\geq \gamma \sqrt{n}(\log n)^\alpha$ for large $n$ due to \eqref{aaa}. 
Rewrite \eqref{bbb} into
$$
\frac{n}{k_n}\pr(\frac{n}{k_n} \sup_{x\in\mathbb{R}}|\hat{F}_X(x)-F_X(x)|>\delta) \leq constant\frac{n}{k_n^{q\beta}(\log \frac{f^\star}{\delta}k_n)^{r_0}}\frac{n}{k_n}. 
$$
Since $\frac{n}{k_n^{q\beta}}\frac{n}{k_n}\to 0$ as $n\to\infty$ due to \eqref{aaa}, we obtain that the condition \eqref{zxc} is satisfied.
\end{proof}

\begin{claim}\label{Claim}
 Under the setup in Section \ref{section hidden confounder}, where $\theta_Z\geq\theta_X, \theta_Y$, the following implications hold: 
 \begin{itemize}
     \item $\mathbf{Y}$ does not cause $\mathbf{X}$ and $ \theta_X\geq\theta_Y \implies \Gamma^{time}_{\mathbf{Y}\to \mathbf{X}}(p) < 1 \text{ for all } p\in\mathbb{N}$.
     \item $\mathbf{X}$  causes  $\mathbf{Y}$ and  $ \theta_X, \theta_Y>0\implies \Gamma^{time}_{\mathbf{X}\to \mathbf{Y}}(q)=1$.
 \end{itemize}
 
On the other hand, if $\theta_X<\theta_Y$, then $\Gamma^{time}_{\mathbf{Y}\to \mathbf{X}}(q) =1$ can happen even if  $\mathbf{Y}$ does not cause $\mathbf{X}$.  Hence, the results from Subsection \ref{Section 2.2} are still valid as long as $\theta_X\geq\theta_Y$, while they might no longer be true if $\theta_X<\theta_Y$. 
\end{claim}
We do not provide rigorous proof of this claim. However, we give here some simple intuition why we believe it is true.  $\Gamma^{time}_{\mathbf{Y}\to \mathbf{X}}(q)$ will be even smaller than in the case when $\theta_X = \theta_Y$, since the effect of $\mathbf{X}$ on $\mathbf{Y}$ in extremes will be much smaller if $\mathbf{X}$ has lighter tails than $\mathbf{Y}$. To
make this more rigorous, it is possible to  rewrite Proposition \ref{TentoTheorem} with unequal tail indexes and claim that

$$
\lim_{u\to\infty}\pr(\sum_{i=0}^\infty a_i \varepsilon_i^X<\lambda\mid \sum_{i=0}^\infty b_i\varepsilon_i^X + \sum_{i=0}^\infty c_i\varepsilon_i^Y>u)\leq \pr(\sum_{i=0}^\infty a_i \varepsilon_i^X<\lambda)\frac{C+\sum_{i\in\Phi}b_i^\theta}{C+B}
$$
for all $\lambda\in\mathbb{R}$, where the notation follows Proposition \ref{TentoTheorem}. The proof of Theorem \ref{Theorem 2.2.} would follow the same steps with modified Proposition  \ref{TentoTheorem}. 

As for the second bullet-point, the proof of Theorem \ref{Theorem 2.1} does not use the regular variation condition. Consequently, if $\mathbf{X}\to \mathbf{Y}$, we deduce that $\Gamma^{time}_{\mathbf{X}\to \mathbf{Y}}(q)=1$, irrespective of the tail indexes.

Model~\ref{Timeseries3definicia} with $\delta_Y=0, \delta_X = 1$ and $\theta_X< \theta_Y$ satisfies $\Gamma^{time}_{\mathbf{Y}\to \mathbf{X}}(3) =1$ while  $\mathbf{Y}$ does not cause $\mathbf{X}$. This follows simply from the identity $\lim_{u\to\infty}P(\varepsilon^X_t>u\mid \varepsilon^X_t+ \varepsilon_t^Y>u) = 1$ (follows from Lemma B.6.1 in \cite{SRE}). Using the causal representation of $\mathbf{X}$ and $\mathbf{Y}$ in Model~\ref{Timeseries3definicia}, we simply obtain $\Gamma^{time}_{\mathbf{Y}\to \mathbf{X}}(3) =1$.  

\renewcommand{\bibname}{Bibliography}


\bibliography{bibliography}


\end{document}